\documentclass[a4paper, 10pt]{article}

\RequirePackage[OT1]{fontenc}
\RequirePackage{amsthm,amsmath}
\RequirePackage[numbers]{natbib}
\RequirePackage[colorlinks,citecolor=blue,urlcolor=blue]{hyperref}

\usepackage{url}
\usepackage{graphicx}
\usepackage{amssymb}
\usepackage{amscd}
\usepackage{comment}
\usepackage{stmaryrd}
\usepackage{xcolor}
\usepackage{enumerate}
\usepackage{array}
\usepackage{thmtools}
\usepackage[linesnumbered,ruled,vlined]{algorithm2e}
\SetAlCapHSkip{0em}

\usepackage{authblk}

\def\R{\mathbb{R}}

\def\E{\mathbb E}
\def\EXP{{\E}}

\def\I{{\mathbb I}}

\def\Rank{\mathop{\rm Rank}}
\def\argmax{\mathop{\rm arg\, max}}

\def\sgn{\mathop{\rm sgn}}

\newcommand{\X}{{\cal X}}
\newcommand{\Y}{{\cal Y}}

\def\N{{\cal N}}

\def\S{{\cal S}}
\def\L{{\cal L}}

\def\V{{\cal V}}
\def\Z{{\cal Z}}
\def\eps{{\varepsilon}}
\def\Var{{\rm Var}}
\newcommand{\RR}{\mathbb{R}}
\newcommand{\NN}{\mathbb{N}}

\def\roc{\rm ROC}
\def\auc{\rm AUC}
\def\dcg{\mathop{\rm DCG}}

\newcommand{\bX}{\mathbf{X}}

\newcommand{\bY}{\mathbf{Y}}
\newcommand{\bZ}{\mathbf{Z}}
\newcommand{\bF}{F}

\newcommand{\iid}{\textit{i.i.d.}}
\newcommand{\cdf}{\textit{c.d.f.}\;}
\newcommand{\df}{\textit{d.f.}}
\newcommand{\rv}{\textit{r.v.}}
\newcommand{\ie}{\textit{i.e.}\;}
\newcommand{\CF}{\textit{cf}\;}
\newcommand{\wrt}{\textit{w.r.t.}\;}
\newcommand{\st}{\textit{s.t.}\;}
\newcommand{\eg}{\textit{e.g.}}

\newcommand{\resp}{\textit{resp. }}
\newcommand{\IFF}{\textit{iff}}

\newcommand{\vc}{{\sc VC}}

{\tiny }

\newtheorem{hyp}{\textit{Assumption}} 

\newtheorem{theorem}{Theorem.}
\newtheorem{corollary}[theorem]{Corollary.}
\newtheorem{lemma}[theorem]{Lemma.}
\newtheorem{proposition}[theorem]{Proposition.}
\newtheorem{definition}[theorem]{Definition.}
\newtheorem{remark}{Remark}

\numberwithin{equation}{section}
\theoremstyle{plain}

\title{Concentration Inequalities for Two-Sample Rank Processes with Application to Bipartite Ranking}
\date{}

\author[1]{Stephan Cl\'emen\c{c}on}
\author[2]{, Myrto Limnios  \thanks{Corresponding author. Authors in alphabetical order.}}
\author[2]{, Nicolas Vayatis}

\affil[1]{ \small \url{stephan.clemencon@telecom-paris.fr}\\
	Telecom Paris, LTCI, Institut Polytechnique de Paris\\
	19 place Marguerite Perey, Palaiseau, 91120, France.}
\affil[2]{\url{myrto.limnios@ens-paris-saclay.fr}, \url{ nicolas.vayatis@ens-paris-saclay.fr}\\
	Université Paris-Saclay, ENS Paris-Saclay\\CNRS, Centre Borelli, F-91190 Gif-sur-Yvette, France.}

\begin{document}
\maketitle

		\begin{abstract}
The $\roc$ curve is the gold standard for measuring the performance of a test/scoring statistic regarding its capacity to discriminate between two statistical populations in a wide variety of applications, ranging from anomaly detection in signal processing to information retrieval, through medical diagnosis.
Most practical performance measures used in scoring/ranking applications such as the $\auc$, the local $\auc$, the  $p$-norm push, the $\dcg$ and others, can be viewed as summaries of the $\roc$ curve. In this paper, the fact that most of these empirical criteria can be expressed as \textit{two-sample linear rank statistics} is highlighted and concentration inequalities for collections of such random variables, referred to as \textit{two-sample rank processes} here, are proved, when indexed by {\sc VC} classes of scoring functions. Based on these nonasymptotic bounds, the generalization capacity of empirical maximizers of a wide class of ranking performance criteria is next investigated from a theoretical perspective. It is also supported by empirical evidence through convincing numerical experiments.   
		\end{abstract}
	
\section{Introduction}

In the context of ranking, a variety of performance measures can be considered.
In the simplest framework of bipartite ranking, where two independent $\iid$ samples $\bX_1,\;\ldots,\; \bX_n$  and  $\bY_1,\;\ldots,\; \bY_m$  defined on the same probability space $(\Omega,\; \mathcal{F},\; \mathbb{P})$, valued in the same space $\Z$, say $\mathbb{R}^d$ with $d\geq 1$ for instance, and drawn from probability distributions $G$ and $H$ respectively (referred to as the 'positive distribution' and the 'negative distribution' respectively), the goal pursued is to learn a preorder on $\Z$ defined through a scoring function $s:\Z \rightarrow \RR$ (which transports the natural order on the real line onto the feature space $\Z$) such that, for any random observation $\bZ \in \Z$ sampled from a distribution that is equal either to the 'positive distribution' or to the 'negative one', the larger the score $s(z)$, the likelier it is drawn from the 'positive distribution' $G$. Though easy to formulate, this simple framework encompasses many practical problems from the design of search engines in Information Retrieval (in this case, for a specific request, $G$ is the distribution of the relevant digitized documents, while $H$ is that of the irrelevant ones) to the elaboration of decision support tools in personalized medicine for instance. In spite of its simplicity  there is not one and only one natural scalar criterion for evaluating the performance of a scoring rule $s(z)$, but many possible options. 
The \textit{Receiving Operator Characteric} curve (the $\roc$ curve in abbreviated form), \textit{i.e.} the PP-plot of the false positive rate \textit{vs} the true positive rate: $$t\in\mathbb{R}\mapsto \left(\mathbb{P}\{s(\bY)>t  \},\; \mathbb{P}\{s(\bX)>t  \}\right),$$ denoting by $\bX$ and $\bY$ two generic $\rv$ with distributions $G$ and $H$ respectively, provides an exhaustive description of the performance of any scoring rule candidate $s$. However, its functional
nature renders direct optimization strategies rather complex, see \textit{e.g.} \cite{CV10CA}. \textit{Empirical risk minimization} methods (ERM) are thus generally based on summaries of the $\roc$ curve, which take the form of empirical risk functionals where the averages involved are no longer taken over $\iid$ sequences.
The most popular choice is undoubtedly the $\auc$ criterion ($\auc$ standing for \textit{Area Under the $\roc$ Curve}), see \cite{AGHHPR05} or \cite{CLV08} for instance,
but when focus is on top-ranked instances, various choices can be considered, \textit{e.g.} the Discounted Cumulative Gain or DCG (see \cite{CosZha06b}),
the $p$-norm push (see \cite{Rud06}), the local $\auc$ (refer to \cite{CV07}) or other variants such as those recently introduced in \cite{MW16}. The present paper starts from the simple observation that most of these summary criteria have a common feature: they belong to the class of \textit{two-sample linear rank statistics}. Such statistics have been extensively
studied in the mathematical statistics literature because of their optimality properties in hypothesis testing,
see \cite{HajSid67}. They are widely used in order to test whether two samples are drawn from the same distribution against the alternative stipulating that the distribution of one of the samples is stochastically larger than the other. For instance, the empirical counterpart of the $\auc$ of a scoring function $s(z)$ corresponds to the popular Mann-Withney-Wilcoxon statistic based on the two (univariate) samples $s(\bX_1),\;\ldots,\; s(\bX_n)$  and  $s(\bY_1),\;\ldots,\; s(\bY_m)$. Other rank statistics can be considered, corresponding to other ways of measuring how the distribution of the 'positive score' $s(\bX)$ is (possibly) stochastically larger than that of the 'negative score' $s(\bY)$. Now, in the statistical learning view, with the
importance of excess risk bounds, the \textit{Empirical Risk Minimization} paradigm must be revisited and new problems, mainly related to the uniform control of the fluctuations of collections of two-sample linear rank statistics, termed rank processes throughout the article, and to the measure of the complexity of nonparametric classes of scoring functions, come up.
The arguments required to deal with risk functionals based on two-sample linear rank statistics have been sketched
in \cite{CV07} in a very special case.
\par In the present paper, we relate two-sample linear rank statistics to performance measures relevant for the ranking problem by showing that the target of ranking algorithms corresponds to optimal ordering rules in this sense and show that the generic structure of two-sample linear rank statistics as an orthogonal decomposition after projection onto the space of sums of $\iid$ random variables is the key to all statistical results related to maximizers of such criteria: consistency, rates of convergence or model selection. Notice incidentally that the empirical $\auc$ is also a $U$-statistic and a decomposition method akin to that considered in this paper (though much less general) has been used in order to handle this specific dependence structure in \cite{CLV08}. In this article, concentration properties of two-sample rank processes (\textit{i.e.} collections of two-sample linear rank statistics) are investigated using the linearization technique aforementioned. While interesting in themselves, the concentration inequalities established for this class of stochastic processes, when indexed by Vapnik-Chervonenkis classes (abbreviated with {\sc VC}-classes) of scoring functions, are next applied to study the generalization capacity of empirical maximizers of a large collection of performance criteria based on two-sample linear rank statistics.
Notice finally that a preliminary version of this work is briefly outlined in the conference paper \cite{CV08NIPS1}. This article presents a much deeper analysis of bipartite ranking via maximization of two-sample linear rank statistics. In particular, it offers a complete and detailed study of the concentration properties of two-sample rank processes (in a slightly different framework, stipulating that two independent $\iid$ samples, positive and negative, are observed, rather than classification data), provides model selection results and, from a practical perspective, tackles the issue of smoothing the risk functionals under study here with statistical learning guarantees.

\par The paper is organized as follows. In Section \ref{sec:motivation}, the main notations are set out, the bipartite ranking problem is formulated as a statistical learning task in a rigorous probabilistic framework and the concept of two-sample linear rank statistic is briefly recalled. 
It is also explained that, unsurprisingly, natural performance criteria in bipartite ranking are of the form of two-sample (linear) rank statistics. Concentration results for 
rank processes, are established in Section \ref{sec:Rproc}. By means of the latter, performance of bipartite ranking rules obtained by maximizing two-sample linear rank statistics are investigated in Section \ref{sec:application}. Finally, Section \ref{sec:num} displays illustrative experimental results, supporting the theoretical analysis carried out in the present article. Proofs, technical details and additional numerical results are deferred to the Appendix section.

\section{Motivation and Preliminaries}\label{sec:motivation}
We start with recalling key notions pertaining to $\roc$ analysis and bipartite ranking, which essentially motivates the theoretical analysis carried out in the subsequent section. We next recall at length the definition of two-sample linear rank statistics, which have been intensively used to design statistical (homogeneity) testing procedures in the univariate setup, and finally highlight that many scalar summaries of empirical $\roc$ curves, commonly used as ranking performance criteria, are precisely of this form. Here and throughout, the indicator function of any event $\mathcal{E}$ is denoted by $\mathbb{I}\{ \mathcal{E} \}$, the Dirac mass at any point $x$ by $\delta_x$, the generalized inverse of any cumulative distribution function $W(t)$ on $\R\cup\{+\infty\}$ by $W^{-1}(u)=\inf\{t\in ]-\infty,\;+\infty]:\; W(t)\geq u\}$, $u\in [0,1]$. We also denote the floor and ceiling functions by $u \in \mathbb{R}\mapsto \lfloor u \rfloor$ and by $u \in \mathbb{R}\mapsto \lceil u \rceil$ respectively.
\subsection{Bipartite Ranking and $\roc$ Analysis} \label{subsec:biproc}
As recalled in the Introduction section, the goal of bipartite ranking is to learn, based on independent 'positive' and 'negative' random samples $\{\bX_1,\;\ldots,\; \bX_n\}$ and  $\{\bY_1,\;\ldots,\; \bY_m\}$, how to score any new observations $\bZ_1,\; \ldots,\; \bZ_k$, being each either 'positive' or else 'negative', that is to say drawn either from $G$ or else from $H$, without prior knowledge, so that positive instances are mostly at the top of the resulting list with large probability. A natural way of defining a total preorder\footnote{A preorder $\preccurlyeq$ on a set $\Z$ is a reflexive and transitive binary relation on $\Z$. It is said to be \textit{total}, when either $z\preccurlyeq z'$ or else $z'\preccurlyeq z$ holds true, for all $(z,z')\in\Z^2$.} on $\Z$ is to map it with the natural order on $\mathbb{R}\cup \{+\infty \}$ by means of a \textit{scoring rule}, \textit{i.e.} a measurable mapping $s:\Z \rightarrow ]-\infty,\; \infty]$. By $\mathcal{S}$ is denoted the set of all scoring rules. 
It is by means of $\roc$ analysis that the capacity of a scoring rule candidate $s(z)$ to discriminate between the positive and negative statistical populations is generally evaluated.
 \medskip
 
\noindent {\bf $\roc$ curves.} The $\roc$ curve is a gold standard to describe the dissimilarity between two univariate probability distributions $G$ and $H$. This criterion of functional nature,  $\roc_{H,G}$, can be defined as the parametrized curve in $[0,1]^2$:
$$ t\in \mathbb{R}\mapsto \left( 1-H(t),\; 1-G(t)  \right),$$
where possible jumps 
are connected by line segments, so as to ensure that the resulting curve is continuous. With this convention, one may then see the $\roc$ curve related to the pair of $\df$ $(H,G)$ as the graph of a c\`ad-l\`ag (\textit{i.e.} right-continuous and left-limited) non decreasing mapping valued in $[0,1]$, defined by:
$$
\alpha \in (0,1) \mapsto  1-G\circ H^{-1}(1-\alpha)~,
$$
at points $\alpha$ such that $G\circ H^{-1}(1-\alpha)=1-\alpha$. Denoting by $\Z_H$ and $\Z_G$ the supports of $H$ and $G$ respectively, observe that it connects the point $(0,1-G( \Z_H))$ to $(H(\Z_G),1)$ in the unit square $[0,1]^2$ and that, in absence of plateau 
(which we assume here for simplicity, rather than restricting the feature space to $G$'s support), the curve $\alpha\in (0,1)\mapsto \roc_{G,H}(\alpha)$ is the image of $\alpha\in (0,1)\mapsto \roc_{H,G}(\alpha)$ by the reflection with the main diagonal of the Euclidean plane (\textit{i.e.} the line of equation '$\beta=\alpha$') as axis. Notice that the curve $\roc_{H,G}$ coincides with the main diagonal of $[0,1]^2$ if and only if the two distributions $H$ and $G$ are equal. Hence, the concept of $\roc$ curve offers a visual tool to examine the differences between two distributions in a pivotal manner, see Fig. \ref{fig:ex1}. For instance, the univariate distribution $G(dt)$ is stochastically larger\footnote{Given two distribution functions $H(dt)$ and $G(dt)$ on $\R\cup\{+\infty\}$, it is said that $G(dt)$ is \textit{stochastically larger} than $H(dt)$ $\IFF$ for any $t\in\mathbb{R}$, we have $G(t)\leq H(t)$. We then write: $H\leq_{sto}G$. Classically, a necessary and sufficient condition for $G$ to be stochastically larger than $H$ is the existence of a coupling $(\bX,\; \bY)$ of $(G,H)$, \textit{i.e.} a pair of random variables defined on the same probability space with first and second marginals equal to $H$ and $G$ respectively, such that $\bX\leq \bY$ with probability one.} than $H(dt)$ if and only if the curve $\roc_{H,G}$ is everywhere above the main diagonal and $\roc_{H,G}$ coincides with the left upper corner of the unit square $\IFF$ the essential supremum of the distribution $H$ is smaller than the essential infimum of the distribution $G$. 
\begin{figure}[!h]
\centering
\begin{tabular}{cc}
\parbox{6cm}{
\begin{center}
\includegraphics[width=5cm]{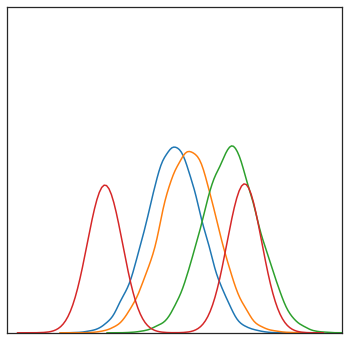}\\
\tiny{a. Probability distributions}
\end{center}
}
\parbox{6cm}{
\begin{center}
\includegraphics[width=5.1cm]{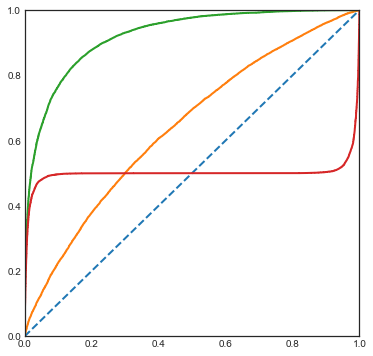}\\
\tiny{b. $\roc$ curves}
\end{center}
}\\
\end{tabular}
\caption{Examples of pairs of distributions and their related $\roc$ curves. The 'negative' distribution $H$ is represented in blue and three examples of 'positive' distributions are represented in red, orange and green, like the associated $\roc$ curves.}
\label{fig:ex1}
\end{figure}
Another advantage of the $\roc$ curve lies in the probabilistic interpretation of the popular $\roc$ curve summary, referred to as the Area Under the $\roc$ Curve ($\auc$ in short)
\begin{equation}\label{eq:auc}
\auc_{H,G}\overset{def}{=}\int_{0}^1\roc_{H,G}(\alpha)d\alpha=\mathbb{P}\left\{Y<X \right\}+\frac{1}{2}\mathbb{P}\left\{X=Y \right\},
\end{equation}
where $(X,Y)$ denotes a pair of independent $\rv$'s with respective marginal distributions $H$ and $G$.
\medskip

\noindent {\bf Bipartite Ranking as $\roc$ curve optimization.} Going back to the multivariate setup, where $H$ and $G$ are probability distributions on $\Z$, say $\Z = \RR^d$ with arbitrary dimension $d\geq 1$, the goal pursued in bipartite ranking can be phrased as that of building a scoring rule $s(z)$ such that the (univariate) distribution $G_s$ of $s(\bX)$ is 'as stochastically larger as possible' than the the distribution $H_s$ of $s(\bY)$.  Hence, the capacity of a candidate $s(z)$ to discriminate between the positive and negative statistical populations can be evaluated by plotting the $\roc$ curve $\alpha\in (0,1)\mapsto \roc(s,\alpha)=\roc_{H_s,G_s}(\alpha)$: the closer to the left upper corner of the unit square the curve $\roc(s,.)$, the better the scoring rule $s$. Therefore, the $\roc$ curve  conveys a partial preorder on the set of all scoring functions: for all pairs of scoring functions $s_1$ and $s_2$, one says that $s_2$ is more accurate than $s_1$ when $\roc(s_1,\alpha)\leq \roc(s_2,\alpha)$ for all $\alpha\in [0,1]$. It follows from a standard Neyman-Pearson argument that the most accurate scoring rules are increasing transforms of the likelihood ratio $\Psi(z)=dG/dH(z)$. Precisely, it is shown in \cite{CV09ieee} (see Proposition 2 therein) that the optimal scoring rules are the elements of the set:
\begin{equation}\label{eq:scoresetopt}
\S^*=\left\{s\in \S \; \text{ \textit{s.t.} for all } z,\; z' \text{ in } \mathbb{R}^d:\;\; \Psi(z)<\Psi(z') \Rightarrow s^*(z)<s^*(z')    \right\}.
\end{equation}
We denote by $\roc^*(.)=\roc(\Psi,\; .)$ and recall incidentally that this optimal curve is non-decreasing and concave and thus always above the main diagonal of the unit square. Now, the bipartite ranking task can be reformulated in a more quantitative manner: the objective pursued is to build a scoring function $s(z)$, based on the training examples $\{\bX_1,\;\ldots,\; \bX_n\} \text{ and } \{\bY_1,\;\ldots,\; \bY_m\}$, with a $\roc$ curve as close as possible to $\roc^*$. A typical way of measuring the deviation between the two curves is to consider the distance in $\sup$ norm:
\begin{equation}\label{eq:sup_dist}
d_{\infty}(s, s^*)=\sup_{\alpha\in (0,1)}\left\vert \roc(s,\alpha)-\roc^*(\alpha) \right\vert~.
\end{equation}
Attention should be paid that this quantity is a distance between $\roc$ curves (or between the related equivalence classes of scoring functions, the $\roc$ curve of any scoring function being invariant by strictly increasing transform) not between the scoring functions themselves. Since the curve $\roc^*$ is unknown in practice, the major difficulty lies in the fact that no straightforward statistical counterpart of the (functional) loss \eqref{eq:sup_dist} is available. In \cite{CV09ieee}
 (see also \cite{CV10CA}), it has been however shown that bipartite ranking can be viewed as a superposition of cost-sensitive classification problems and somehow 'discretized' in an adaptive manner, so as to apply empirical risk minimization with statistical guarantees in the $d_{\infty}$-sense, at the price of an additional bias term inherent to the approximation step.
 Alternatively, the performance of a candidate scoring rule $s$ can be measured by means of the $L_1$-norm in the $\roc$ space. Observing that, in this case, the loss can be decomposed as follows:
 \begin{equation}\label{eq:L1_norm}
 d_1(s,s^*)=\int_{0}^1\vert \roc(s,\alpha)-\roc^*(\alpha)  \vert d\alpha=\int_{0}^1\roc^*(\alpha)  d\alpha -\int_{0}^1 \roc(s,\alpha) d\alpha~,
 \end{equation}
 minimizing the $L_1$-distance to the optimal $\roc$ curve boils down to maximizing the area under the curve $\roc(s,\; .)$, that is to say
\begin{equation}\label{eq:auc_s}
\auc(s)\overset{def}{=}\auc_{H_s,G_s}=\mathbb{P}\{s(\bY)<s(\bX) \}+\frac{1}{2}\mathbb{P}\{s(\bY)=s(\bX) \}~,
\end{equation}
where $\bX$ and $\bY$ are random variables defined on the same probability space, independent, with respective distributions $G$ and $H$, denoting by $G_s$ and $H_s$ the distributions of $s(\bX)$ and $s(\bY)$ respectively.
The scalar performance criterion $\auc(s)$ defines a total preorder on $\mathcal{S}$ and its maximal value is denoted by $\auc^*=\auc(s^*)$, with $s^*\in \S^*$. Bipartite ranking through maximization of empirical versions of the $\auc$ criterion has been studied in several articles, including \cite{AGHHPR05} or \cite{CLV08}. Extension to \textit{multipartite ranking} (\textit{i.e.} when the number of samples/distributions under study is larger than $3$) is considered in \cite{CRV13}, see also \cite{CR15}. In contrast to \cite{CV09ieee} or \cite{CV10CA}, where the task of learning scoring rules with statistical guarantees in $\sup$ norm in the $\roc$ space is considered, the present article focuses on optimization of summary scalar empirical criteria generalizing the $\auc$ that takes the form of two-sample linear \textit{rank statistics}, as could be naturally expected when addressing ranking problems.
\subsection{Two-Sample Linear Rank Statistics}\label{subsec:2sample_rank_stats}
If the curve $\roc_{H,G}$ is the appropriate tool to examine to which extent a univariate distribution $G$ is stochastically larger than another one $H$, practical decisions are generally made on the basis of the observations of two univariate  independent random $\iid$ samples $\{X_1,\;\ldots,\; X_n\}$ and  $\{Y_1,\;\ldots,\; Y_m\}$, drawn from $G$ and $H$ respectively. Computing the empirical cumulative distribution functions 
$\widehat{H}_m(t)=(1/m)\sum_{j=1}^m\mathbb{I}\{Y_j\leq t\}$ and $\widehat{G}_n(t)=(1/n)\sum_{i=1}^n\mathbb{I}\{X_i\leq t\}$ for $t\in \mathbb{R}$, one can plot the empirical $\roc$ curve:
\begin{equation}\label{eq:emp_ROC}
\widehat{\roc}=\roc_{\widehat{H}_m, \; \widehat{G}_n}.
\end{equation}
 Observe that the $\roc$ curve \eqref{eq:emp_ROC} is an increasing broken line connecting $(0,0)$ to $(1,1)$ in the unit square $[0,1]^2$ and is fully determined by the set of ranks occupied by the positive instances within the pooled sample $\{ \Rank(X_i):\; i=1,\; \ldots,\; n\}$, where:

\begin{equation}\label{eq:rankdef}
\forall i\in\{1,\; \ldots,\; n  \}\quad \Rank(X_i)=N\widehat{F}_{N}(X_i)~,
\end{equation}

with $\widehat{F}_{N}(t)=(1/N)\sum_{i=1}^n\mathbb{I}\{X_i \leq t  \}+ (1/N)\sum_{j=1}^m\mathbb{I}\{Y_j \leq t  \}$ and $N=n+m$. Breakpoints of the piecewise linear curve \eqref{eq:emp_ROC} necessarily belong to the set of gridpoints $\left\{ \left(j/m,\; i/n  \right):\; j\in\{1,\; \ldots,\; m-1  \} \text{ and } i\in\{1,\; \ldots,\; n-1  \}  \right\}$.
Denote by $X_{(i)}$ the order statistics related to the sample $\{X_1,\;\ldots,\; X_n\}$, \textit{i.e.}
$
\Rank(X_{(n)})>\cdots >\Rank(X_{(1)})
$,
and by $Y_{(j)}$ those related to the sample $\{Y_1,\;\ldots,\; Y_m\}$. Consider the c\`ad-l\`ag step function:
\begin{equation}\label{eq:step}
\alpha\in [0,1]\mapsto \sum_{j=1}^m \widehat{\gamma}_j \cdot \mathbb{I}\{\alpha\in [(j-1)/m,\; j/m[\}~,
\end{equation}
where, for all $j\in\{1,\; \ldots,\; m\}$, we set:
\begin{multline*}
\widehat{\gamma}_j=\frac{1}{n}\sum_{i=1}^n\mathbb{I}\{  X_i>Y_{(m-j+1)} \}=\frac{1}{n}\sum_{i=1}^n\mathbb{I}\{  \Rank(X_{(n-i+1)})>\Rank(Y_{(m-j+1)}) \}\\
= \frac{1}{n}\sum_{i=1}^n\mathbb{I}\{j\geq N-\Rank(X_{(n-i+1)})-i+2  \}~.
\end{multline*}

The $\roc$ curve \eqref{eq:emp_ROC} is the continuous broken line that connects the jump points of the step curve \eqref{eq:step} and can thus be expressed as a function of the 'positive ranks' \ie the $\Rank(X_i)$'s only. As a consequence, any summary of the empirical $\roc$ curve, is a two-sample rank statistic, that is a measurable function of the 'positive ranks'. In particular, the empirical $\auc$, \ie the $\auc$ of the empirical $\roc$ curve \eqref{eq:emp_ROC}, also termed the rate of concording pairs or the \textit{Mann-Whitney statistic}, can be easily shown to coincide, up to an affine transform, with the sum of 'positive ranks', the well-known \textit{rank-sum Wilcoxon statistic} \cite{Wil45}:
\begin{equation}\label{eq:rank_sum}
\widehat{W}_{n,m}=\sum_{i=1}^n \Rank(X_i)~.
\end{equation}
Indeed, we have:
$$
\widehat{W}_{n,m}=nm\auc_{\widehat{H}_m, \widehat{G}_n}+\frac{n(n+1)}{2}~.
$$
However, two-sample rank statistics (\ie functions of the $\Rank(X_i)$'s) form a very rich collection of statistics and this is by no means the sole possible choice to summarize the empirical $\roc$ curve.

\begin{definition}{\sc (Two-sample linear rank statistics)}\label{def:R_stat} Let $\phi:[0,1]\rightarrow  \RR$ be a nondecreasing function. The two-sample linear rank statistics with 'score-generating function' $\phi(u)$ based on the random samples $\{X_1,\;\ldots,\; X_n\}$ and  $\{Y_1,\;\ldots,\; Y_m\}$ is given by:
\begin{equation}\label{eq:R_stat}
\widehat{W}^{\phi}_{n,m}=\sum_{i=1}^n\phi\left( \frac{\Rank(X_i)}{N+1} \right).
\end{equation}
\end{definition}
The statistics \eqref{eq:R_stat} defined above are all distribution-free when $H=G$ and are, for this reason, particularly useful to detect differences between the distributions $H$ and $G$ and widely used to perform homogeneity tests in the univariate setup. Tabulating their distribution under the null assumption, they can be used to design unbiased tests at certain levels $\alpha$ in $(0,1)$. The choice of the score-generating function $\phi$ can be guided by the type of difference between the two distributions (\textit{e.g.} in scale, in location) one possibly expects, and may then lead to locally most powerful testing procedures, capable of detecting 'small' deviations from the homogeneous situation. More generally, depending on the statistical test to perform, one may use particular function $\phi$, Figure \ref{fig:scoregenstatclass} shows classic score-generating functions broadly used for two-sample statistical tests (refer to \cite{Haj62}). 
One may refer to Chapter 9 in \cite{Ser80} or to Chapter 13 in \cite{vdV98} for an account of the (asymptotic) theory of rank statistics. 
In the present paper, two-sample linear rank statistics are used for a very different purpose, as empirical performance measures in bipartite ranking based on two independent multivariate samples $\{\bX_1,\;\ldots,\; \bX_n\}$ and  $\{\bY_1,\;\ldots,\; \bY_m\}$. The analysis of the bipartite ranking problem carried out in Section \ref{sec:application}, based on the concentration inequalities established in Section \ref{sec:Rproc}, shows the relevance of evaluating the ranking performance of a scoring rule candidate $s(z)$ by computing a two-sample linear rank statistic based on the univariate samples obtained after scoring $\{s(\bX_1),\;\ldots,\; s(\bX_n)\}$ and  $\{s(\bY_1),\;\ldots,\; s(\bY_m)\}$ and establishes statistical guarantees for the generalization capacity of scoring rules built by optimizing such an empirical criterion.

\begin{figure}[!h]
	\centering
	\includegraphics[width=6cm, height=6cm]{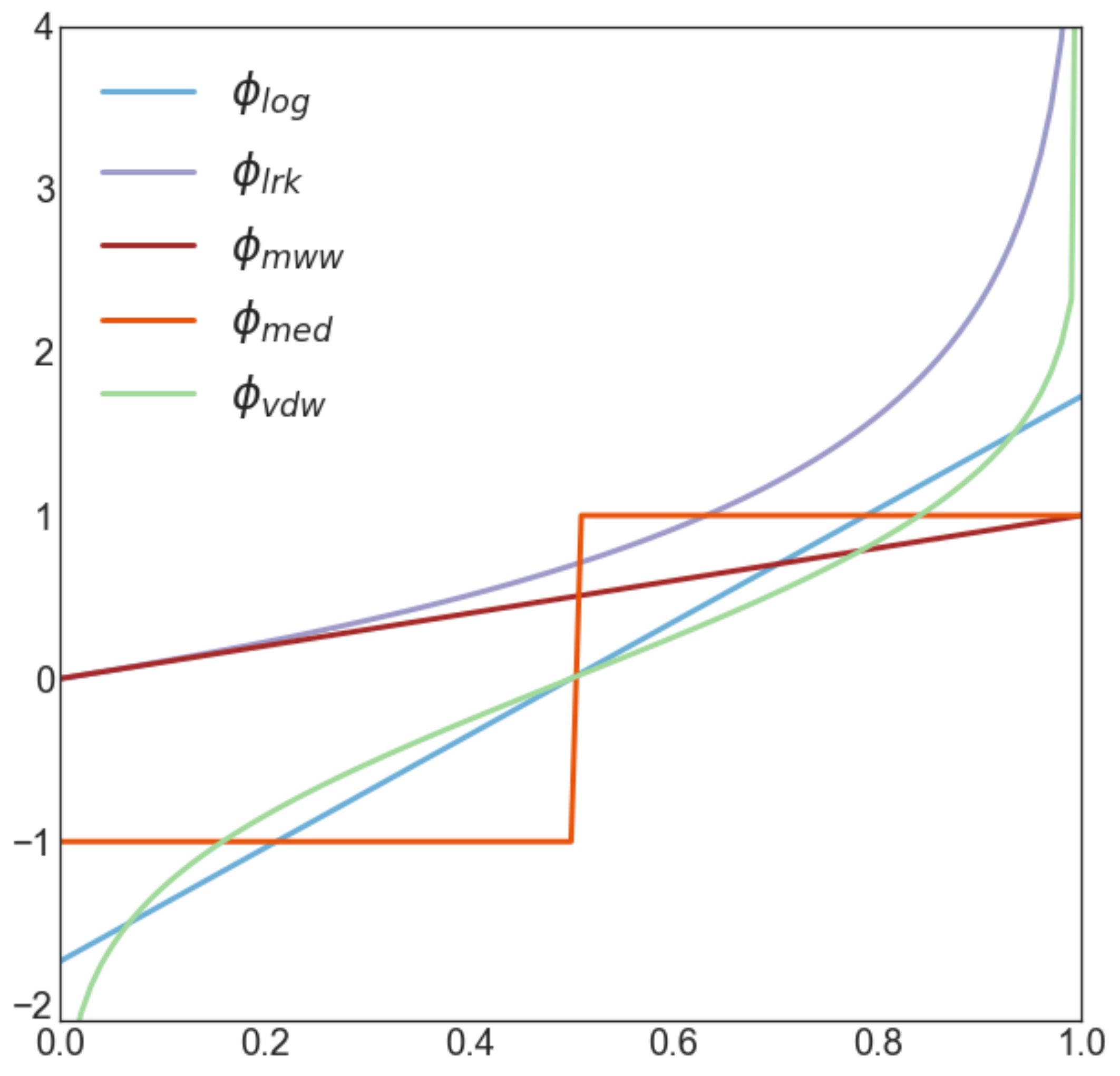}
	\caption{Curves of two-sample score-generating functions with the associated statistical test: Logistic test $\phi_{log}(u)= 2 \sqrt{3}(u - 1/2)$ in blue,  Logrank test $\phi_{lrk}(u)= - \log(1-x)$ in purple, Mann-Whitney-Wilcoxon test $\phi_{mww}(u)= u$ in red,  Median test  $\phi_{med}(u)= \sgn (u - 1/2)$  in orange,  Van der Waerden test $\phi_{vdw}(u)= \Phi^{-1}(u)$  in green, $\Phi$ being the normal quantile function.}
	\label{fig:scoregenstatclass}
\end{figure}

\subsection{Bipartite Ranking as Maximization of Two-Sample Rank Statistics} \label{sec:criteria}
As foreshadowed above, empirical performance measures in bipartite ranking should be unsurprisingly based on ranks. 
We propose here to evaluate empirically the ranking performance of any scoring function candidate $s(z)$ in $\mathcal{S}$ by means of statistics of the type: 
\begin{equation}\label{Wstat}
\widehat{W}^{\phi}_{n,m}(s) = \sum_{i=1}^n\phi \left(\frac{\Rank(s(\bX_i))}{N+1}\right),
\end{equation}
where $N=n+m$, $\phi : [0,1] \to \RR$ is some Borelian nondecreasing function. This quantity is a two-sample linear rank statistic (see Definition \ref{def:R_stat}) related to the score-generating function $\phi(u)$ and the samples $\{s(\bX_1),\;\ldots,\; s(\bX_n)\}$ and $\{s(\bY_1),\;\ldots,\; s(\bY_m)\}$. This statistic is invariant by increasing transform of the scoring function $s$, just like the (empirical) $\roc$ curve and, as recalled in the previous section, it is a natural and common choice to quantify differences in distribution between the univariate samples $\{s(\bX_1),\;\ldots,\; s(\bX_n)\}$ and $\{s(\bY_1),\;\ldots,\; s(\bY_m)\}$, to evaluate to which extent the distribution of the first sample is stochastically larger than that of the second sample in particular. It consequently appears as legitimate to learn a scoring function $s$ by maximizing the criterion \eqref{Wstat}. Whereas rigorous arguments are developed in Section \ref{sec:application}, we highlight here that,
for specific choices of the score-generating function $\phi$, many relevant criteria considered in the ranking literature can be accurately approximated by statistics of this form:
\begin{itemize}
	\item $\phi(u)=u$ - this choice leads to the celebrated Wilcoxon-Mann-Whitney statistic which is related to the empirical version of the $\auc$.
	\item $\phi(u)=u \mathbb{I}\{u\ge u_0\}$, for some $u_0\in (0, 1)$ - such a score-generating function corresponds to the local $\auc$ criterion, introduced recently in \cite{CV07}. Such a criterion is of interest when one wants to focus on the highest ranks.
	\item $\phi(u)=u^q$ - this is another choice which puts emphasis on high ranks but in a smoother way than the previous one. This is related to the $q$-norm push approach taken in \cite{Rud06}. However, we point out that the criterion studied in the latter work relies on a different definition of the rank of an observation.  Namely, the
	rank of positive instances among negative instances (and not in the pooled sample) is used. This choice permits to use independence which makes the technical part much simpler, at the price of increasing the variance of the criterion. 
	\item
	$\phi(u)=\phi_N(u)=c\left((N+1)u\right)  \mathbb{I}\{u\geq k/(N+1)\}$ - this corresponds to the $\dcg$ criterion in the bipartite setup (see \cite{CosZha06b}), one of the 'gold standard quality measure' in information retrieval, when \textit{grades} are binary. The $c(i)$'s denote the \textit{discount factors}, $c(i)$ measuring the importance of rank $i$. The integer $k$ denotes the number of top-ranked instances to take into account. Notice that, with our indexation, top positions correspond to the largest ranks and the sequence $\{c(i)\}_{i\leq N}$ should be chosen increasing. 
\end{itemize}

Depending on the choice of the score-generating function $\phi$, some specific patterns of the preorder induced by a scoring function $s(z)$ can be either enhanced by the criterion \eqref{Wstat} or else completely disappear: for instance, the value of \eqref{Wstat} is essentially determined by the possible presence of positive instances among top-ranked observations, when considering a score generating function $\phi$ that rapidly vanishes near $0$ and takes much higher values near $1$.
\medskip

 Investigating the performance of maximizers of the criterion \eqref{Wstat} from a nonasymptotic perspective is however far from straightforward, due to the complexity of the latter (\textit{i.e.} a sum of strongly dependent random variables).
It requires in particular to prove concentration inequalities for collections of two-sample linear rank statistics, indexed by classes of scoring functions of controlled complexity (\textit{i.e.} of {\sc VC}-type), referred to as two-sample rank processes throughout the article. It is the purpose of the next section to establish such results.

\section{Concentration Inequalities for Two-Sample Rank Processes}\label{sec:Rproc}

This section is devoted to prove concentration bounds for collections of two-sample linear rank statistics \eqref{Wstat}, indexed by classes $\mathcal{S}_0\subset \mathcal{S}$ of scoring functions.
In order to study the fluctuations of \eqref{Wstat} as the full sample size $N$ increases, it is of course required to control the fraction of 'positive'/'negative' observations in the pooled dataset. Let $p\in (0,1)$ be the 'theoretical' fraction of positive instances. For $N\geq 1/p$, we suppose that $n = \lfloor pN \rfloor$ and $m = \lceil (1-p)N \rceil = N - n$. Define the mixture probability distribution $F=pG+(1-p)H$. For any $s\in\mathcal{S}$, the distribution of $s(\bX)$ (\textit{i.e.} the image of $G$ by $s$) is denoted by $G_s$, that of $s(\bY)$ (\textit{i.e.} the image of $H$ by $s$) by $H_s$. We also denote by $F_s$ the image of distribution $F$ by $s$. For simplicity, the same notations are used to mean the related cumulative distribution functions. We also introduce their statistical versions $\widehat{G}_{s,n}(t)=(1/n)\sum_{i=1}^n\mathbb{I}\{s(\bX_i)\leq t  \}$ and $\widehat{H}_{s,m}(t)=(1/m)\sum_{j=1}^m\mathbb{I}\{s(\bY_j)\leq t  \}$ and define: 
\begin{equation}\label{eq:emp_cdf_raw}
\widehat{F}_{s,N}(t)=(n/N)\widehat{G}_{s,n}(t)+(m/N)\widehat{H}_{s,m}(t)~.
\end{equation}
Since $n/N\rightarrow p$ as $N$ tends to infinity, the quantity above is a natural estimator of the \cdf $F_s$.
Equipped with these notations, we can write:
\begin{equation}\label{eq:crit_emp}
\frac{1}{n}\widehat{W}^{\phi}_{n,m}(s)=\frac{1}{n}\sum_{i=1}^n\phi \left(\frac{N}{N+1} \widehat{F}_{s,N}(s(\bX_i))  \right).
\end{equation}
Hence, the statistic \eqref{eq:crit_emp} can be naturally seen as an empirical version of the quantity defined below, around which it fluctuates.
\begin{definition}
	For a given score-generating function $\phi$, the functional
	\begin{equation}\label{eq:Wtrue}
	W_{\phi}(s) = \mathbb{E} [(\phi \circ F_s)(s(\bX))]~,
	\end{equation}
	is referred to as the "$W_{\phi}$-ranking performance measure".
\end{definition}

Indeed, replacing $\widehat{F}_{s,N}(s(\bX_i))$ in \eqref{eq:crit_emp} by $F_s(s(\bX_i))$ and taking next the expectation permits to recover $\eqref{eq:Wtrue}$. Observe in addition that, for $\phi(u)=u$, the quantity \eqref{eq:Wtrue} is equal to $\auc(s)$ \eqref{eq:auc_s} a soon as the distribution $F_s$ is continuous. The next lemma reveals that the criterion \eqref{eq:Wtrue} can be viewed as a scalar summary of the $\roc$ curve.
\begin{lemma}
Let $\phi$ be a score-generating function. We have, for all $s$ in $\S$, 
		\begin{equation}\label{eq:W_roc}
		W_{\phi}(s) =\frac{1}{p}\int_{0}^1\phi(u)du -\frac{1-p}{p}\int_{0}^1 \phi\left(p(1-\roc(s, \alpha))+(1-p)(1-\alpha)\right)~d\alpha~.
		\end{equation}
\end{lemma}
\begin{proof}
Using the decomposition $F_s = pG_{s}+(1-p)H_{s}$, we are led to the following expression:
\[
p W_{\phi}(s)  = \int_{0}^1\phi(u)~du-(1-p)\mathbb{E}[(\phi \circ F_s)(s(\bY))]~.
\]
Then, using a change of variable, we get:
\[
\mathbb{E}[(\phi \circ F_s)(s(\bY))]
= \int_{0}^1 \phi(p(1-\roc(s, \alpha))+(1-p)(1-\alpha))~d\alpha~.
\]

\end{proof}
As revealed by Eq. \eqref{eq:W_roc}, a score-generating function $\phi$ that takes much higher values near $1$ than near $0$ defines a criterion \eqref{eq:Wtrue} that mainly summarizes the behavior of the $\roc$ curve near the origin, \textit{i.e.} the preorder on the set of instances with highest scores. 
\medskip

Below, we investigate the concentration properties of the process:
\begin{equation}\label{eq:W_process}
\left\{ \frac{1}{n}\widehat{W}^{\phi}_{n,m}(s) - W_{\phi}(s)   \right\}_{s\in \mathcal{S}_0}.
\end{equation}
As a first go, we prove, by means of linearization techniques,
that two-sample linear rank statistics can be uniformly approximated by
much simpler quantities, involving $\iid$ averages and two-sample $U$-statistics. This will be key to establish probability bounds for the maximal deviation:

\begin{equation}\label{eq:supremum}
\sup_{s\in \mathcal{S}_0}\left\vert  \frac{1}{n}\widehat{W}^{\phi}_{n,m}(s)-W_{\phi}(s)  \right\vert,
\end{equation}

  under adequate complexity assumptions for the class $\mathcal{S}_0$ of scoring functions considered and to study next the generalization ability of maximizers of the empirical criterion \eqref{eq:crit_emp} in terms of $W_{\phi}$-ranking performance. Throughout the article, all the suprema considered, such as \eqref{eq:supremum}, are assumed to be measurable and we refer to Chapter 2.3 in \cite{vdVWell96} for more details on the formulation in terms of outer measure/expectation that guarantees measurability.
\medskip

{\bf Uniform approximation of two-sample linear rank statistics.} Whereas statistical guarantees for Empirical Risk Minimization in the context of classification or regression can be directly obtained by means of classic concentration results for empirical processes (\textit{i.e.} averages of $\iid$ random variables), the study of the fluctuations of the process \eqref{eq:W_process} is far from straightforward, insofar as the terms averaged in \eqref{eq:crit_emp} are not independent.
For averages of non-$\iid$ random variables, the underlying statistical structure can be revealed by orthogonal projections onto the space of sums of $\iid$ random variables in many situations. This projection argument was the key for the study of empirical $\auc$ maximization or that of within cluster point scatter, which involved $U$-processes, see \cite{CLV08} and \cite{CLEM14}. In the case of $U$-statistics, this orthogonal decomposition is known as the \textit{Hoeffding decomposition} and the remainder may be expressed as a degenerate $U$-statistic, see \cite{Hoeffding48}. For rank statistics, a similar though more complex decomposition can be considered. We refer to \cite{Haj68} for a systematic use of the \textit{projection method} for investigating the asymptotic properties of general statistics.
From the perspective of ERM in statistical learning theory, through the \textit{projection method}, well-known concentration results for standard empirical processes and $U$-processes may carry over to more complex collections of random variables such as \textit{two-sample linear rank processes}, as revealed by the approximation result stated below. It holds true under the following technical assumptions.


\begin{hyp}\label{hyp:sabscont}
Let $M>0$.	For all $s\in \S_0$, the random variables $s(\bX)$ and $s(\bY)$ are continuous, with density functions that are twice differentiable and have Sobolev $\mathcal{W}^{2,\infty}$-norms\footnote{Recall that the Sobolev space $\mathcal{W}^{2,\infty}$ is the space of all Borelian functions $h:\mathbb{R}\rightarrow \mathbb{R}$ such that $h$ and its first and second order weak derivatives $h'$ and $h''$ are bounded almost-everywhere. Denoting by $\vert\vert.\vert\vert_{\infty}$ the norm of the Lebesgue space $L_{\infty}$ of Borelian and essentially bounded functions, $\mathcal{W}^{2,\infty}$ is a Banach space when equipped with the norm $\vert\vert h\vert\vert_{2,\infty}=\max\{\vert\vert h\vert\vert_{\infty},\; \vert\vert h'\vert\vert_{\infty},\; \vert\vert h''\vert\vert_{\infty}   \}$.} bounded by $M<+\infty$.
\end{hyp}


\begin{hyp}\label{hyp:phic2}
	The score-generating function $\phi : [0,1] \mapsto \RR$, is nondecreasing and twice continuously differentiable.
\end{hyp}

\begin{hyp}\label{hyp:VC}
	The class of scoring functions $\S_0$ is a {\sc VC} class of finite {\sc VC} dimension $\V<+\infty$.
\end{hyp}
 
For the definition of  {\sc VC} classes of functions, one may refer to \textit{e.g.} \cite{vdVWell96}, see section 2.6.2 therein, and also recalled in Appendix section \ref{appsubsec:vcperm}.
By means of the proposition below, the study of the fluctuations of the two-sample linear rank process \eqref{eq:W_process} boils down to that of basic empirical processes.
 
\begin{proposition}\label{prop:hajek} Suppose that Assumptions \ref{hyp:sabscont}-\ref{hyp:VC} are fulfilled. The two-sample linear rank process \eqref{eq:W_process} can be linearized/decomposed as follows. For all $ s\in \S_0$,

\begin{equation}\label{eq:linear}
\widehat{W}_{n,m}^{\phi}(s) = n \widehat{W}_{\phi}(s) + \left(\widehat{V}_{n}^X(s) -\mathbb{E}\left[ \widehat{V}_{n}^X(s) \right]\right)+\left( \widehat{V}_{m}^Y(s) -\mathbb{E}\left[  \widehat{V}_{m}^Y(s) \right]\right)+ \mathcal{R}_{n,m}(s)~,
\end{equation}
\noindent where
\begin{eqnarray*}
\widehat{W}_{\phi}(s)&=& \frac{1}{n} \sum_{i=1}^{n} \left(\phi \circ \bF_s\right)(s(\bX_i))~, \\
	\widehat{V}_{n}^X(s) &=& \frac{n-1}{N+1} \sum_{i=1}^n \int_{s(\bX_i)}^{+\infty} (\phi' \circ \bF_s)(u) dG_{s}(u)~,\\
\widehat{V}_{m}^Y(s) &=& \frac{n}{N+1} \sum_{j=1}^m \int_{s(\bY_j)}^{+\infty}( \phi' \circ \bF_s)(u) dG_{s}(u)~.
\end{eqnarray*}
	



For any $\delta \in (0,1)$, there exist constants $c_1,  \; c_3 >0, \; c_2 \geq 1, \; c_4>6, \;c_5>3 $, depending on $\phi$ and $\V$, such that 

\begin{equation}\label{eq:remdevboundprophajek}
	\mathbb{P}\left\{	\sup_{s \in \S_0} \vert \mathcal{R}_{n,m}(s)  \vert < t \right\} \geq 1 - \delta~,
\end{equation}

\noindent where $t = c_1 + c_2\log(c_4/\delta) $ as soon as $N\geq (c_3/p)\log(c_5/\delta)$.
 
\end{proposition}

The proof of this linearization result is detailed in the Appendix section \ref{pf:hajek} (refer to it for a description of the constants involved in the bound stated above). Its main argument consists in decomposing \eqref{eq:crit_emp} by means of a Taylor expansion at order two of the score generating function $\phi(u)$ and applying next the H\'ajek orthogonal projection technique (recalled at length in the 
Introduction Lemma \ref{appsubsec:hajmeth} for completeness) to the component corresponding to the first order term. The quantity  $\mathcal{R}_{n,m}(s)$ is then formed by bringing together the remainder of the H\'ajek projection and the component corresponding to the second order term of the Taylor expansion, while the probabilistic control of its order of magnitude is established by means of concentration results for (degenerate) one/two-sample $U$-processes (see the Appendix section \ref{subsec:ineq_Yproc} for more details).
It follows from decomposition \eqref{eq:linear} combined with triangular inequality that:
\begin{multline}\label{eq:max_dev}
 \sup_{s\in \S_0} \left\vert \frac{1}{n} \widehat{W}_{n,m}^{\phi}(s)-W_{\phi}(s) \right\vert \leq  \sup_{s\in \S_0} \left\vert \widehat{W}_{\phi}(s)-W_{\phi}(s) \right\vert \\
 +\sup_{s\in \S_0}\frac 1 n \left\vert \widehat{V}_{n}^X(s) -\mathbb{E}\left[ \widehat{V}_{n}^X(s) \right] \right\vert 
 + \sup_{s\in \S_0} \frac 1 n\left\vert \widehat{V}_{m}^Y(s) -\mathbb{E}\left[ \widehat{V}_{m}^Y(s) \right] \right\vert \\
 +\sup_{s\in \S_0}\frac 1 n \left\vert \mathcal{R}_{n,m}(s) \right\vert.
\end{multline}
Hence, nonasymptotic bounds for the maximal deviation of the process \eqref{eq:W_process} can be deduced from concentration inequalities for standard empirical processes, as shall be seen below. 
Before this, a few comments are in order.

\begin{remark} \label{rk:comp} {\sc (On the complexity assumption)} We point out that alternative complexity measures could be naturally considered, such as those based on Rademacher averages, see \textit{e.g.} \cite{Kolt06}. 
However, as different types of stochastic process (\textit{i.e.} empirical process, degenerate one-sample $U$-process and degenerate two-sample $U$-process) are involved in the present nonasymptotic study, different types of Rademacher complexities (see \textit{e.g.} \cite{CLV08}) should be introduced to control their fluctuations as well. For the sake of simplicity, the concept of {\sc VC}-type class of functions is used here.
\end{remark}

\begin{remark}\label{rk:nonsmooth} {\sc (Smooth score-generating functions)} The subsequent analysis is restricted to the case of smooth score-generating functions for simplification purposes. We nevertheless point out that, although one may always build smooth approximants of irregular score generating functions, the theoretical results established below can be directly extended to non-smooth situations, at the price of a significantly greater technical complexity.
\end{remark}

The theorem below provides a concentration bound for the two-sample rank process \eqref{eq:W_process}. The proof is based on the uniform approximation result precedingly established, refer to the Appendix section \ref{pf:bound_twosample} for technical details.

\begin{theorem}\label{thm:bound_twosample}
Suppose that the assumptions of Proposition \ref{prop:hajek} are fulfilled. Then, there exist constants $C_1, \; C_2 \geq 24$, depending on $\phi, \; \V$,  such that for all $C_4\geq C_1$ depending on $\phi$, and $t$:  



\begin{equation}
\mathbb{P}\left\{ \sup_{s \in \S_0} \bigg| \frac{1}{n}  \widehat{W}^{\phi}_{n,m}(s)- W_{\phi}(s)  \bigg| > t \right\}
\leq C_2e^{-pC_3Nt^2} ~,
\end{equation}

\noindent as soon as  $C_1 /  \sqrt{pN}  \leq t \leq  C_4 \min(p,1-p)$, where $C_3 = \log(1+C_{4}/(4C_{2}))/(C_2C_4)$ depends on $\phi, \; \V$. 


\end{theorem}

The concentration inequalities stated above are extensively used in the next section to study the ranking bipartite learning problem, when formulated as $W_{\phi}$-ranking performance maximization.

\section{Performance of Maximizers of Two-Sample Rank Statistics in Bipartite Ranking}\label{sec:application}
This section provides a theoretical analysis of bipartite ranking methods, based on maximization of the empirical ranking performance measure \eqref{Wstat}. While the concentration inequalities established in the previous section are the key technical tools to derive nonasymptotic bounds for the deficit of $W_{\phi}$-ranking performance measure of empirical maximizers, we start by showing that the criterion \eqref{eq:Wtrue} is relevant to measure ranking performance, whatever the score generating function $\phi$ is chosen, beyond the examples listed in Subsection \ref{sec:criteria}. 
\medskip

\noindent {\bf Optimal elements.} The next result states that optimal scoring functions do maximize the $W_{\phi}$-ranking performance and form a collection that coincides with the set $\mathcal{S}^*_{\phi}$ of maximizers of \eqref{eq:Wtrue}, provided that the score-generating function $\phi$ is strictly increasing on $(0,1)$.

\begin{proposition} \label{prop:opt} Let $\phi$ be a score-generating function. The assertions below hold true.

\begin{itemize}
\item[(i)] For all $(s,\; s^*)\in\S\times \S^*$, we have $W_{\phi}(s) \le  W_{\phi}(s^*)=W_{\phi}^*$,
where $W_{\phi}^* \overset{def}{=} W_{\phi}(\Psi)$.
\item[(ii)] Assuming in addition that the score-generating function $\phi$ is strictly increasing on $(0,1)$, we have: $\mathcal{S}^*_{\phi}=\S^*$.
\end{itemize}
\end{proposition}
The proof immediately results from \eqref{eq:W_roc} combined with the fact that the $\roc$ curve of increasing transforms of the likelihood ratio $\Psi(z)$ dominates everywhere any other $\roc$ curve, as recalled in Section \ref{subsec:biproc}:  $\forall (s,\; s^*)\in\S\times \S^*$, $\forall \alpha\in (0,\; 1)$, $\roc(s,\alpha)\leq \roc(s^*,\alpha)=\roc^*(\alpha)$. Details are left to the reader.

\begin{remark} \label{rk:smooth}{\sc (On plug-in ranking rules)} Theoretically, a possible approach to bipartite ranking  is the {\em plug-in} method (\cite{DGL96}), which consists of using an estimate $\hat{\Psi}$ of the
likelihood function as a scoring function. As shown by the subsequent bound, when $\phi$ is differentiable with a bounded derivative,
when $\hat{\Psi}$ is close to $\Psi$ in the $L_1$-sense, it leads to a nearly optimal ordering in terms of W-ranking criterion:
$$
W_{\phi}^*-W_{\phi}\left(\widehat{\Psi}\right)\leq (1-p)\vert\vert \phi'\vert\vert_{\infty}\EXP[\vert \widehat{\Psi}(\bX)-\Psi(\bX)\vert]~.
$$
However, the bound above may be loose and the plug-in approach faces computational difficulties when dealing with high-dimensional data, see \cite{GKKW02}, which provide the motivation for
exploring algorithms based on $W_{\phi}$-ranking performance maximization.
\end{remark}

\begin{remark}{\sc (Alternative probabilistic framework)} We point out that the present analysis can be extended to the alternative setup, where, rather than assuming that two samples of sizes $n$ and $m$, 'positive' and 'negative', are available for the learning tasks considered in this paper, the $\iid$ observations $Z$ are supposed to come with a random label $Y$ either equal to $+1$ or else to $-1$, indicating whether $Z$ is distributed according to $G$ or $H$. If $p$ denotes the probability that the label $Y$ is equal to $1$, the number $n$ of positive observations among a training sample of size $N$ is then random, distributed as a binomial of size $N$ with parameter $p$. 
\end{remark}

Consider any maximizer of the empirical $W_{\phi}$-ranking performance measure over a class $\mathcal{S}_0\subset \mathcal{S}$ of scoring rules:
\begin{equation}
\widehat{s}\in \argmax_{s\in \mathcal{S}_0}\widehat{W}^{\phi}_{n,m}(s)~.
\end{equation}
Since we obviously have:
\begin{equation}\label{eq:perfineq}
W^*_{\phi}-W_{\phi}(\widehat{s})\leq 2\sup_{s\in \mathcal{S}_0}\left\vert \frac{1}{n} \widehat{W}^{\phi}_{n,m}(s)- W_{\phi}(s)\right\vert + \left(  W^*_{\phi}-\sup_{s\in \mathcal{S}_0}W_{\phi}(s) \right),
\end{equation}
the control of deficit of $W$-ranking performance of empirical maximizers of \eqref{eq:crit_emp} can be deduced from the concentration properties of the process \eqref{eq:W_process}.

\subsection{Generalization Error Bounds and Model Selection} \label{sec:bound}

The corollary below describes the generalization capacity of scoring rules based on empirical maximization of $W_{\phi}$-ranking performance criteria. It straightforwardly results from Theorem \ref{thm:bound_twosample} combined with the bound \eqref{eq:perfineq}.

\begin{corollary}\label{cor:bound_twosample}Let $\hat{s}$ be an empirical $W_{\phi}$-ranking performance maximizer over the class $\S_0$, \textit{i.e.} $\hat{s} \in \argmax_{s\in\S_0} \widehat{W}_{n,m}^{\phi}(s)$.
	Under the assumptions of Proposition \ref{prop:hajek}, for any $\delta \in (0,1)$, we have with probability at least $1-\delta$:

	\begin{equation}
		W^*_{\phi}-W_{\phi}(\hat{s}) \leq 2 \sqrt{\frac{\log(C_2/\delta)}{p C_3N}} + \left(  W^*_{\phi}-\sup_{s\in \mathcal{S}_0}W_{\phi}(s) \right),
	\end{equation}
as soon as   $N\geq 1/(p \min(p,1-p)^2C_3C_4^2)  \log(C_2/\delta)$ and $\delta \leq C_2e^{-C_1^2C_3}$ 
 where the constants $C_i, \; i \leq 4$, being the same as those involved in Theorem \ref{thm:bound_twosample}.

\end{corollary}

The result above establishes that maximizers of the empirical criterion \eqref{Wstat} achieve a classic learning rate bound of order $O_{\mathbb{P}}(1/\sqrt{N})$ when based on a training data set of size $N$, just like in standard classification, see \textit{e.g.} \cite{DGL96}.
Refer to the Appendix section \ref{app:expectation} for the proof of an additional result, that provides a bound in expectation for the deficit of $W_{\phi}$-ranking performance measure, similar to that established in the subsequent analysis, devoted to the model selection issue.

\medskip

\noindent{\bf Model selection by complexity penalization.}
We have investigated the issue of approximately recovering the best scoring rule in a given class $\mathcal{S}_0$ in the sense of the $W_{\phi}$-ranking performance measure \eqref{eq:Wtrue}, which is satisfactory only when the minimum achieved over $\mathcal{S}_0$ is close to $W_{\phi}^*$ of course. We now address the problem of model selection, that is the problem of selecting a good scoring function from one of a collection of {\sc VC} classes $\S_k, \; k\geq1$. A model selection method is a data-based procedure that aims at achieving a trade-off regarding two contradictory objectives, \textit{i.e.} at finding a class $\S_k$ rich enough to include a reasonable approximant of an element of $\S^*$, while being not too complex so that the performance of the empirical minimizer over it $\hat{s}_k= \argmax_{s \in \S_k}\; \widehat{W}^{\phi}_{n,m}(s)$ can be statistically guaranteed. We suppose that all class candidates $\S_k$, $k\geq 1$, fulfill the assumptions of Proposition \ref{prop:hajek} and denote by $\mathcal{V}_k$ the {\sc VC} dimension of the class $\S_k$. Various model selection techniques, based on (re-)sampling or data-splitting procedures, could be naturally considered for this purpose. Here, in order to avoid overfitting, we focus on a complexity regularization approach, of which study can be directly derived from the rate bound analysis previously carried out, that consists in substracting to the empirical ranking performance measure the penalty term (increasing with $\mathcal{V}_k$) given by:
\begin{equation}
\text{pen}(N,k)= B_1\sqrt{\frac{\V_k}{pN}}+\sqrt{\frac{2C\log k}{p^2N}}~,
\end{equation}
 	for $pN \geq B_2 \V_k $ where the constants $B_1$ and  $B_2$ are those involved in Proposition \ref{prop:Ebound_twosample} and $C = 6( \Vert \phi \Vert_{\infty}^2 + 9\Vert \phi' \Vert_{\infty}^2+ 9\vert\vert \phi''\vert\vert_{\infty}^2)$.
The scoring function selected maximizes the penalized empirical ranking performance measure, it is $\hat{s}_{\hat{k}}(z)$ where:
\begin{equation}
\hat{k}=\argmax_{k\geq 1}\left\{ \frac{1}{n} \widehat{W}^{\phi}_{n,m}(s)   - \text{pen}(N,k)  \right\}~.
\end{equation}

The result below shows that the scoring rule $\hat{s}_{\hat{k}}$ nearly achieves the expected deficit of $W_{\phi}$-ranking performance that would have been attained with the help of an oracle, revealing the model minimizing $W^*_{\phi}-\mathbb{E}[ W_{\phi}(\hat{s}_k)]$.

\begin{proposition}\label{cor:oracle} Suppose that the assumptions of Proposition \ref{prop:hajek} are fulfilled for any class $\S_k$ with $k\geq 1$ and that $\sup_{k\geq 1}\V_k<+\infty$. Then, we have:
	\begin{multline}
	W_{\phi}^*  - \mathbb{E}\left[W_{\phi}(\hat{s}_{\hat{k}})\right] \leq \\
	\underset{k\geq 1}{\min} \left\{2\text{pen}(N,k) + \left( W^*_{\phi}-\sup_{s\in \S_k}W_{\phi}(s) \right)   \right\}   + 2\sqrt{\frac{C}{p^2N}}~,
	\end{multline}
	as soon as $pN \geq B_2 \sup_{k\geq 1}\V_k $, where the constant $B_2>0$ is the same as that involved in Proposition \ref{prop:Ebound_twosample} and $C = 6(\Vert \phi \Vert_{\infty}^2 + 9\Vert \phi' \Vert_{\infty}^2+ 9\vert\vert \phi''\vert\vert_{\infty}^2)$.
\end{proposition}
Refer to the Appendix section \ref{pf:oracle} for the technical proof.

\subsection{Kernel Regularization for Ranking Performance Maximization}\label{sec:densityreg}
Many successful algorithmic approaches to statistical learning (\textit{e.g.} boosting, support vector machines, neural networks) consist in smoothing the empirical risk/performance functional to be optimized, so as to use computationally feasible techniques based on gradient descent/ascent methods. Concerning the empirical criterion \eqref{Wstat}, although one may choose a regular score generating function $\phi$ (\textit{cf} Remark \ref{rk:nonsmooth}), smoothness issues arise when replacing $F_s$ in \eqref{eq:Wtrue} by the raw empirical $\cdf$ \eqref{eq:emp_cdf_raw}. A classic remedy involves using a kernel-smoothed version of the empirical $\cdf$ instead. 
Let  $K:\mathbb{R}\to \mathbb{R}$ be a second-order Parzen-Rosenblatt kernel \textit{i.e.} a Borelian symmetric function, integrable $\wrt$ the Lebesgue measure such that $\int K(t)dt=1$ and $\int t^2K(t)dt<+\infty$. Precisely, for any $h>0$ and all $t\in\mathbb{R}$, define the smoothed approximation of the $\cdf$ $F_s(t)$:

\begin{equation}\label{eq:smooth_cdf}
\widetilde{F}_{s,h}(t)=\int_{ \mathbb{R}}\kappa\left(\frac{t-u}{h}\right)F_s(du)~,
\end{equation}

\noindent where $\kappa(t)=\int_{-\infty}^tK(u)du$ and $h>0$ is the bandwidth that determines the degree of smoothing, see \textit{e.g.} \cite{Nadaraya64}. The uniform integrated error $\sup_{s\in \S_0}\int\vert \widetilde{F}_{s,h}(t)-F_s(t)\vert dt $ is shown to be of order $O(h^2)$ under the assumptions recalled below, see \cite{Jones90}. 

\begin{hyp}\label{hyp:approx}
	Let $R>0$. For all $s$ in $\S_0$, the cumulative distribution function $F_s$ is differentiable with derivative $f_s$ such that $\int(f_s'(t))^2dt\leq R$.
\end{hyp}

\begin{hyp}\label{hyp:kernelreg} The kernel function $K$ is of the  form $K_1 \circ K_2$, where $K_1$ is a function of bounded variation and $K_2$ is a polynomial.
\end{hyp}

Notice that Assumption \ref{hyp:approx} is fulfilled as soon as Assumption \ref{hyp:sabscont} is satisfied with $R\geq M$. The statistical counterpart of \eqref{eq:smooth_cdf} is then:

\begin{equation}
\widehat{F}_{s,N,h}(t)= \frac{1}{N}\sum_{i=1}^n \kappa\left( \frac{t-s(\mathbf{X}_i)}{h} \right) + \frac{1}{N}\sum_{j=1}^m \kappa\left( \frac{t-s(\mathbf{Y}_j)}{h} \right).
\end{equation}

\noindent A smooth version of the theoretical criterion \eqref{eq:Wtrue} is given by:

\begin{equation}\label{eq:Wsmooth}
\widetilde{W}_{\phi,h}(s) = \mathbb{E} [(\phi \circ \widetilde{F}_{s,h})(s(\bX))]~,
\end{equation}

\noindent for all $s\in \mathcal{S}$ and an empirical version of the latter is $\widehat{W}^{\phi}_{n,m,h}(s)/n$, where:

\begin{equation}\label{eq:Wsmooth_emp}
\widehat{W}^{\phi}_{n,m,h}(s)=\sum_{i=1}^n(\phi \circ  \widehat{F}_{s,N,h})(s(\bX_i))~.
\end{equation}
For any maximizer $\tilde{s}$ of \eqref{eq:Wsmooth_emp} over the class $\S_0$ of scoring function candidates, we almost-surely have:
\begin{multline}\label{eq:decomp_smooth}
W_{\phi}^*-W_{\phi}\left( \tilde{s} \right)\leq 2\sup_{s\in \S_0}\left\vert \frac{1}{n}\widehat{W}^{\phi}_{n,m,h}(s) - 	\widetilde{W}_{\phi,h}(s)   \right\vert + \sup_{s\in \S_0}\left\vert \widetilde{W}_{\phi,h}(s) - W_{\phi}(s) \right\vert \\
+\left\{ W_{\phi}^*- \sup_{s\in \S_0}W_{\phi}(s) \right\}~.
\end{multline}
This decomposition is similar to that obtained  in \eqref{eq:perfineq} for maximizers of the criterion \eqref{Wstat}, apart from the additional bias term. Since the latter can be shown to be of order $O(h^2)$ under appropriate regularity conditions and the first term on the right hand side of the equation above can be controlled like in Theorem \ref{thm:bound_twosample}, one may bound the deficit of $W_{\phi}$-ranking performance measure of $\tilde{s}$ as follows.

\begin{proposition}\label{prop:erm_bound2}
	Suppose that the assumptions of Proposition \ref{prop:hajek}  are fulfilled, as well as Assumptions \ref{hyp:approx} and \ref{hyp:kernelreg}.  Let $\tilde{s}$ be any maximizer of the smoothed criterion \eqref{eq:Wsmooth_emp} over the class $\S_0$. Then, for any $\delta\in (0,1)$, there exist constants $C_1, \; C_3 >0, \; C_2 \geq 24$ depending on $\phi, \; K, \; R, \; \V$,  $C_4\geq C_1$, and  $C_5>0$ is a constant depending on $\phi$, $K$ and $R$, such that we have with probability at least $1-\delta$:\\
	
	\begin{equation}
W^*_{\phi}-W_{\phi}(\tilde{s})\leq 2\sqrt{\frac{\log(C_2/\delta)}{pC_3N}}  + C_5h^2 + \left\{  W^*_{\phi}-\sup_{s\in \mathcal{S}_0}W_{\phi}(s) \right\},
\end{equation}		
as soon as $N\geq 1/(p \min(p,1-p)^2C_3C_4^2)  \log(C_2/\delta)$ and $\delta \leq C_2e^{-C_1^2C_3}$, with $C_3 = \log(1+C_{4}/(4C_{2}))/(C_2C_4)$. 
\end{proposition}

\noindent The proof is detailed in the Appendix section \ref{pf:erm_bound2}.

\section{Numerical Experiments}\label{sec:num}

It is the purpose of this section to illustrate empirically various points highlighted by the theoretical analysis previously carried out: in particular, the capacity of ranking rules obtained by maximization of empirical $W_{\phi}$-performance measures to generalize well and the impact of the choice of the score generating function $\phi$ on ranking performance from the perspective of $\roc$ analysis. Some practical issues, concerning the maximization of smoothed versions of the empirical $W_{\phi}$-performance criterion, are also discussed through numerical experiments. Additional experimental results can be found in the Appendix section \ref{ann:expes}. All experiments displayed in this article can be reproduced using the code available at \url{https://github.com/MyrtoLimnios/grad\_2sample}.

\subsection{A Gradient-Based Algorithmic Approach}

We start by describing the gradient ascent method (GA) used in the experiments in order to maximize the smoothed criterion \eqref{eq:Wsmooth_emp} obtained by kernel smoothing over the class of scoring functions $\S_0$ considered, as proposed in section \ref{sec:densityreg}, see Algorithm \ref{algo:GA}. 
Precisely, suppose that $\S_0$ is a parametric class, indexed by a parameter space $\Theta \subset \RR^d$ with $d\geq 1$ say: $\S_0=\{s_{\theta}:\X\rightarrow \RR,\;\; \theta\in \Theta\}$. Assume also that, for all $z \in \Z$, the mapping $\theta \in \Theta\mapsto s_{\theta}(z)$ is of class $ \mathcal{C}^1$ (\textit{i.e.} continuously differentiable) with gradient $\partial_{\theta}s_{\theta}(z)$ and that the score-generating function $\phi$ fulfills Assumption \ref{hyp:phic2}.  The gradient of the smoothed ranking performance measure of $s_{\theta}$ \wrt to the parameter $\theta$, is given by: for all $\theta \in \Theta, \; h >0$,

\begin{equation}\label{eq:gradWkernel}
\nabla_{\theta} \left(\widehat{W}^{\phi}_{n,m,h}(s_{\theta})\right) =\sum_{i=1}^{n}  \phi' \left(\widehat{F}_{s_{\theta},N,h}(s_{\theta}(\bX_i))\right)\nabla_{\theta}\left( \widehat{F}_{s_{\theta},N,h}(s_{\theta}(\bX_i)) \right)~,
\end{equation}

where the gradient of $\widehat{F}_{s_{\theta},N,h}(s_{\theta}(z))$ \wrt to $\theta$ is:

\begin{multline}\label{eq:gradcdfkernel}
\nabla_{\theta} \left(\widehat{F}_{s_{\theta},N,h}(s_{\theta}(z)) \right)=  \frac{1}{Nh}  \sum_{i =1}^n K\left( \frac{s_{\theta}(z) - s_{\theta}(\bX_i)}{h}   \right)  \left(\partial_{\theta}s_{\theta}(z) - \partial_{\theta}s_{\theta}(\bX_i)\right)  \\
+ \frac{1}{Nh}  \sum_{j=1}^m K\left( \frac{s_\theta(z) - s_{\theta}(\bY_j)}{h}   \right) \left(\partial_{\theta}s_{\theta}(z) - \partial_{\theta}s_{\theta}(\bY_j)\right)~,
\end{multline}
for any $z \in \Z$, using the fact that $\kappa'=K$.\\

\begin{algorithm}[H]\label{algo:GA}
	\SetAlgoLined
	
	\KwData{Independent $\iid$ samples $\{\bX_i\}_{i\leq n}$ and $\{\bY_j\}_{j\leq m}$.}
	\KwIn{Score-generating function $\phi$, kernel $K$, bandwidth $h >0$, number of iterations $T \geq1$, step size $\eta>0$.}

	\KwResult{Scoring rule $s_{\widehat{\theta}_{n,m}}(z)$.}
	\BlankLine

	Choose the initial point $\theta^{(0)} $ in $\Theta$\;
	
	\For{$t=0, \ldots, \; T-1$}{
		 compute the gradient estimate $ \nabla_{\theta } \left(\widehat{W}^{\phi}_{n,m,h}(s_{\theta^{(t)}}) \right)$ \;
		 
		 update the parameter $\theta^{(t+1)} = \theta^{(t)}  + \eta \nabla_{\theta } \left( \widehat{W}^{\phi}_{n,m,h}(s_{\theta^{(t)}}) \right)$ \;

	}
 Set $\widehat{\theta}_{n,m}= \theta^{(T)}$.
\caption{Gradient Ascent for $W$-ranking performance maximization}
\end{algorithm}

\medskip

In practice, the iterations are continued until the order of magnitude of the variations $\vert\vert\theta^{(t+1)}-\theta^{(t)}\vert\vert$ becomes negligible.
Then, the approximate maximum $s_{\widehat{\theta}_{n,m}}(z)$ output by Algorithm \ref{algo:GA} is next used to rank test data. Averages over several Monte-Carlo replications are computed in order to produce the results displayed in Subsection \ref{sec:resdisc}.  
\subsection{Synthetic Data Generation}\label{sec:synthdata}

We now describe the data generating models used in the simulation experiments, as well as the parametric class of scoring functions, which the learning algorithm previously described is applied to.

\paragraph{Score-generating functions.}

To illustrate the importance of the function $\phi$ in the $W_{\phi}$-performance ranking criterion, we successively consider $\phi_{MWW}(u) = u$ (MWW), $\phi_{Pol}(u) = u^q, \; q \in\NN^* $  (Pol, \cite{Rud06}) and  $\phi_{RTB}(u) = \text{SoftPlus}(u-u_0) + u_0  \text{Sigmoid}(u-u_0)$, $u_0\in (0,1)$ (RTB, smoothed version of \cite{CV07}), where the activation functions are defined by: $\text{SoftPlus}(u) =  (1/\beta) \log (1+ \exp(\beta u))$ and $\text{Sigmoid}(u) = 1/(1 + \exp(-\lambda u))$, $\beta, \lambda >0$ being hyperparameters to fit and control the derivative's slope.

\begin{figure}[!h]
	\centering
	\includegraphics[width=6cm, height=6cm]{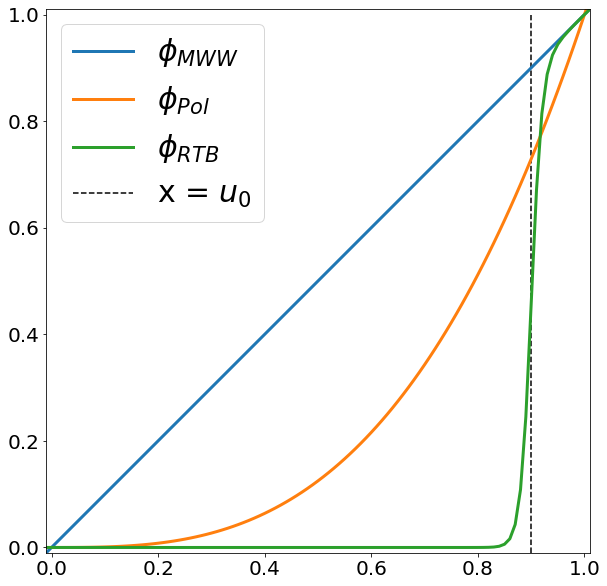}
	\caption{Curves of the three score-generating functions under study: $\phi_{MWW}(u)=u$ in blue,  $\phi_{Pol}(u)=u^3$ in orange, $\phi_{RTB}(u)= \text{SoftPlus}(u-u_0) + u_0  \text{Sigmoid}(u-u_0)$ the smoothed version of $u \mapsto u\mathbb{I}\{u\geq u_0\}$ in green, vertical line at $x = u_0$ in black.}
	\label{fig:scoregenplot}
\end{figure}

\paragraph{Probabilistic models.}
Two classic two-sample statistical models are used here, namely the location and the scale models, where both samples are drawn from multivariate Gaussian distributions.
We denote by $S_d^{+} (\RR)$ the set of positive definite matrices of dimension $d\times d$, by $ \I_d$ the identity matrix. 

\paragraph{Location model.} Inspired by the optimality properties of linear rank statistics regarding shift detection in the univariate setup (\CF Subsection \ref{subsec:2sample_rank_stats}), the model considered stipulates that $\bX \sim \N_d (\mu_X, \Sigma)$ and $\bY \sim \N_d (\mu_Y, \Sigma)$ where $\Sigma\in S_d^{+} (\RR)$ and the mean/location parameters $\mu_X$ and $\mu_Y$ differ. 
The Algorithm \ref{algo:GA} is implemented here with $\Z=\mathbb{R}^d=\Theta$ and $\S_0 = \{s_{\theta}(\cdot)=\langle \cdot, \theta \rangle,\;\; \theta\in \Theta \}$ as class of scoring functions, where $ \langle \cdot, \cdot \rangle$ denotes the Euclidean scalar product on the feature space $\RR^d$, and consequently exhibits no bias caused by the model. Indeed, by computing the $\log$likelihood ratio, one may easily check that the function $\langle \theta^*,\; \cdot \rangle$, where $\theta^* = \Sigma^{-1}(\mu_X-\mu_Y)$, is an optimal scoring function for the related bipartite ranking problem.
Denoting by $\Delta(t)=(1/\sqrt{2\pi})\int_{-\infty}^t\exp(-u^2/2)du$, $t\in \mathbb{R}$, the $\cdf$ of the centered standard univariate Gaussian distribution, one may immediately check that the optimal $\roc$ curve is given by:
\begin{equation*}
\forall \alpha\in (0,1),\;\; \roc^*(\alpha)= 1-\Delta\left( \Delta^{-1}(1-\alpha) + \sqrt{( \mu_X-\mu_Y)^T\Sigma^{-1}(\mu_X-\mu_Y)}\right) .
\end{equation*}

Three levels of difficulty are tested through the implementations Loc1, Loc2 and Loc3. The nearly diagonal covariance matrix of the three models has its eigenvalues in $[0.5,1.5]$ and  $\mu_X = (1 +\varepsilon ) \mu_Y $ with $\varepsilon = 0.10$ (\resp $\varepsilon = 0.20$ and $\varepsilon = 0.30$) for Loc1 (\resp Loc2 and Loc3). The empirical $\roc$ curves over the test pooled samples and additional curves are depicted in Fig. \ref{fig:rocloc1}, \ref{fig:rocloc2}, \ref{fig:rocloc3} for \resp Loc1, 2 and 3. The averaged $\roc$ curves and the \textit{best} one are gathered for the three models in Fig. \ref{fig:roccurvesallloc}. In Fig. \ref{fig:losslocall}, the evolution of the averaged empirical value of the $W_{\phi}$-criteria on the train set during the algorithm is computed. Fig. \ref{fig:roclocrtball} shows the results for Loc2 and 3 for three different parameters of the RTB model with $u_0 \in \{0.70, \; 0.90, \; 0.95\}$.

\paragraph{Scale model.} Consider now the situation where $\bX \sim \N_d (\mu, \Sigma_X)$ and $\bY \sim \N_d (\mu, \Sigma_Y)$, the distributions having the same location vector $\mu \in \RR^d$ but different scale parameters $\Sigma_X$ and $\Sigma_Y$ in $S_d^{+} (\RR)$. The Algorithm \ref{algo:GA} is implemented with  $\Z=\mathbb{R}^d$, $\Theta = S_d^{+} (\RR)$ and $\S_0 = \{s_{\theta}(z)= \langle z,\theta^{-1} z \rangle,\;\; \text{for all} \quad z \in \Z, \;\; \theta\in \Theta \}$, with the notations previously introduced. By computing the likelihood ratio, one immediately checks that $s_{\theta^*}(\cdot)$, with $\theta^* = \Sigma_X^{-1} - \Sigma_Y^{-1}$, is an optimal scoring function for the related scale model. For models Scale1, Scale2 and Scale3, observations are centered, $\Sigma_Y = \I_d$ and $\Sigma_X = \I_d + (\varepsilon/d) H$, where $\varepsilon$ is taken equal to $0.70$, $0.80$ and $0.90$ respectively and $H$ a $d\times d$  symmetric matrix with real entries such that all the eigenvalues of $\Sigma_X \in S_d^{+} (\RR)$ are close to $1$. 

Similar to the location models, the empirical $\roc$ curves over the test pooled samples and additional curves are depicted in Fig. \ref{fig:rocscale1}, \ref{fig:rocscale2}, \ref{fig:rocscale3} for \resp Scale1, 2 and 3. The averaged $\roc$ curves and the \textit{best} one are gathered for the three models in Fig. \ref{fig:roccurvesallscale}. In Fig. \ref{fig:lossscaleall}, the evolution of the averaged empirical value of the $W_{\phi}$-criteria on the train set during the Algorithm is computed. Fig. \ref{fig:rocscalertball} shows the results for Scale2 for three different parameters of the RTB model with $u_0 \in \{0.60, \; 0.70, \; 0.80\}$.

\paragraph{Experimental parameters.} 
In all the experiments below, the pooled train sample is balanced, \ie
$n = m = 150$ and the dimension of the feature space is $d = 15$. Similarly for the test sample with $n = m = 10^6$ and $d = 15$. Concerning the score-generating functions, we consider $q=3$ (Pol)  and $u_0 = 0.9$ (RTB). We use the Gaussian smoothing kernel $K(u) =(1/\sqrt{2\pi}) \exp \{-u^2/2\}$ with a bandwidth $h\sim N^{-1/5}$, yielding an (asymptotically) optimal trade-off between bias and variance. Algorithm \ref{algo:GA} is implemented with $T = 50$ and a learning step size $\eta$ of order $1/\sqrt T$. For each model, $B = 50$ Monte-Carlo replications of the train pooled sample. Based on the latter, a standard deviation for the test average ROC curve is computed for each model.

\paragraph{Evaluation of the criteria.}\label{subsection:evalcrit}
In order to evaluate the performance of the scoring function produced by an early-stopped version of Algorithm \ref{algo:GA} depending on the score-generating function chosen, it is used to score the test sample and the corresponding $\roc$ curves and its average are compared to those of the optimal scoring function $s_{\theta^*}(z)$. Also we consider the \textit{best}/\textit{worst} curves in the sense of \textit{resp.} the minimization/maximization of the generalization error of the set of $\roc$ curves obtained computed over the test pooled sample. 
Particular attention is paid to the behavior of these curves near the origin, which reflects the ranking performance for the instances with highest score values.

\subsection{Results and Discussion}\label{sec:resdisc}

 We now analyze the experimental results, by commenting on the test $\roc$ curves obtained after learning the scoring functions, using the early-stopped version of the Algorithm \ref{algo:GA} described above, that maximize the chosen (smoothed variant of the) $W_{\phi}$-performance measure: MWW, Pol and RTB. We compare them with $\roc^*$.
 All the experiments were run using Python. \\
 
 For both the location and scale models, we ran the algorithm for three increasing levels of difficulty defined by the decreasing value of the parameter $\varepsilon$. 
 Figures \ref{fig:roccurvesallloc} (location) and \ref{fig:roccurvesallscale} (scale) show that the three methods (MWW, Pol, RTB) learn an empirical parameter $\widehat{\theta}_{n,m}$ such that the corresponding $\roc$ curve gets close to $\roc^*$ (red curves)  and the more $\varepsilon$ increases and the more the scoring rule learned generalizes well. Fig. \ref{fig:losslocall} (location) and \ref{fig:lossscaleall} (scale) reveal the monotonicity of the evolution of the empirical criteria, as the number of iterative steps of Algorithm \ref{algo:GA} increases. Unsurprisingly, all the results show an increasing ability to learn a scoring function that maximizes the three $W_{\phi}$-performance measures, as $\varepsilon$ increases (\textit{i.e.} when the distribution $G$ and $H$ are significantly more different from each other). \\

 Analyzing the average of the empirical $\roc$ curves obtained, MWW performs better for the location model as its corresponding curve converges faster to $\roc^*$ for all $\varepsilon$. This phenomenon was expected due to the well-known high power of the related Mann-Whitney-Wilcoxon test statistic in this modeling. The aggregated $\roc$ curve for the Pol method also performs well, while RTB's presents a low performance compared to MWW, see Fig. \ref{fig:roccurvesallloc}. Indeed, considering only the best ranked observations at each iteration in the learning procedure, does not always achieve a good scoring parameter and is enhanced by the early-stopped rule. It results in a higher variance and a larger spectrum of the empirical curves both at the same time, see the light blue curves in Fig. \ref{fig:rocloc2}.3. and \ref{fig:rocloc3}.3. (Loc2 and Loc3). The slow convergence for the RTB method is illustrated with Loc1, where almost both samples are blended/coincide, for which only the $\roc$ curves above the diagonal were kept. For the scale model, the aggregated $\roc$ curves are comparable for the three methods with a slightly higher performance obtained by RTB and we note the faster convergence of the algorithm for this model, see Fig. \ref{fig:lossscaleall}.\\

 Looking at the \textit{best} $\roc$ curves (dark blue lines), defined as those obtained by the scoring function minimizing the generalization error for each criterion, RTB yields to a scoring function that generalizes best for most of the models. In particular, when focussing on the 'best' instances in the learning procedure, the obtained empirical scoring functions have higher performance at the beginning of the $\roc$ curves, see the zoomed plots. Also, choosing the optimal proportion $1 - u_0$ of observations to consider for the score-generating function results in different performance measures. Figure \ref{fig:roclocrtball} gathers the resulting plots for models Loc2 and 3 with $u_0$ in $\{0.7, 0.9, 0.95\}$ while Fig. \ref{fig:rocscalertball} depicts the scale model 2 with $u_0$ in $\{0.6, 0.7, 0.8\}$ and a higher number of loops $T = 70$. Considering the \textit{best} $\roc$ curves for all models shows that when $u_0$ tends to one, the beginning of the curve is accurately learned. Incidentally, note that the proportion of observations considered has to be large enough, so that the optimization algorithm performs well.

 \begin{figure}[!h]
 	\centering
 	\begin{tabular}{cc}
 		\parbox{5.5cm}{	
 			\includegraphics[width=5.5cm, height=5cm]{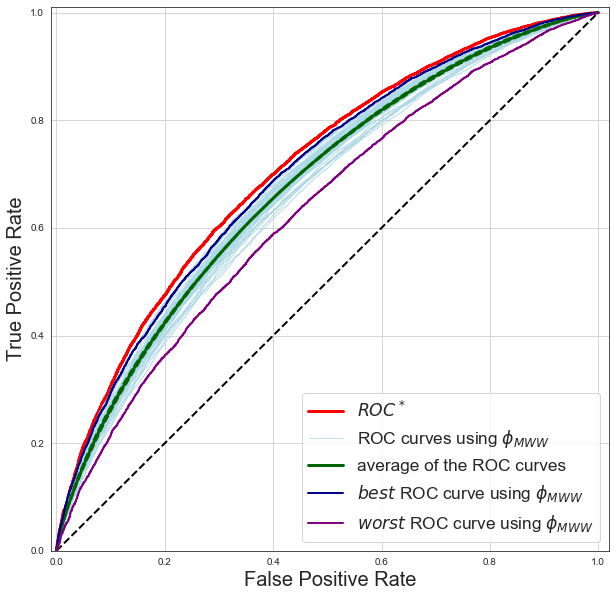}
			{\scriptsize 1.a. $\phi_{MWW}(u) = u$  }\\
 			\includegraphics[width=5.5cm, height=5cm]{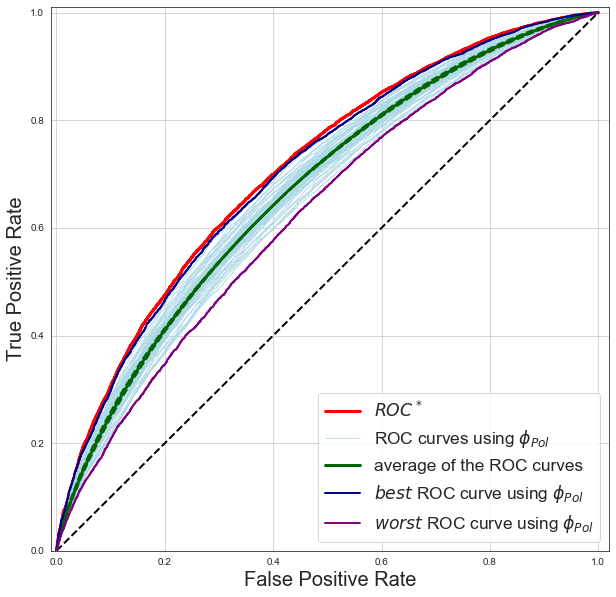}		
 			{\scriptsize 2.b. $\phi_{Pol}(u) = u^3$}\\
 		}
 		\parbox{5.5cm}{
 			\includegraphics[width=5.5cm, height=5cm]{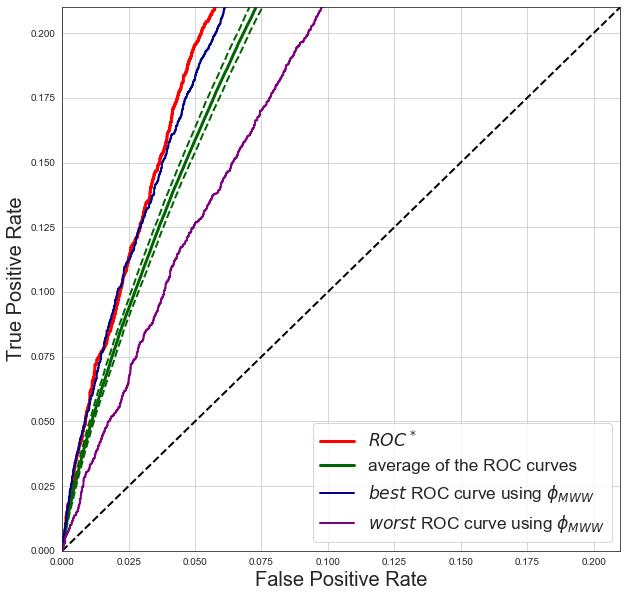}
 			{\scriptsize 1.a. $\phi_{MWW}(u) = u$ }\\
 			\includegraphics[width=5.5cm, height=5cm]{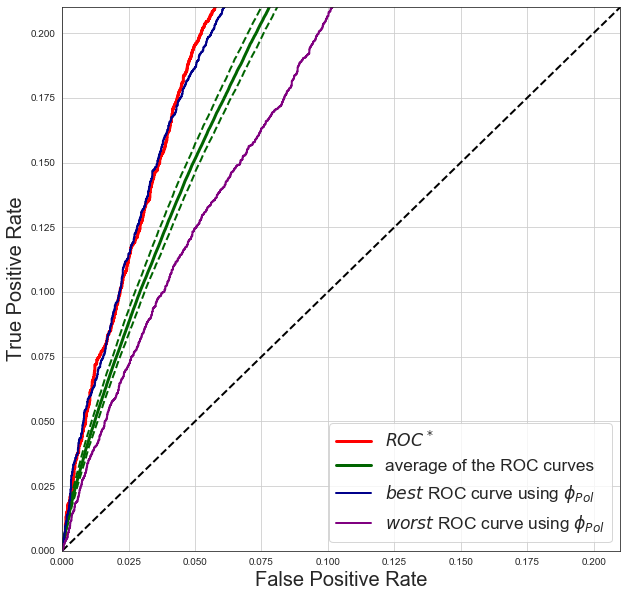}
 			{\scriptsize 2.b. $\phi_{Pol}(u) = u^3$}\\
 		}
 		\medskip
 	\end{tabular}
	
 \end{figure}

\begin{figure}[!h]
	\centering
	\begin{tabular}{cc}
		\parbox{5.5cm}{	
			\includegraphics[width=5.5cm, height=5cm]{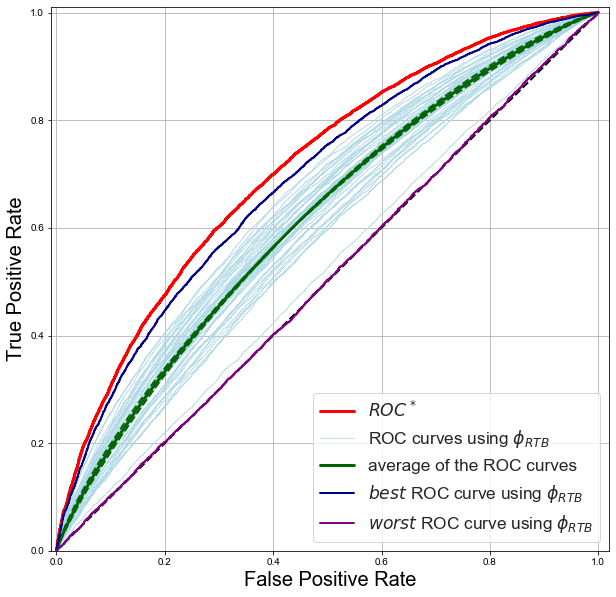}		
			{\scriptsize 3.c. $\phi_{RTB}(u) = u\mathbb{I}\{u \geq 0.9\}$}\\
		}
		\parbox{5.5cm}{
			\includegraphics[width=5.5cm, height=5cm]{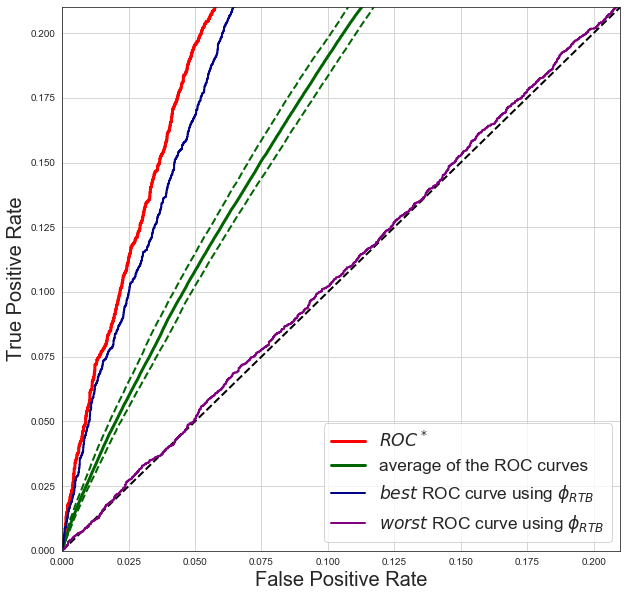}		
			{\scriptsize 3.c. $\phi_{RTB}(u) =  u\mathbb{I}\{u \geq 0.9\}$}\\
		}
		\medskip
	\end{tabular}
	\caption{Empirical $\roc$ curves and average $\roc$ curve for Loc2  ($\varepsilon = 0.20$). Samples are drawn from multivariate Gaussian distributions according to section \ref{sec:synthdata},
		scored with early-stopped GA algorithm's optimal parameter for the class of scoring functions. Hyperparameters: $u_0 = 0.9$, $q = 3$, $B = 50$, $T = 50$. Parameters for the training set: $n=m=150$; $d=15$; for the testing set:  $n=m=10^6$; $d=15$.
		Figures $1, 2, 3$ correspond resp. to the models MMW, Pol, RTB. Light blue curves are the $B(=50)$ $\roc$ curves that are averaged in green (solid line) with $+/-$ its standard deviation (dashed green lines). The dark blue and purple curves correspond to the best and worst scoring functions in the sense of minimization and maximization of the generalization error among the $B$ curves. The red curve corresponds to $\roc^*$.}
	\label{fig:rocloc2}
\end{figure}

  \begin{figure}[!h]
 	\centering
 	\begin{tabular}{cc}
 		\parbox{4cm}{	
 			\includegraphics[width=4cm, height=4cm]{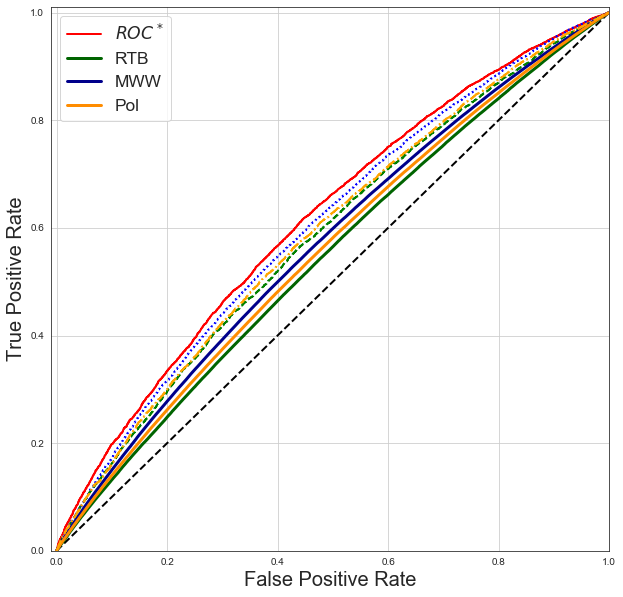}
 			\includegraphics[width=4cm, height=4cm]{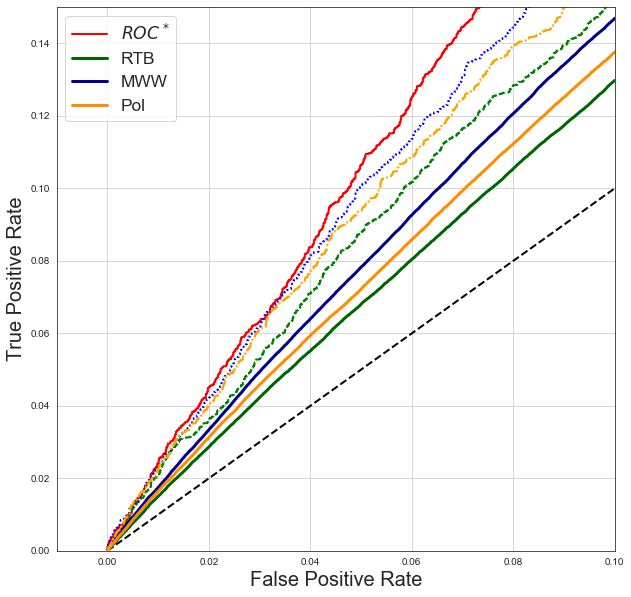}
 			{\scriptsize 1. Loc1, $\varepsilon = 0.10$}\\
 		}
 		\parbox{4cm}{
 			\includegraphics[width=4cm, height=4cm]{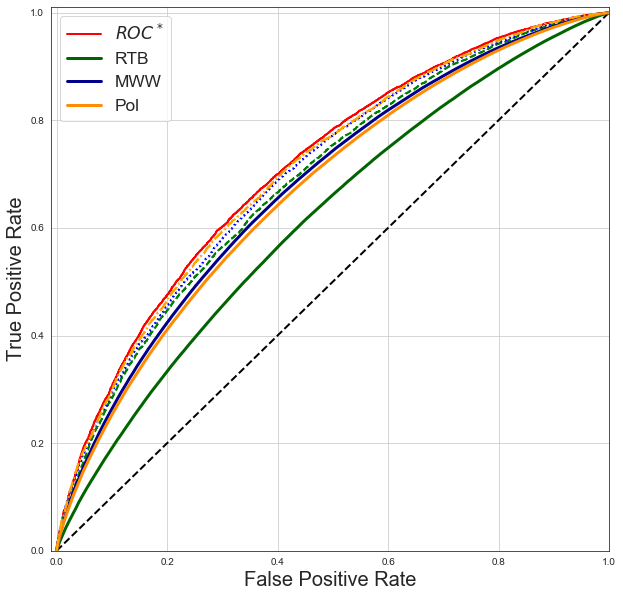}
 			\includegraphics[width=4cm, height=4cm]{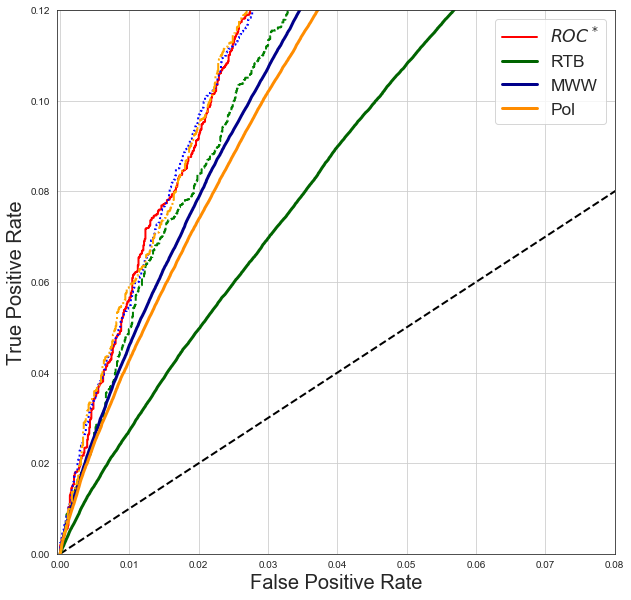}
 			{\scriptsize 2. Loc2, $\varepsilon = 0.20$}\\
 		}
 		\parbox{4cm}{	
 			\includegraphics[width=4cm, height=4cm]{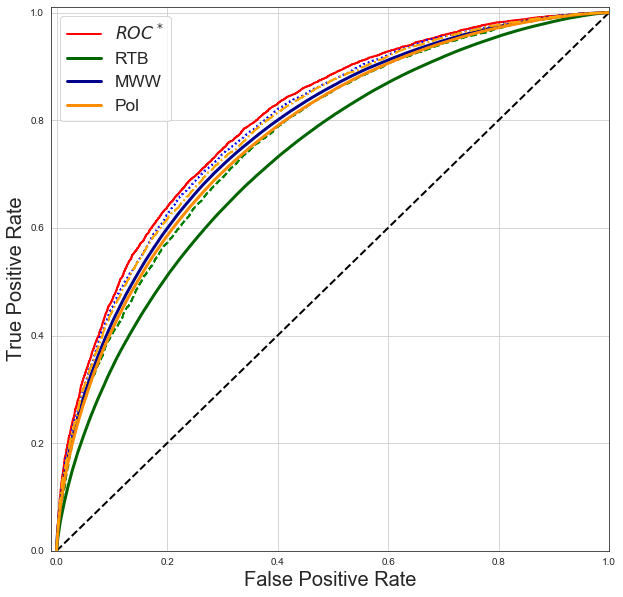}
 			\includegraphics[width=4cm, height=4cm]{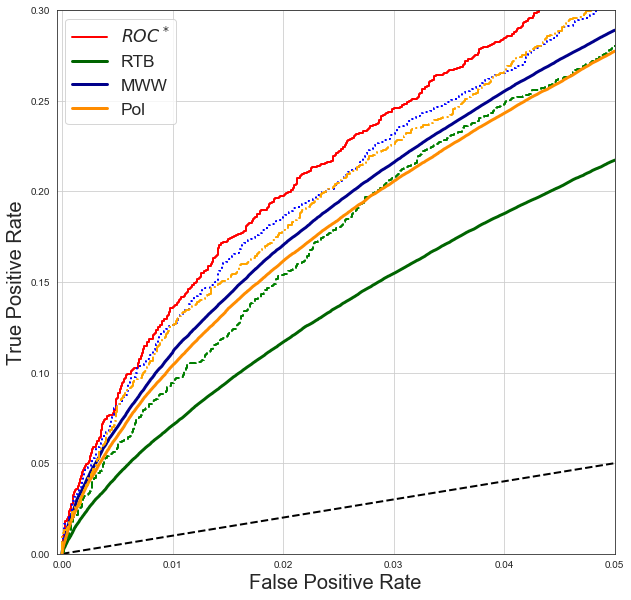}
 			{\scriptsize 3. Loc3, $\varepsilon = 0.30$}\\
 		}
 		\medskip
 	\end{tabular}
 	\caption{Average of the $\roc$ curves (solid line), \textit{best} $\roc$ curves (dashed line) for the three location models Loc1, Loc2 and Loc3. In blue for MWW, orange for Pol, green for RTB, red for $\roc^*$. Samples are drawn from multivariate Gaussian distributions according to section \ref{sec:synthdata}, scored with early-stopped GA algorithm's optimal parameter for the class of scoring functions and averaged after $B = 50$ loops. Hyperparameters: $u_0 = 0.9$; $q = 3$, $B = 50$, $T = 50$. Parameters for the training set: $n=m=150$; $d=15$; for the testing set:  $n=m=10^6$; $d=15$.}
 	\label{fig:roccurvesallloc}
 \end{figure}

 \begin{figure}[!h]
 	\centering
 	\begin{tabular}{cc}
 		\parbox{4cm}{	
 			\includegraphics[width=4cm, height=4cm]{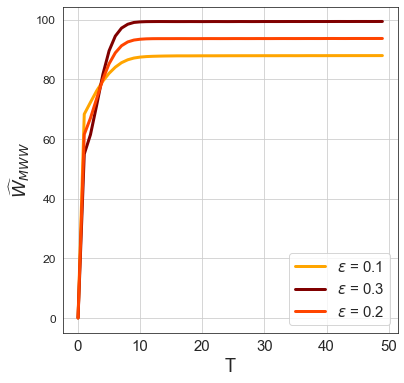}
 			{\scriptsize 1. $\phi_{MWW}(u) = u$ }\\
 		}
 		\parbox{4cm}{
 			\includegraphics[width=4cm, height=4cm]{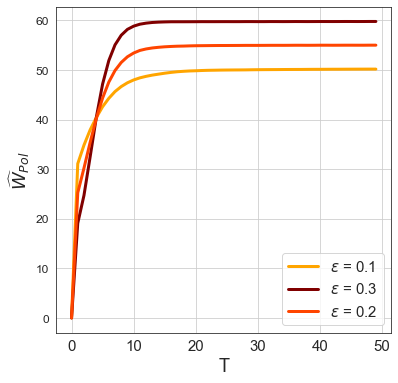}
 			{\scriptsize 2. $\phi_{Pol}(u) = u^3$}\\
 		}
 		\parbox{4cm}{	
 			\includegraphics[width=4cm, height=4cm]{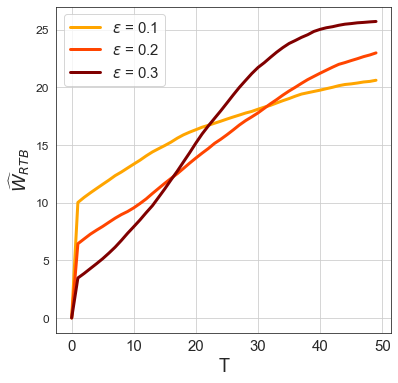}
 			{\scriptsize 3. $\phi_{RTB}(u) = u\mathbb{I}\{u \geq 0.9\}$}\\
 		}
 		\medskip
 	\end{tabular}
 	\caption{Average of the empirical $W_{\phi}$-ranking performance measure over the $B=50$ loops for the three location models Loc1, Loc2 and Loc3. Samples are drawn from multivariate Gaussian distributions according to section \ref{sec:synthdata}, scored with early-stopped GA algorithm's optimal parameter for the class of scoring functions and averaged after $B = 50$ loops. Hyperparameters: $u_0 = 0.9$; $q = 3$, $B = 50$, $T = 50$. Parameters for the training set: $n=m=150$; $d=15$; for the testing set:  $n=m=10^6$; $d=15$.}
 	\label{fig:losslocall}
 \end{figure}

 \begin{figure}[!h]
 	\centering
 	\begin{tabular}{cc}
 		\parbox{5.5cm}{	
 			\includegraphics[width=5.5cm, height=5cm]{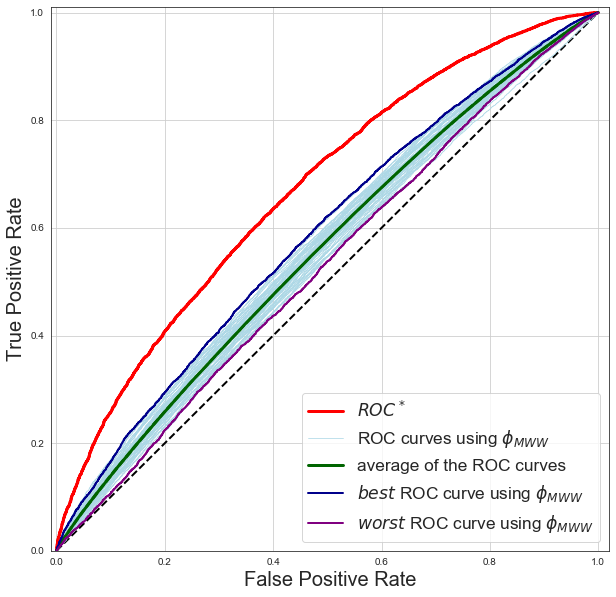}
 			{\scriptsize 1.a. $\phi_{MWW}(u) = u$  }\\
 			\includegraphics[width=5.5cm, height=5cm]{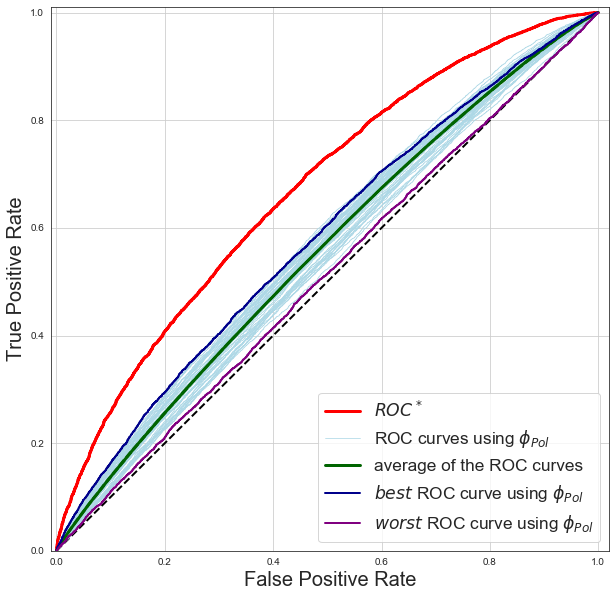}		
 			{\scriptsize 2.b. $\phi_{Pol}(u) = u^3$}\\
 		}
 		\parbox{5.5cm}{
 			\includegraphics[width=5.5cm, height=5cm]{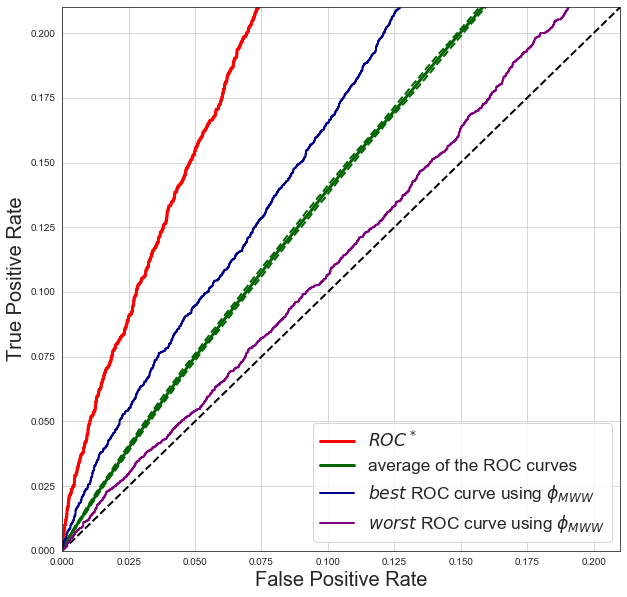}
 			{\scriptsize 1.a. $\phi_{MWW}(u) = u$ }\\
 			\includegraphics[width=5.5cm, height=5cm]{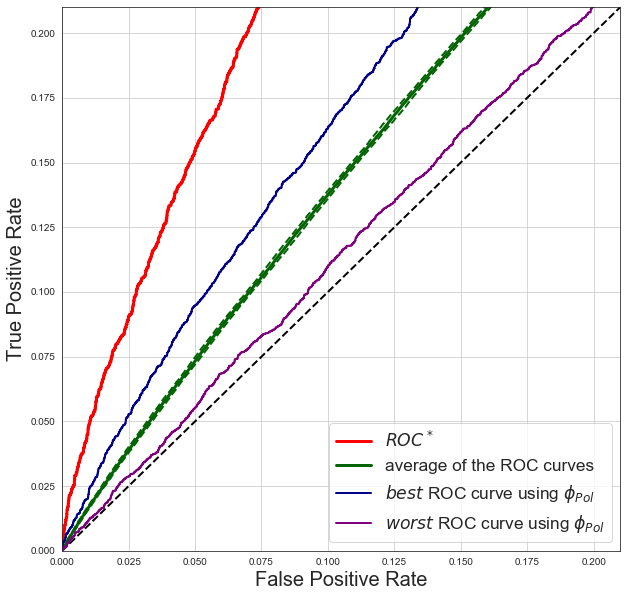}
 			{\scriptsize 2.b. $\phi_{Pol}(u) = u^3$}\\
 		}
 		\medskip
 	\end{tabular}
 	
 \end{figure}
 
 \begin{figure}[!h]
 	\centering
 	\begin{tabular}{cc}
 		\parbox{5.5cm}{	
 			\includegraphics[width=5.5cm, height=5cm]{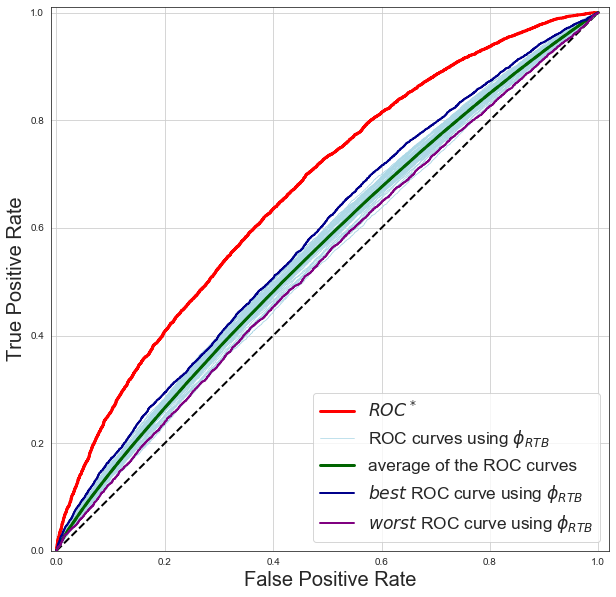}		
 			{\scriptsize 3.c. $\phi_{RTB}(u) = u\mathbb{I}\{u \geq 0.9\}$}\\
 		}
 		\parbox{5.5cm}{
 			\includegraphics[width=5.5cm, height=5cm]{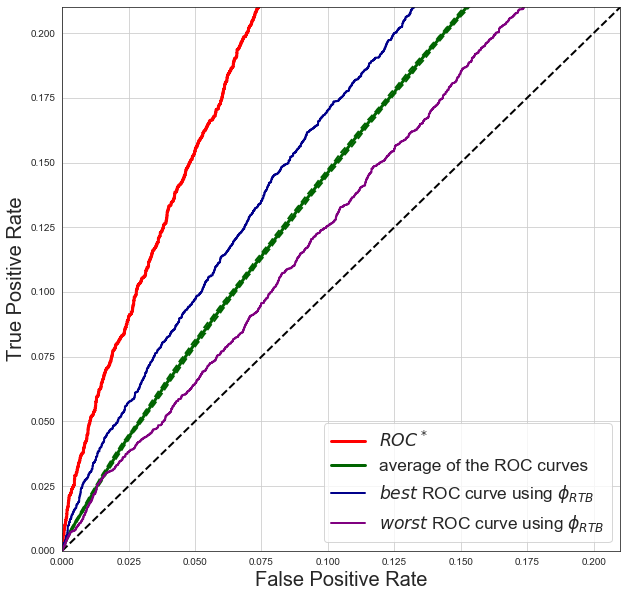}		
 			{\scriptsize 3.c. $\phi_{RTB}(u) =  u\mathbb{I}\{u \geq 0.9\}$}\\
 		}
 		\medskip
 	\end{tabular}
 	\caption{Empirical $\roc$ curves and average $\roc$ curve for Scale1  ($\varepsilon = 0.70$). Samples are drawn from multivariate Gaussian distributions according to section \ref{sec:synthdata},
 		scored with early-stopped GA algorithm's optimal parameter for the class of scoring functions. Hyperparameters: $u_0 = 0.9$, $q = 3$, $B = 50$, $T = 50$. Parameters for the training set: $n=m=150$; $d=15$; for the testing set:  $n=m=10^6$; $d=15$.
 		Figures $1, 2, 3$ correspond \resp to the models MMW, Pol, RTB. Light blue curves are the $B(=50)$ $\roc$ curves that are averaged in green (solid line) with $+/-$ its standard deviation (dashed green lines). The dark blue and purple curves correspond to the best and worst scoring functions in the sense of minimization and maximization of the generalization error among the $B$ curves. The red curve corresponds to $\roc^*$.}
 	\label{fig:rocscale1}
 \end{figure}

 \begin{figure}[!h]
 	\centering
 	\begin{tabular}{cc}
 		\parbox{4cm}{	
 			\includegraphics[width=4cm, height=4cm]{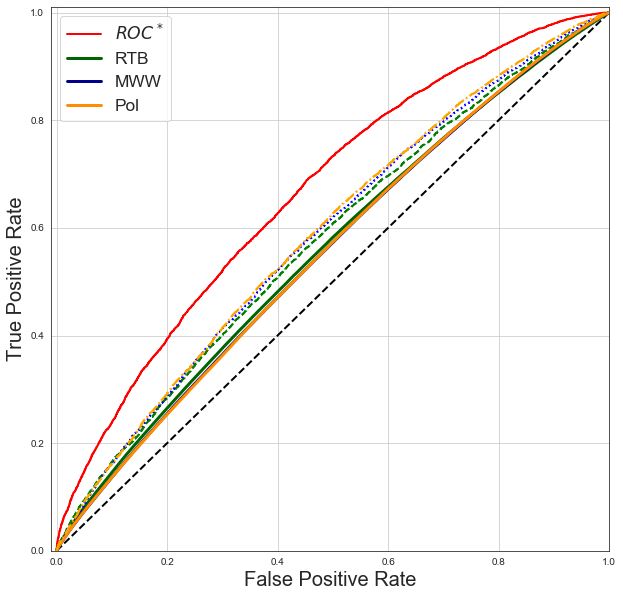}
 			\includegraphics[width=4cm, height=4cm]{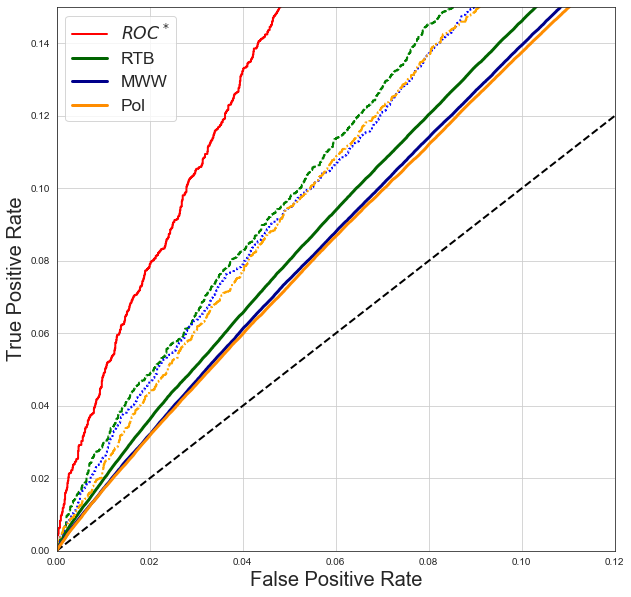}
 			{\scriptsize 1. Scale1, $\varepsilon = 0.70$}\\
 		}
 		\parbox{4cm}{
 			\includegraphics[width=4cm, height=4cm]{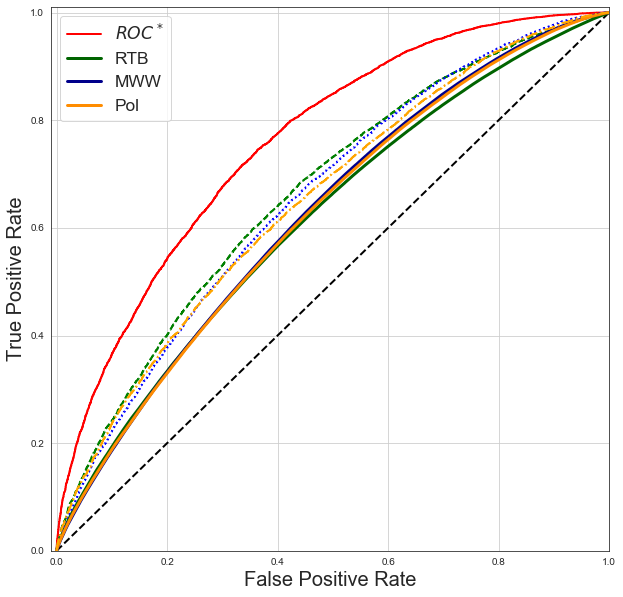}
 			\includegraphics[width=4cm, height=4cm]{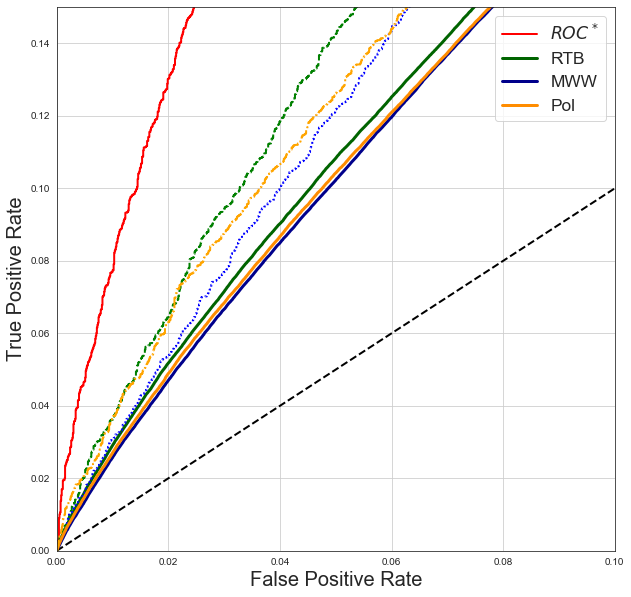}
 			{\scriptsize 2. Scale2, $\varepsilon = 0.90$}\\
 		}
 		\parbox{4cm}{	
 			\includegraphics[width=4cm, height=4cm]{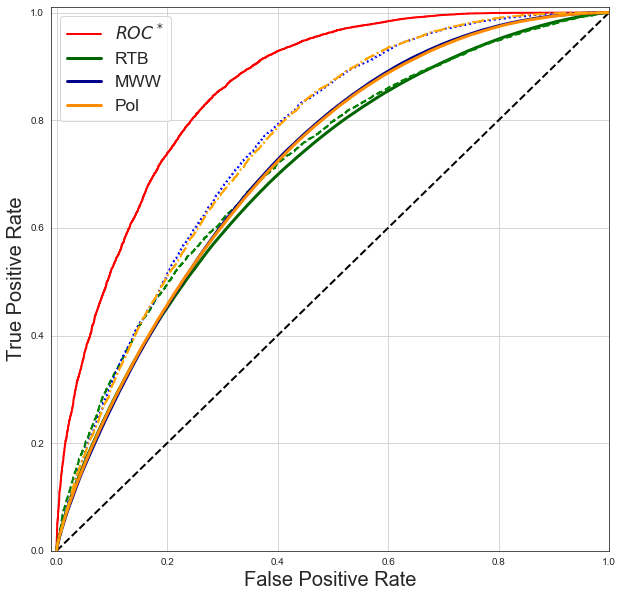}
 			\includegraphics[width=4cm, height=4cm]{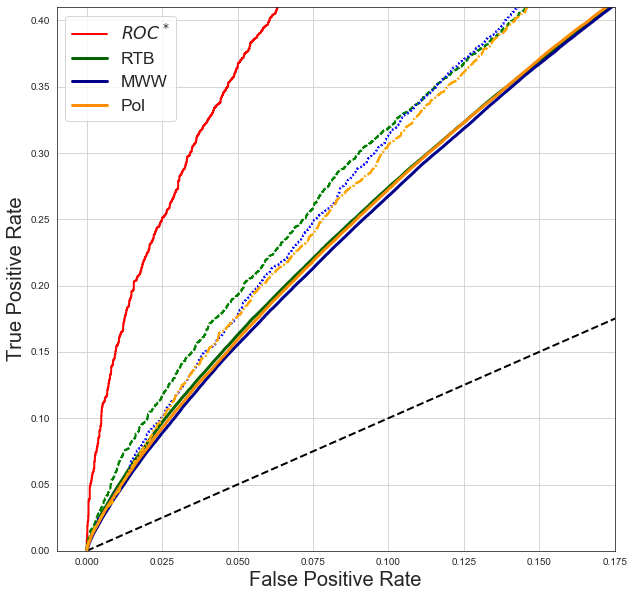}
 			{\scriptsize 3. Scale3, $\varepsilon = 1.10$}\\
 		}
 		\medskip
 	\end{tabular}
 	\caption{Average of the $\roc$ curves (solid line), \textit{best} $\roc$ curves (dashed line) for the three scale models Scale1, Scale2 and Scale3. In blue for MWW, orange for Pol, green for RTB, red for $\roc^*$. Samples are drawn from multivariate Gaussian distributions according to section \ref{sec:synthdata}, scored with early-stopped GA algorithm's optimal parameter for the class of scoring functions and averaged after $B = 50$ loops. Hyperparameters: $u_0 = 0.9$; $q = 3$, $B = 50$, $T = 50$. Parameters for the training set: $n=m=150$; $d=15$; for the testing set:  $n=m=10^6$; $d=15$.}
 	\label{fig:roccurvesallscale}
 \end{figure}

  \begin{figure}[!h]
 	\centering
 	\begin{tabular}{cc}
 		\parbox{4cm}{	
 			\includegraphics[width=4cm, height=4cm]{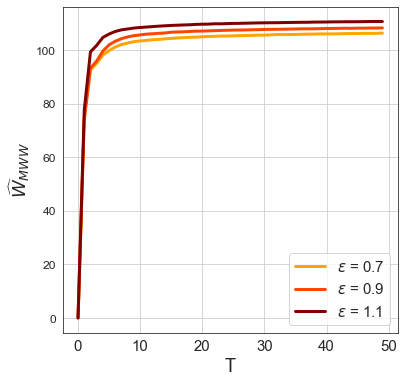}
 			{\scriptsize 1. $\phi_{MWW}(u) = u$ }\\
 		}
 		\parbox{4cm}{
 			\includegraphics[width=4cm, height=4cm]{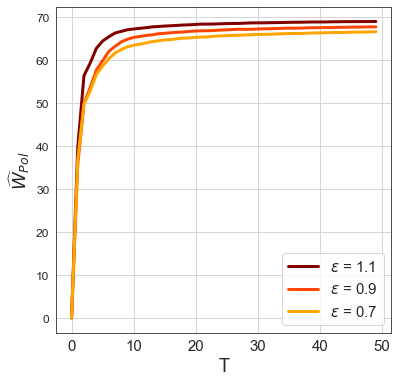}
 			{\scriptsize 2. $\phi_{Pol}(u) = u^3$}\\
 		}
 		\parbox{4cm}{	
 			\includegraphics[width=4cm, height=4cm]{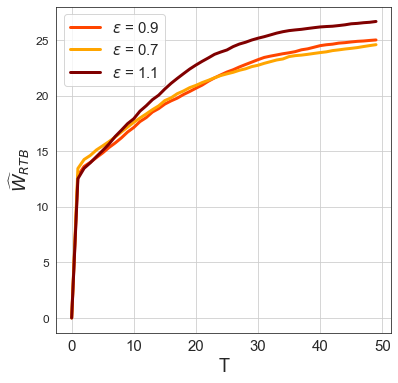}
 			{\scriptsize 3. $\phi_{RTB}(u) = u\mathbb{I}\{u \geq 0.9\}$}\\
 		}
 		\medskip
 	\end{tabular}
 	\caption{Average of the empirical $W_{\phi}$- ranking performance measure over the $B=50$ loops for the three location models Loc1, Loc2 and Loc3. Samples are drawn from multivariate Gaussian distributions according to section \ref{sec:synthdata}, scored with early-stopped GA algorithm's optimal parameter for the class of scoring functions and averaged after $B = 50$ loops. Hyperparameters: $u_0 = 0.9$; $q = 3$, $B = 50$, $T = 50$. Parameters for the training set: $n=m=150$; $d=15$; for the testing set:  $n=m=10^6$; $d=15$.}
 	\label{fig:lossscaleall}
 \end{figure}

\section{Conclusion}

This article argues that two-sample linear rank statistics provide a very flexible and natural class of empirical performance measures for bipartite ranking. We have showed that it encompasses in particular well-known criteria used in medical diagnosis and information retrieval and proved that, in expectation, these criteria are maximized by optimal scoring functions and put the emphasis on specific parts of their ROC curves, depending on the score generating function involved in the criterion considered. We have established concentration results for collections of such statistics, referred to as \textit{two-sample rank processes} here, under general assumptions and have deduced from them statistical learning guarantees for the maximizers of such ranking criteria in the form of a generalization bound of order $O_{\mathbb{P}}(1/\sqrt{N})$, where $N$ means the size of the pooled training sample. Algorithmic issues concerning practical maximization have also been investigated and we have displayed numerical results supporting the theoretical analysis carried out.

\appendix

\section{Definitions and Preliminary Results}\label{app:prelrapp}
For the sake of clarity, crucial concepts and results extensively used in the technical analysis subsequently carried out are first recalled.

\subsection{H\'ajek Projection Method}\label{appsubsec:hajmeth}

The H\'ajek projection method introduced in the seminal contribution \cite{Haj68} aims at decomposing (linearizing) any (possibly complex) square integrable statistic based on independent observations, so as to express it as  an average of independent \rv's plus an uncorrelated term. The proof of Proposition \ref{prop:hajek} crucially relies on this technique. For completeness, it is described in the following lemma, one may refer to Chapter 11 in \cite{vdV98} for further details.

\begin{lemma} {\sc (H\'ajek projection, \cite{Haj68})}  \label{lem:hajek}Let $Z_1,\; \ldots,\; Z_n$ be independent $\rv$'s and $T_n=T_n(Z_1,\; \ldots,\; Z_n)$ be a real-valued square integrable statistic. The \textit{H\'ajek projection} of $T_n$ is defined as $\widehat{T}_n=\sum_{i=1}^n\mathbb{E}[T_n\mid Z_i]-(n-1)\mathbb{E}[T]$.
	It is the orthogonal projection of the square integrable r.v. $T_n$ onto the subspace of all variables of the form $\sum_{i=1}^ng_i(Z_i)$, for arbitrary measurable functions $g_i$ s.t. $\mathbb{E}[g^2_i(Z_i)]<+\infty$.
\end{lemma}

\subsection{$U$-statistics and $U$-processes}\label{app:subUstatproc}
As mentioned in Section \ref{sec:Rproc}, (degenerate) one/two-sample $U$-statistics are involved in the definition of the residual term introduced in Proposition \ref{prop:hajek}. We recall the definition of such statistics generalizing basic $\iid$ sample averages, as well as some of their properties. See \textit{e.g.} \cite{Lee90} for an account of the theory of $U$-statistics.

\begin{definition}\label{def:onesampleU}{\sc (One-sample $U$-Statistic of degree two)} Let $n\geq 2$. Consider a i.i.d. sequence $X_1,\; \ldots,\; X_n$ drawn from a probability distribution $\mu$ on a measurable space $\X$ and $k: \X^2 \rightarrow \R $ a square integrable function \wrt $\mu\otimes\mu$. The one-sample $U$-statistic of degree $2$ and kernel function $k$ based on the $X_i$'s is defined as:
	\begin{equation}\label{eq:1sUstat}
	U_n(k) = \frac{1}{n(n-1)} \sum_{1\leq i\neq j\leq n} k(X_i, X_j)~.
	\end{equation}
\end{definition}
As can be shown by a basic Lehmann-Scheff\'e argument, the statistic $U_n(h)$ is the unbiased estimator of the parameter $\theta(k)=\int k(x_1,\; x_2)\mu(dx_1)\mu(dx_2)$ with minimum variance.  Its H\'ajek projection  can be expressed as follows: the projection of $U_n(k)-\theta(k)$ onto the space of all random variables $\sum_{i=1}^ng_i(X_i)$ with $\int g_i^2(x)\mu(dx)<+\infty$ is $\widehat{U}_n(k)=(1/n)\sum_{i=1}^nk_1(X_i)$, with $k_1=k_{1,1}+k_{1,2}$, $k_{1,1}(x)=\mathbb{E}[k(X_1,\; x)]-\theta$ and $k_{1,2}(x)=\mathbb{E}[k(x,\; X_2)]-\theta$ for all $x\in \X$. The $U$-statistic \eqref{eq:1sUstat} is said to be \textit{degenerate} when the $k_{1,l}(X_1)$'s are equal to zero with probability one, it is then of order $O_{\mathbb{P}}(1/n)$. Hence, once recentered, the $U$-statistic \eqref{eq:1sUstat} can be written as the $\iid$ average $\widehat{U}_n(h)$ plus a degenerate $U$-statistic. This decomposition is known as the (second) Hoeffding representation of $U$-statistics and provides the key argument to establish limit results for such functionals, see \textit{e.g.} \cite{Ser80}. 
\smallskip

The notion of $U$-statistic can be generalized in several ways, by considering kernels with a number of arguments (\textit{i.e.} degree) higher than $2$ or by extending it to the multi-sample framework. 

\begin{definition}\label{def:twosampleU}{\sc (Two-sample $U$-Statistic of degree $(1,1)$)} Let $n,\; m$ in $\NN^*$. Consider two independent i.i.d. sequences $X_1,\; \ldots,\; X_n$ and $Y_1,\; \ldots,\;  Y_m$ respectively drawn from probability distributions $\mu$ and $\nu$ on the measurable spaces $\X$ and $\Y$.  Let $\ell: \X \times \Y \rightarrow \R $ be a square integrable function \wrt $\mu \otimes \nu$. The two-sample $U$-statistic of degree $(1,1)$, with kernel function $\ell(x,y)$ and based on the $X_i$'s and the $Y_j$'s is defined as:
	\begin{equation}\label{eq:2sUstat}
	U_{n,m}(\ell) = \frac{1}{nm} \sum_{i=1}^n\sum_{j=1}^m \ell(X_i,Y_j)~.
	\end{equation}
\end{definition}

A classic example of two-sample $U$-statistic of degree $(1,1)$ is the Mann-Whitney statistic, with symmetric kernel $\ell(x,y)=\mathbb{I}\{ y<x \}+(1/2)\mathbb{I}\{y=x \}$ on $\mathbb{R}^2$ and degree $(1,1)$. It is a natural (unbiased) estimator of the $\auc$: when computed from univariate samples $X_1,\; \ldots,\; X_n$ and $Y_1,\; \ldots,\; Y_m$ with distributions $H$ and $G$ on $\mathbb{R}$, it is equal to $\auc_{\widehat{H}_m,\widehat{G}_n}$ with the notations of Subsection \ref{subsec:2sample_rank_stats} and can be thus viewed as an affine transform of the rank-sum Wilcoxon statistic \eqref{eq:rank_sum}. \\ The H\'ajek projection of \eqref{eq:2sUstat} is obtained by computing the orthogonal projection of the recentered $\rv$ $U_{n,m}(\ell)-\mathbb{E}[U_{n,m}(\ell)]$ onto the subpace of $L_2$ composed of all random variables $\sum_{i=1}^ng_i(X_i)+\sum_{j=1}^mf_j(Y_j)$ with $\int g_i^2(x)\mu(dx)<+\infty$ and $\int f_j^2(y)\nu(dy)<+\infty$, namely $\widehat{U}_{n,m}(\ell)=(1/n)\sum_{i=1}^n\ell_{1,1}(X_i)+(1/m)\sum_{j=1}^m\ell_{1,2}(Y_j)$, with $\ell_{1,1}(x)=\mathbb{E}[\ell(x,\; Y_1)]-\mathbb{E}[U_{n,m}(\ell)]$ and $\ell_{1,2}(y)=\mathbb{E}[\ell(X_1,\; y)]-\mathbb{E}[U_{n,m}(\ell)]$  for all $(x,y)\in \X\times \Y$. The $U$-statistic $U_{n,m}(\ell)$ is said to be \textit{degenerate} when the random variables $\ell_{1,1}(X_1)$ and $\ell_{1,2}(Y_1)$ are equal to zero with probability one. Similar to \eqref{eq:1sUstat}, the recentered version of the two-sample $U$-statistic of degree $(1,1)$ \eqref{eq:2sUstat} can be written as a sum of two $\iid$ averages 
$\widehat{U}_{n,m}(\ell)$ plus a degenerate $U$-statistic of order $O_{\mathbb{P}}(1/n)+O_{\mathbb{P}}(1/m)$. Again, the Hoeffding decomposition is the key to directly extend limit results known for $\iid$ averages (\textit{e.g.} SLLN, CLT, LIL) to statistics of the type \eqref{eq:2sUstat}. In the subsequent technical analysis, nonasymptotic uniform results are required for $U$-processes, namely collections of $U$-statistics indexed by classes of kernels. By means of the Hoeffding decomposition, concentration bounds for $U$-processes can be obtained by combining classic concentration bounds for empirical processes 
and concentration bounds for degenerate $U$-processes, such as those recalled in \ref{subsec:ineq_Yproc}.

\subsection{$\vc$-type Classes of Functions - Permanence Properties}\label{appsubsec:vcperm}

The concentration inequalities for $U$-processes  recalled in Appendix \ref{subsec:ineq_Yproc} and involved in the proof of the main results stated in this article apply to collections of kernels that are of \vc-type, a classic concept used to quantify the complexity of classes of functions. It is recalled below, see $\eg$ \cite{vdVWell96} for generalizations and further details.

\begin{definition}
A class $\mathcal{F}$ of real-valued functions defined on a measurable space $\Z$ is a bounded $\vc$-type class with parameter $(A,\V)\in (0,\; +\infty)^2$ and constant envelope $L_{\mathcal{F}}>0$ if for all $\varepsilon\in (0,1)$:
\begin{equation}\label{eq:covnumb}
\underset{Q}{\sup} \; N(\mathcal{F},L_2(Q),\varepsilon L_{\mathcal{F}})\leq \left(\frac{A}{\varepsilon} \right)^{\V}~,
\end{equation}

where the supremum is taken over all probability measures $Q$ on $\Z$ and the smallest number of $L_2(Q)$-balls of radius less than $\varepsilon$ required to cover class $\mathcal{F}$ (\ie covering number) is meant by $N(\mathcal{F},L_2(Q),\varepsilon)$.
\end{definition}
Recall that a bounded {\sc VC} class of functions with {\sc VC} dimension $V<+\infty$ is of {\sc VC}-type and fulfills the condition above with $\V=2(V-1)$ and $A=(cV(16e)^V)^{1/(2(V-1))}$, where $c$ is a universal constant, see $\eg$ Theorem 2.6.7 in \cite{vdVWell96}.
The lemma stated below permits to control the complexity of the classes of kernels/functions involved in the Hoeffding decompositions of a two-sample $U$-process of degree $(1,1)$  or of a one-sample $U$-process of degree $2$, \textit{cf} subsection \ref{app:subUstatproc}.

 \begin{lemma} \label{lem:permdeg}
 	Let $X$ and $Y$ be two independent random variables, valued in $\X$ and $\Y$ respectively, with probability distributions $\mu$ and $\nu$. Consider $\mathcal{L}$ a {\sc VC}-type bounded class of kernels $\ell: \X \times \Y \rightarrow \R $
 	with parameters $(A,\V)$ and constant envelope $L_{\mathcal{L}} > 0 $. 
 	Then, the sets of functions $\{x\in\X\mapsto \mathbb{E}[\ell(x,\; Y)]:\; \ell\in \mathcal{L}\}$, $\{ y\in\Y\mapsto \mathbb{E}[\ell(X,\; y)]:\; \ell\in \mathcal{L}\}$,  $\{ \ell(x,y) -  \mathbb{E}[\ell(X,\; y)] - \mathbb{E}[\ell(x,\; Y)]:\; \ell\in \mathcal{L}\}$
 	are also {\sc VC}-type bounded classes.
  \end{lemma}

\begin{proof} 
 Consider first the uniformly bounded class $\mathcal{L}_1$ composed of functions $x \in  \X \mapsto \mathbb{E}[\ell(x,\; Y  )]$ with $\ell\in \mathcal{L}$. Let $\varepsilon >0$ and $P$ be any probability measure on $\X$. 
	Define the probability measure 
	$P_{\nu}(dx, dy) = P(dx)\nu(dy)$ on $\X \times \Y$ and consider a $\varepsilon$-covering of the class $ \mathcal{L} $ with centers $\ell_1, \ldots,\; \ell_K$ \wrt the metric $L_2( P_{\nu})$, $K\geq 1$.  For all $\ell \in  \mathcal{L}$, there exists $k\leq K$ such that:
	\begin{eqnarray*}
		\int_{x \in \X}  \int_{y \in \Y} ( \ell(x,\; y)  - \ell_k(x,\; y  ) )^2 P_{\nu}(dx, dy) \leq \varepsilon^2~.
	\end{eqnarray*}
	By virtue of Jensen's inequality, we have
	\begin{multline*}
	\int_{\X}( \mathbb{E} [   \ell(x,\; Y ) ] - \mathbb{E} [ \ell_k(x,\; Y ) ] )^2 P(dx) \\
	\leq\int_{\X} \mathbb{E} [   (\ell(x,\; Y  )-  \ell_k(x,\; Y  ) )^2]  P(dx)\\
	= \int_{\X}  \int_{\Y} ( \ell(x,\; y)  - \ell_k(x,\; y  ) )^2 \nu(dy) P(dx) \leq \varepsilon^2~.
	\end{multline*}
	Hence, one gets a $\varepsilon$-covering of the class $\mathcal{L}_1$ with balls of centers $\{\mathbb{E} [ \ell_k(\cdot,\; Y )   ]:\; k=1,\; \ldots,\; K\}$ in $L_2(P)$. This proves that 
	$$
	N(\L_1,  L_2(P), \varepsilon L_{\mathcal{L}})  \leq N(\L, L_2(P_{\nu}) , \varepsilon L_{\mathcal{L}}) .
	$$
As a similar reasoning can be applied to the two other classes of functions, one then gets the desired result.
	\end{proof}

\subsection{Concentration Inequalities for Degenerate $U$-processes.}\label{subsec:ineq_Yproc}
In \cite{Major2006} (see Theorem 2 therein), a concentration bound for one-sample degenerate $U$-processes of arbitrary degree indexed by $L_2$-dense classes of non-symmetric kernels is established. The lemma below is a formulation of the latter in the specific case of degenerate $U$-processes of degree $2$ indexed by {\sc VC}-type bounded classes of non-symmetric kernels.

\begin{lemma}\label{thm:major2006} Let $n\geq 2$ and $X_1,\; \ldots,\; X_n$ be $\iid$ random variables drawn from a probability distribution $\mu$ on a measurable space $\X$. Let $\mathcal{K}$ be a class of measurable kernels $k: \X^2 \rightarrow \R $ such that $\sup_{x, x' \in \X^2} \vert k(x,x') \vert \leq D < + \infty$ and $\int_{\X^2} k^2(x, x')\mu(dx) \mu(dx') \leq \sigma^2 \leq D^2$, that defines a degenerate one-sample $U$-process of degree $2$, based on the $X_i$'s: $\{U_n(k)\:\; k\in \mathcal{K}\}$. Suppose in addition that the class $\mathcal{K}$ is of {\sc VC}-type with parameters $(A,\V)$. Then, there exist constants $ C_1>0,\; C_2\geq 1$ and $C_3\geq 0$ depending on $(A,\V)$ such that:
	
	\begin{equation}
	\mathbb{P}\left\{\sup_{k \in \mathcal{K}} \left\vert U_n(k) \right\vert \geq t \right\}\leq C_2 \exp\left\{ -\frac{C_3 (n-1)t}{\sigma} \right\}~,
	\end{equation}
	as soon as $C_1\log(2D/\sigma)\leq (n-1)t/\sigma \leq n\sigma^2/D^2$.
\end{lemma}

The next lemma provides a similar nonasymptotic result for degenerate two-sample $U$-processes of degree $(1,1)$.

\begin{lemma}\label{lem:devbound2}
	Let $(n, \; m) \in \NN^*$. Consider  two independent $\iid$ random samples $X_1, \ldots, X_n $ and  $Y_1, \ldots, Y_m$ respectively drawn from the probability distributions $\mu$ and $\nu$ on the measurable spaces $\X$ and $\Y$.  Let $\L$ be a class of degenerate non-symmetrical kernels $\ell : \X \times \Y \to \RR$  such that $\sup_{(x, y) \in \X\times \Y} \vert \ell(x,y) \vert \leq L < + \infty$ and $\int_{\X\times \Y} \ell^2(x, y)\mu(dx) \nu(dy) \leq \sigma^2 \leq L^2$, that defines a degenerate two-sample $U$-process of degree $(1,1)$, based on the $X_i, Y_j$'s: $\{U_{n,m}(\ell)\:\; \ell\in \L\}$. Suppose in addition that the class $\L$ is of {\sc VC}-type with parameters $(A,\V)$. 
Then, for all $t>0$, there exists a universal constant $K>2$ such that:

\begin{equation}
	\mathbb{P} \left\{\sup_{\ell\in \L}  \vert U_{n,m}(\ell) \vert   \geq t  \right\} \leq K 2^{\V}(A/L)^{2\V}e^{4/L^2}\exp \left\{ -  \frac{nmt^2}{ML^2} \right\}~,
\end{equation}

for all $nmt^2 >  \max( 8^4\log(2) L^2 \V ,( \log(2) L^2 \V/2)^{1+\delta}) $, $\delta \in (1,2)$ constant and $M=16^3/2$.
\end{lemma}

Its proof is given in \ref{pf:2s_Uproc} and is inspired from that of 
Lemma 2.14.9 in \cite{vdVWell96} and of Lemma 3.2 in \cite{vdG00} for empirical processes, and from Lemma 2.4 in \cite{Neum2004}  which gives a version in expectation applicable to degenerate two-sample $U$-processes of arbitrary degree indexed by $L_p$-dense classes of kernels.

\section{Technical Proofs}
The proofs of the results stated in the paper are detailed below.
\subsection{Proof of Proposition \ref{prop:hajek}} \label{pf:hajek}

Let $\theta_0\in (0,1)$. Since $\phi(u) \in \mathcal{C}^2([0,1], \RR)$ by virtue of Assumption \ref{hyp:phic2}, a Taylor expansion of order two yields: for all $\theta\in (0,1)$

\begin{equation}\label{eqprop:taylorphi}
	\phi(\theta) = \phi(\theta_0) + (\theta-\theta_0)\phi'(\theta_0) +  \int_{\theta_0}^{\theta} (\theta - u) \phi''(u)du~.
\end{equation}

\noindent Let $s \in \S_0$. For all $t\in \RR$, we have

\begin{multline}\label{eqprop:taylorphirank}
\phi \left(  \frac{N\widehat{F}_{s,N}(t)}{N+1} \right) = \phi \circ \bF_s(t) + \left(\frac{N\widehat{F}_{s,N}(t)}{N+1}-\bF_s(t)\right)\phi' \circ \bF_s(t) \\
+\int_{\bF_s(t)}^{N\widehat{F}_{s,N}(t)/(N+1)} \left(\frac{N\widehat{F}_{s,N}(t)}{N+1}-u\right) \phi''(u)du~, 
\end{multline}
\noindent with probability one. Let $i \leq n$, for $t=s(\bX_i)$, \eqref{eqprop:taylorphirank} writes: 

\begin{multline}\label{eq:proptaylphi}
\phi \left(  \frac{N\widehat{F}_{s,N}(s(\bX_i))}{N+1} \right) = \phi \circ \bF_s(s(\bX_i)) \\
+ \left(\frac{N\widehat{F}_{s,N}(s(\bX_i))}{N+1}-\bF_s(s(\bX_i))\right)\phi' \circ \bF_s(s(\bX_i)) + t_i(s) \quad a.s.~,
\end{multline}
where $$\vert t_i(s) \vert\leq (\Vert \phi''\Vert_{\infty} /2)\left(N/(N+1) \widehat{F}_{s,N}(s(\bX_i))-\bF_s(s(\bX_i))\right)^2.$$
Hence, by summing over $i\in\{1,\; \ldots,\; n\}$, one gets that the approximation of $\widehat{W}_{n,m}(s)$ stated below holds true almost-surely:

\begin{equation}
\widehat{W}_{n,m}(s)= n \widehat{W}_{\phi}(s) + B_{n,m}(s) + \widehat{T}_{n,m}(s)~,
\end{equation}

\noindent where
\begin{eqnarray}\label{eqprop:defBT}
	B_{n,m}(s) &=& \sum_{i=1}^{n} \left(\frac{N\widehat{F}_{s,N}(s(\bX_i))}{N+1}-\bF_s(s(\bX_i))\right)\phi' \circ \bF_s(s(\bX_i)),\\
	\vert \widehat{T}_{n,m}(s) \vert &=&\sum_{i=1}^{n} \vert t_i(s) \vert  \le  \frac{\Vert \phi''\Vert_{\infty} }{2}\sum_{i=1}^{n} \left(\frac{N\widehat{F}_{s,N}(s(\bX_i))}{N+1}-\bF_s(s(\bX_i))\right)^2~. 
\end{eqnarray}

\paragraph{Linearization of $B_{n,m}(\cdot)$.}
First, observe that
\begin{multline}\label{eqB}
B_{n,m}(s)
=  \frac{1}{N+1} \sum_{i=1}^{n} \sum_{j\neq i}^n \mathbb{I} \{ s(\bX_j)\leq s(\bX_i) \} \phi' \circ \bF_s(s(\bX_i))  \\
+ \frac{1}{N+1} \sum_{i=1}^{n} \sum_{j=1}^{m}  \mathbb{I} \{ s(\bY_j)\leq s(\bX_i)  \} \phi' \circ \bF_s(s(\bX_i)) \\
+  \sum_{i=1}^n \left(\frac{1}{N+1}  - F_{s}(s(\bX_i))  \right)\phi' \circ \bF_s(s(\bX_i))~.
\end{multline}
Notice that the first two terms are $U$-processes indexed by $\S_0$, \textit{cf}  Section \ref{app:subUstatproc}, while the last term is an empirical process. Indeed, one may write

\begin{equation}\label{eq:UdecomB} 
B_{n,m}(s) = \frac{n(n-1) }{N+1} U_n(k_s)  + \frac{nm }{N+1} U_{n,m}(\ell_s)  + \widehat{K}_{n,m}(s)~,
\end{equation}
where
\begin{equation}\label{eq:uproconessample}
	U_n(k_s)  = \frac{1}{n(n-1)} \sum_{i=1}^{n} \sum_{j\neq i}^n \mathbb{I}\{s(\bX_j)\leq s(\bX_i)\}  \phi' \circ \bF_s(s(\bX_i))
\end{equation}
is a (nondegenerate) $1$-sample $U$-process of degree $2$ based on the random sample  $\{ \bX_1,\; \ldots,\; \bX_n  \}$ with nonsymmetric kernel $k_s (x,x') = \mathbb{I}\{s(x')\leq s(x)\} \phi' \circ \bF_s(s(x))$ on $\X \times \X$,
\begin{equation}\label{eq:uproctwossample}
U_{n,m}(\ell_s) = \frac{1}{nm} \sum_{i=1}^n \sum_{j=1}^m  \mathbb{I}\{s(\bY_j)\leq s(\bX_i)\}  \phi' \circ \bF_s(s(\bX_i))
\end{equation}
is a (nondegenerate) two-sample $U$-process of degree $(1,1)$ based on the samples $\{ \bX_1,\; \ldots,\; \bX_n  \}$ and $\{ \bY_1,\; \ldots,\; \bY_m  \}$ with kernel $\ell_s (x,y) =\mathbb{I}\{s(y)\leq s(x)\}  \phi' \circ \bF_s(s(x))$ on $\X \times \Y$, and 
\begin{equation*}
\widehat{K}_{n,m}(s) = \sum_{i=1}^n \left( \frac{1}{N+1}- F_{s}(s(\bX_i))  \right)\phi' \circ \bF_s(s(\bX_i))
\end{equation*}
is an empirical process based on the $\bX_i$'s.
In order to write $B_{n,m}$ as an empirical process plus a (negligible) remainder term, the Hoeffding decomposition is applied to the $U$-processes above, \textit{cf} Appendix \ref{app:subUstatproc}: 
\begin{eqnarray}
	U_n(k_s)  &=& \mathbb{E}[U_n(k_s)  ]+ \widehat{U}_n(k_s) +  \mathcal{R}_{n}(k_s)~, \label{eq:uprocdefhajek1}\\
	\label{eq:uprocdefhajek2}
	U_{n,m}(\ell_s) &=& \mathbb{E}[U_{n,m}(\ell_s) ] + \widehat{U}_{n,m}(\ell_s)  +  \mathcal{R}_{n,m}(\ell_s)~,
\end{eqnarray}
where
\begin{equation}\label{eq:udecomphaj1}
	\widehat{U}_n(k_s)   = \frac{1}{n} \sum_{i =1}^n k_{s,1,1}(\bX_i) + \frac{1}{n} \sum_{i =1}^n  k_{s,1,2}(\bX_i)~,
	\end{equation}
	with $k_{s,1,1} (x)=\mathbb{E}[k_s(x,\bX)]-\mathbb{E}[U_n(k_s) ]$ and $k_{s,1,2} (x)=\mathbb{E}[k_s(\bX,x)]-\mathbb{E}[U_n(k_s) ]$, and 
\begin{equation}\label{eq:udecomphaj2}
	\widehat{U}_{n,m}(\ell_s)   =\frac{1}{m}  \sum_{j =1}^m \ell_{s,1,1}(\bY_j) +\frac{1}{n}  \sum_{i =1}^n  \ell_{s,1,2}(\bX_i)~,
	\end{equation}
	with $\ell_{s,1,1} (y)=\mathbb{E}[\ell_s(\bX,y)]-\mathbb{E}[U_{n,m}(\ell_s) ] $ and $\ell_{s,1,2} (x)=\mathbb{E}[\ell_s(x,\bY)]-\mathbb{E}[U_{n,m}(\ell_s) ] $.


 

\noindent Consequently, the H\'ajek projection of the process $B_{n,m}(s)$ is given by

\begin{equation}\label{eq:Blinear}
\widehat{B}_{n,m}(s) - \mathbb{E}[\widehat{B}_{n,m}(s) ]  = \frac{n(n-1) }{N+1} \widehat{U}_n(k_s)  +  \frac{nm }{N+1}  \widehat{U}_{n,m}(\ell_s)  + \widehat{K}_{n,m}(s) - \mathbb{E}[\widehat{K}_{n,m}(s)]~.
\end{equation}



\noindent The following result provides an approximation of \eqref{eq:Blinear} and is proved in Appendix \ref{pflemma:HajekB}.

\begin{lemma}\label{lemma:HajekB}   
	Under Assumptions \ref{hyp:sabscont}-\ref{hyp:VC}, the H\'ajek projection of the stochastic process $B_{n,m}(\cdot)$, denoted by $\widehat{B}_{n,m}(\cdot)$ and indexed by $\S_0$, onto the subspace generated by the random variables $\bX_1,\; \ldots,\; \bX_n$ and $\bY_1,\; \ldots,\; \bY_m$ can be approximated as follows: for all $s \in \S_0$,
	
	\begin{equation}\label{eq:lemhajekBP}
	\widehat{B}_{n,m} (s) -  \mathbb{E} \left[ \widehat{B}_{n,m}(s)  \right]= \widehat{V}_{n}^X(s) +  \widehat{V}_{m}^Y(s)  + \widehat{R}_{n,m}(s)~,
	\end{equation}	
	
\noindent	where $$\widehat{V}_{n}^X(s) = \frac{n-1}{N+1} \sum_{i=1}^n  k_{s,1,1}  (\bX_i), \; \widehat{V}_{m}^Y(s)  = \frac{n}{N+1}\sum_{j=1}^m  \ell_{s,1,1}  (\bY_j)~.$$ 
	Let $\delta>0$, there exist constants $A_1 >0, \; A_2 \geq 2$ depending on $\phi$ and $\V$ such that for all $A_4 \geq A_1$ and $A_3 = (1/4A_{4}A_{2})\log(1+A_{4}/(4A_{2}))$
	
	\begin{equation}\label{lemeq:unifremhajekB} 
\mathbb{P}\left\{	\sup_{s \in \S_0} \bigg| \widehat{R}_{n,m}(s)   \bigg| > t \right\}\leq A_2 \exp\left\{ -\frac{A_3 N t^2}{ p \sigma^2} \right\}~,
	\end{equation}
as soon as $2A_1 \sigma  \sqrt{p\log(2 \| \phi'\|_{\infty}  /\sigma)/N} \leq t \leq 2pA_4  \| \phi'\|_{\infty}$, with $ \sigma^2 = \int_{[0,1]} \phi'^2$. 
	
\end{lemma}



The last step relies on all previous decompositions, so as to approximate $B_{n,m}(\cdot)$ by the sum of two empirical processes $\widehat{V}_{n}^X(\cdot) $ and $\widehat{V}_{n}^Y(\cdot)$, with a uniform control of the error. All residual terms, $\widehat{R}_{n,m}(s)$ (Lemma \ref{lemma:HajekB}) plus the remainders of the $U$-processes, are the components of the process $\mathcal{R}_{n,m}^B(s)$ , see the following Lemma \ref{lemma:hajekrem}.

\begin{lemma}\label{lemma:hajekrem} Suppose that Assumptions \ref{hyp:sabscont}-\ref{hyp:VC} are fulfilled. The stochastic process $B_{n,m} (.)$ can be approximated as follows: for all $s \in \S_0$,
	\begin{equation}\label{eq:decomplemmhajrem}
	B_{n,m} (s) -  \mathbb{E} \left[ B_{n,m}(s)  \right]= \widehat{V}_{n}^X(s) + \widehat{V}_{m}^Y(s) + \mathcal{R}_{n,m}^B(s)~.
	\end{equation}

\noindent 	Let $\delta >0$. There exist $D_1>0$ universal constant, and constants $D_3, \; D_4>0, \; D_2 \geq 1$, $d_1, \; d_2>3$ depending on $\phi$ and $\V$, such that with probability at least $1-\delta$:
	
	\begin{equation}\label{eqn:supBR} 
\sup_{s \in \S_0} \vert \mathcal{R}_{n,m}^B(s)  \vert \leq   \| \phi'\|_{\infty}  \sqrt{ p(1-p)   D_1\log(d_1/ \delta)}+ ( p \| \phi'\|_{\infty}D_4)  \log(d_2/\delta)~,
	\end{equation}
 as soon as $N\geq (pD_3)^{-1}\log(D_2/\delta)$.
\end{lemma}
\noindent Refer to Appendix \ref{pflemma:hajekrem} for the detailed proof. 

\paragraph{A uniform bound for $\widehat{T}_{n,m}(\cdot)$.} By virtue of \eqref{eqprop:defBT}, we have:

\begin{equation}\label{eq1}
	 \sup_{s\in \S_0} \vert \widehat{T}_{n,m}(s) \vert
	\leq  n\Vert \phi''\Vert_{\infty} \left( \sup_{(s,t) \in \S_0\times \mathbb{R} }  \left(\widehat{F}_{s,N}(t) -\bF_s(t) \right)^2 
	+  \frac{ 1}{(N+1)^2}\right).
\end{equation}

\noindent Observe also that
\begin{multline}\label{eq2}
\sup_{(s,t )\in \S_0 \times \RR } \vert \widehat{F}_{s,N}(t) -\bF_s(t)\vert \leq p \sup_{(s,t )\in \S_0 \times \RR } \vert \widehat{G}_{s,n}(t) -G_s(t)\vert  \\+(1-p)\sup_{(s,t )\in \S_0 \times \RR } \vert \widehat{H}_{s,m}(t) -H_s(t)\vert+ \frac{2}{N}~.
\end{multline}

\noindent A classic concentration bound for empirical processes based on the {\sc VC} inequality (see $\eg$ Theorems 3.2 and 3.4 in \cite{BBL05}) shows that, for any $\delta \in (0,1)$, we have with probability at least $1-\delta$:
\begin{equation*}
	\sup_{(s,t )\in \S_0 \times \RR } \vert \widehat{G}_{s,n}(t) -G_s(t)\vert \leq c\sqrt{\frac{\V}{n}}+\sqrt{\frac{2\log(1/\delta)}{n}}~, 
\end{equation*}
 where $c>0$ is a universal constant. In a similar fashion, we have, with probability larger than $1-\delta$,
 
 \begin{equation*}
 	\sup_{(s,t )\in \S_0 \times \RR } \vert \widehat{H}_{s,m}(t) -H_s(t)\vert \leq c\sqrt{\frac{\V}{m}}+\sqrt{\frac{2\log(1/\delta)}{m}}~.
 \end{equation*}


\noindent Combining the bounds above with the union bound, \eqref{eq2} and \eqref{eq1} we obtain that, for any $\delta\in (0,1)$, we have with probability larger than $1-\delta$:

\begin{multline}\label{eqproof:remTbound}
	\sup_{s\in \S_0} \vert \widehat{T}_{n,m}(s) \vert \leq n\vert\vert \phi''\vert\vert_{\infty}\left(  12\left( \frac{c^2\V+\log(2/\delta)}{N}+\frac{1}{N^2}\right)+\frac{1}{(N+1)^2}\right)\\
	\leq B_1+B_2\log(2/\delta)~,
\end{multline}
where $B_1$ (\resp $B_2$) is a constant that only depends on $\phi$ and $\V$ (\resp on $\phi$).







\noindent To end the proof, it suffices to observe that the remainder process is the sum of $ \mathcal{R}_{n,m}^B(s)$ and $\widehat{T}_{n,m}(s)$. Combining bounds \eqref{eqn:supBR} and \eqref{eqproof:remTbound}, we get that, with probability at least $1 - \delta$,

\begin{equation}
 \sup_{s\in \S_0} \vert  \mathcal{R}_{n,m}(s) \vert  = \sup_{s\in \S_0} \vert  \mathcal{R}_{n,m}^B(s) + \widehat{T}_{n,m}(s) \vert 
 \leq B_1 +  \| \phi'\|_{\infty} \kappa_p D  \log(2d/\delta) + B_2\log(4/\delta)
 \end{equation}

\noindent as soon as $N\geq (pD_3)^{-1} \log(D_2/\delta)$, with $D = \max(\sqrt{D_1},D_4)$, $d=\max(d_1,d_2)$, $\kappa_p = \max(\sqrt{p(1-p)},p)$. As $B_2>1$ , $d\geq3$,  and for small $\delta$, we obtain the upperbound $B_1 + ( \| \phi'\|_{\infty} \kappa_p D   + B_2)\log(2d/\delta)$. 




\subsection{Intermediary Results}\label{pf:auxlemm}



The intermediary results involved in Section \ref{pf:hajek} are now established.

\subsubsection{Permanence Properties}\label{pf:auxlemmperm}

The lemmas below claim that the collections of kernels/functions involved in the decomposition obtained in Appendix \ref{pf:hajek} are of \vc-type and uniformly bounded.

\begin{lemma}\label{lem:permkernel}
Suppose that Assumptions \ref{hyp:phic2} and \ref{hyp:VC} are fulfilled. Then, the collections of kernels $\{k_s(x,x'):\; s\in \S_0\}$ and $\{ \ell_s(x,y):\; s\in \S_0 \}$ and are bounded {\sc VC}-type classes of functions with parameters fully determined by $\V$ and $\phi$.

\end{lemma}

\begin{proof}
Recall that: $\forall (x,x')\in \X^2$,
$$
k_s(x,x')=\mathbb{I}\{s(x')\leq s(x)\}  \left(\phi' \circ \bF_s\right)(s(x)).
$$
Hence, we have $\sup_{(x,x')\in \X^2}\vert k_s(x,x')\vert \leq \vert\vert \phi' \vert\vert_{\infty}$ for all $s\in \S_0$.
In additions, since the collections $\{(x,x')\in \X^2\mapsto s(x):\; s\in \S_0\}$ and $\{(x,x')\in \X^2\mapsto s(x'):\; s\in \S_0\}$ are {\sc VC} classes of functions, classic permanence properties of {\sc VC} classes of functions (see $\eg$ Lemma 2.6.18) shows that $\{(x,x')\in \X^2\mapsto s(x)-s(x'):\; s\in \S_0\}$ is also a {\sc VC} class, as well as the class of indicator functions $\{(x,x')\in \X^2\mapsto \mathbb{I}\{s(x')\leq s(x)\} :\; s\in \S_0\}$. Consequently, the argument of Lemma \ref{lem:permdeg}'s proof permits to see easily that 
$\{(x,x')\in \X^2\mapsto \bF_s(s(x))=\mathbb{E}[\mathbb{I}\{s(X)\leq s(x)\}]:\; s\in \S_0\}$ is of {\sc VC} type, just like $\{(x,x')\in \X^2\mapsto (\phi' \circ \bF_s)(s(x)):\; s\in \S_0\}$ using the Lipschitz property of $\phi'$, \textit{cf} Assumption \ref{hyp:phic2}. Finally, being composed of products of a function in the bounded {\sc VC}-type class $\{(x,x')\in \X^2\mapsto \mathbb{I}\{s(x')\leq s(x)\} :\; s\in \S_0\}$ by a function in the bounded {\sc VC}-type class $\{(x,x')\in \X^2\mapsto (\phi' \circ \bF_s)(s(x)):\; s\in \S_0\}$, the collection $\{k_s:\; s\in \S_0\}$  is still a bounded {\sc VC}-type class of functions. A similar reasoning can be applied to show that $\{\ell_s:\; s\in \S_0\}$ is a bounded {\sc VC}-type class of kernels on $\X\times \Y$.
\end{proof}

\noindent The following result is straightforwardly deduced from the lemma above combined with Lemma \ref{lem:permdeg}.
\begin{lemma}\label{lem:permkerneldeg}
Suppose that Assumptions \ref{hyp:phic2} and \ref{hyp:VC} are fulfilled. Then, the collections of functions/kernels $\{k_{s,1,1}(x):\; s\in \S_0\}$, $\{k_{s,1,2}(x):\; s\in \S_0\}$, $\{k_{s}(x,x')-k_{s,1,1}(x)-k_{s,1,2}(x'):\; s\in \S_0\}$, $\{ \ell_{s,1,1}(y):\; s\in \S_0 \}$, $\{ \ell_{s,1,2}(x):\; s\in \S_0 \}$ and $\{ \ell_s(x,y)-\ell_{s,1,1}(y)-\ell_{s,1,2}(x):\; s\in \S_0 \}$ are bounded {\sc VC}-type classes with parameters fully determined by $\V$ and $\phi$.
\end{lemma}

\subsubsection{Proof of Lemma \ref{lemma:HajekB} }\label{pflemma:HajekB}

For $s\in \S_0$, by adding the diagonal term, the empirical process can be written
\begin{equation}\label{eq:lemmremhajdecom}
\widehat{R}_{n,m}(s) = 
\left(\frac{n}{N+1} -p\right)\sum_{i=1}^n  k_{s,1,2}(\bX_i) +  \left( \frac{m}{N+1}-(1-p)\right)\sum_{i=1}^n  \ell_{s,1,2}(\bX_i) ~.
\end{equation}

We uniformly bound all three empirical processes in probability using classic concentration bounds, see \textit{e.g.} Theorem 2.1 in \cite{GG02}, as follows.  Assuming Assumptions \ref{hyp:phic2}-\ref{hyp:VC}, Lemma \ref{lem:permkerneldeg} states that each class of functions $\{k_{s,1,2}: s\in \S_0 \}$, $\{\ell_{s,1,2}: s\in \S_0 \}$ is uniformly bounded and {\sc VC}-type of parameters depending only on $\phi$ and on the {\sc VC} dimension $\V$. For the class $\{x \mapsto \phi' \circ \bF_s(s(x)) : \; s \in \S_0  \}$, the arguments are exposed in the proof of Lemma \ref{lem:permkernel}. The variance of the kernels can be bounded for all $s\in \S_0$, by $ \sigma^2 = \int_{[0,1]} \phi'^2$ and $\sigma^2 \leq  \vert\vert \phi' \vert\vert_{\infty}^2$ and notice that $\vert n/(N+1) - p \vert \leq 1/N$ and  $\vert m/(N+1) - (1-p) \vert \leq 1/N$. 
Let $t>0$, there exist a sequence of constants $A_{1,i} >0, A_{2,i} \geq 1$ depending on $\phi$ and $\V$, $i\in \{1,2\}$, such that for all $A_{4,i} \geq A_{1,i}  $ and and $A_{3,i} =(1/A_{4,i}A_{2,i})\log(1+A_{4,i}/(4A_{2,i}))>0$, the following inequalities hold true.
\begin{equation}
\mathbb{P}\left\{	\frac{1}{N} \sup_{s \in \S_0} \bigg| \sum_{i=1}^n  k_{s,1,2}(\bX_i)   \bigg| > t \right\}\leq A_{2,1} \exp\left\{ -\frac{A_{3,1}N  t^2}{p  \sigma^2} \right\}~,
\end{equation}
as soon as $A_{1,1} \sigma  \sqrt{p\log(2 \| \phi'\|_{\infty}  /\sigma)/N} \leq t \leq p A_{4,1}  \| \phi'\|_{\infty}$,
\begin{equation}
\mathbb{P}\left\{	\frac{1}{N} \sup_{s \in \S_0} \bigg| \sum_{i=1}^n  \ell_{s,1,2}(\bX_i)   \bigg| > t \right\}\leq A_{2,2} \exp\left\{ -\frac{A_{3,2}N  t^2}{ p \sigma^2} \right\}~,
\end{equation}
\noindent as soon as $A_{1,2}  \sigma  \sqrt{p\log(2 \| \phi'\|_{\infty}  /\sigma)/N} \leq t \leq p A_{4,2}  \| \phi'\|_{\infty}$. 
The union bound with threshold $t/2$ yields

\begin{equation}
\mathbb{P}\left\{	\sup_{s \in \S_0} \bigg| \widehat{R}_{n,m}(s)   \bigg| > t \right\}\leq A_2 \exp\left\{ -\frac{A_3 N t^2}{ p \sigma^2} \right\}~,
\end{equation}


\noindent as soon as $2A_1 \sigma  \sqrt{p\log(2 \| \phi'\|_{\infty}  /\sigma)/N} \leq t \leq 2 p A_4 \| \phi'\|_{\infty}$
, with $A_1 = \max(A_{1,1},$\\$A_{1,2})$, $A_2 = 2\max(A_{2,1},A_{2,2}) $, $A_4 = \min(A_{4,1},A_{4,2})$ $\st$  $A_4 \geq A_1$,   $A_3 = (1/4A_{4}A_{2})\log(1+A_{4}/(4A_{2}))$ $\st$ $A_3 =(1/4)\min (A_{3,1}, A_{3,2})$,  at the price of changing $A_2$.

\subsubsection{Proof of Lemma \ref{lemma:hajekrem}}\label{pflemma:hajekrem}



The remainder of the decomposition \eqref{lemma:hajekrem} is obtained by combining Eq. \eqref{eq:UdecomB}, \eqref{eq:Blinear} and yields, for all $s\in \S_0$


\begin{equation*}
\left\vert \mathcal{R}_{n,m}^B(s) \right\vert   \leq   \vert  \widehat{R}_{n,m}(s)  \vert +  p^2N  \vert \mathcal{R}_{n}(k_s)\vert + p(1-p) N \vert  \mathcal{R}_{n,m}(\ell_s) \vert ~.
\end{equation*}

Suppose Assumptions \ref{hyp:phic2}-\ref{hyp:VC} are fulfilled. The first process can be uniformly bounded on $\S_0$ as proved in Lemma \ref{lemma:HajekB}. For the two others, we apply the results of Lemmas \ref{thm:major2006} and \ref{lem:devbound2} as follows. The process $ \mathcal{R}_{n}(k_s)$ (\resp $ \mathcal{R}_{n,m}(\ell_s)$) is the residual term obtained by decomposing the $U$-process $U_n(k_s)$ (Eq. \eqref{eq:uprocdefhajek1}, \resp \eqref{eq:uprocdefhajek2})), for all $s\in \S_0$. By Lemma  \ref{lem:permkerneldeg}, its class of degenerate kernels $\{(x,x') \mapsto k_{s}(x,x')-k_{s,1,1}(x)-k_{s,1,2}(x'):\; s\in \S_0\}$ (\resp $\{ (x,y)\mapsto \ell_s(x,y)-\ell_{s,1,1}(y)-\ell_{s,1,2}(x):\; s\in \S_0 \}$) is uniformly bounded and {\sc VC}-type of parameters depending only on $\phi$ and on the {\sc VC} dimension $\V$. 
Notice that the three classes of functions have variances and envelopes which can be similarly bounded by $ \sigma^2 = \int_{[0,1]} \phi'^2 \leq \vert\vert \phi' \vert\vert_{\infty}^2$, up to a multiplicative constant for both residuals. 
Let $\delta>0$, there exist constants $A_1, B_1 >0, A_2, B_2 \geq 1, A_3, B_3> 0 $ depending on $\phi$ and $\V$ \st with probability at least $1-\delta$
\begin{equation}\label{eq:lem18hajek}
		\sup_{s \in \S_0} \bigg| \widehat{R}_{n,m}(s)   \bigg|  \leq  \| \phi'\|_{\infty}\sqrt{ \frac{p\log(A_2/\delta)}{A_3 N }}~,
\end{equation}

\noindent as soon as  $ N \geq (4pA_3A_4^2 )^{-1}\log(A_2/\delta) $. Also by Lemma \ref{thm:major2006}
\begin{equation}\label{eq:lem18onesample}
p^2 N \sup_{s \in \S_0} \vert \mathcal{R}_{n}(k_s) \vert \leq( p \| \phi'\|_{\infty}/B_3)  \log(B_2/\delta)~,
\end{equation}

\noindent when $N\geq (pB_3)^{-1} \log(B_2/\delta)   $.  
And, by Lemma \ref{lem:devbound2}, there exist constants $C_1>0$, $C_2>1$ depending on $\V, \; \phi$  and a universal constant $ C_3>0$ such that
\begin{equation}\label{eq:lem18twosample}
p(1-p) N \sup_{s \in \S_0}  \vert \mathcal{R}_{n,m}(\ell_s) \vert \leq  \| \phi'\|_{\infty}  \sqrt{ p(1-p)   C_3\log(C_2/\delta)}~,
\end{equation}

\noindent for  
$\log(C_2/\delta) \geq C_1(\| \phi'\|_{\infty}^2C_3)^{-1}$. 
 The union bound concludes by considering constants such that with probability at least $1-\delta$

\begin{equation}
  \sup_{s \in \S_0} \; \vert  \mathcal{R}_{n,m}^B(s) \vert \leq   \| \phi'\|_{\infty}  \sqrt{ p(1-p)   C_3\log(3C_2/\delta)}+ ( p \| \phi'\|_{\infty}/B_3)  \log(3B_2/\delta)~,
\end{equation}

 \noindent  
 as soon as $N\geq (pD_3)^{-1} \log(D_2/\delta)$, where $D_2 = 3 \max(A_2,B_2)$ and $D_3 = \min(4A_3A_4^2, B_3)$. 

\subsection{ Proof of Theorem \ref{thm:bound_twosample}}\label{pf:bound_twosample}

Observe, by virtue of Proposition \ref{prop:hajek} and for all $s \in \S_0$

\begin{multline*}
	 	\left\vert \frac{1}{n}  \widehat{W}^{\phi}_{n,m}(s)- W_{\phi}(s)\right\vert  
	 \leq  \frac{1}{n} \bigg|  \sum_{i=1}^{n} \phi\circ \bF_s(s(\bX_i)) -  \mathbb{E} [\phi\circ \bF_s(s(\bX))]  \bigg| \\
	  + \frac{1}{N} \bigg|  \sum_{i=1}^{n} k_{s,1,1}(\bX_i) \bigg|  + \frac{1}{N} \bigg|  \sum_{j=1}^{m} \ell_{s,1,1}(\bY_j) \bigg|  + \frac{1}{n} \bigg|   \mathcal{R}_{n,m}(s) \bigg| ~.
\end{multline*}
Under Assumptions \ref{hyp:phic2}-\ref{hyp:VC}, we sequentially provide uniform bounds in probability for all processes. The classes of kernels $\{x \mapsto k_{s,1,1}(x):\; s\in \S_0\}$ and $\{y \mapsto \ell_{s,1,1}(y):\; s\in \S_0 \}$, by Lemma \ref{lem:permkerneldeg}, are bounded and  {\sc VC}-type of parameters depending on $\phi$ and on the  {\sc VC} dimension $\V$ of $\S_0$. Their variance can be  bounded, for all $s\in \S_0$, by $ \sigma^2 = \int_{[0,1]} \phi'^2$ and $\sigma^2 \leq  \vert\vert \phi' \vert\vert_{\infty}^2$.  As well for the collection $\{x\mapsto \phi\circ \bF_s(s(x)): s \in \S_0   \}$ where the arguments are detailed in Lemma \ref{lem:permkernel} and of variance bounded by, for all $s\in \S_0$, by $ \Sigma^2 = \int_{[0,1]} \phi^2$ and $\Sigma^2 \leq  \vert\vert \phi \vert\vert_{\infty}^2$.   Similarly to Lemma \ref{lemma:HajekB}, we apply Theorem 2.1 in \cite{GG02} to the empirical processes $ \widehat{W}_{\phi}(s)$,  $\widehat{V}_{n}^X(s)$ and $\widehat{V}_{m}^Y(s) $ as follows.  

\noindent Let $ t>0$. There exist a sequence of constants $C_{1,i} >0, C_{2,i} \geq 1, $ depending on $\phi$ and $\V$, such that for all  $C_{4,i}  \geq C_{1,i} $ and $C_{3,i} =(1/C_{4,i})\log(1+C_{4,i}/(4C_{2,i}))$,  $i\in \{1,2,3\}$, the following inequalities hold true. 

\begin{equation}
\mathbb{P}\left\{ \sup_{s \in \S_0} \bigg|  \widehat{W}_{\phi}(s)- W_{\phi}(s)   \bigg| > t \right\}\leq C_{2,1} \exp\left\{ -\frac{C_{3,1} p N t^2}{C_{2,1} \Sigma^2} \right\}~,
\end{equation}
as soon as $C_{1,1} \| \phi\|_{\infty}   \sqrt{(1/pN) \log(2 \| \phi\|_{\infty}  /\Sigma)} \leq t \leq C_{4,1} \| \phi\|_{\infty} $.
\begin{equation}
\mathbb{P}\left\{	\frac{1}{N} \sup_{s \in \S_0} \bigg| \sum_{i=1}^n  k_{s,1,1}(\bX_i)   \bigg| > t \right\}\leq C_{2,2} \exp\left\{ -\frac{C_{3,2} N t^2}{ p \sigma^2  C_{2,2} } \right\}~,
\end{equation}
as soon as $C_{1,2} \vert\vert \phi' \vert\vert_{\infty} \sqrt{(p /N) \log(2 \| \phi'\|_{\infty}  /\sigma)} \leq t \leq  pC_{4,2} \vert\vert \phi' \vert\vert_{\infty}$.

\begin{equation}
\mathbb{P}\left\{	\frac{1}{N} \sup_{s \in \S_0} \bigg| \sum_{j=1}^{m} \ell_{s,1,1}(\bY_j) \bigg| > t \right\}\leq C_{2,3} \exp\left\{ -\frac{C_{3,3}  N t^2}{ (1-p)  \sigma^2C_{2,3} } \right\}~,
\end{equation}
as soon as $C_{1,3} \vert\vert \phi' \vert\vert_{\infty} \sqrt{((1-p)/N)\log(2 \| \phi'\|_{\infty}  /\sigma)} \leq t \leq  (1-p) C_{4,3}\vert\vert \phi' \vert\vert_{\infty}$. Proposition \ref{prop:hajek} provides the existence of constants $C>6, \; D>0$ and $c_3>0, \;c_5>3 $ depending on $\phi$ and $\V$, such that

\begin{equation}\label{eq:remquadthm}
\mathbb{P}\left\{ \frac{1}{n}	\sup_{s \in \S_0} \vert \mathcal{R}_{n,m}(s)  \vert > t \right\} \leq C \exp\left\{-\frac{pNt}{( \| \phi'\|_{\infty} \kappa_p D   + B_2)}  \right\}~,
\end{equation}
as soon as $N\geq (c_3/p)\log(c_5/\delta)$. The remainder process is negligible with respect to the empirical processes and  we gather the four bounds to get


\begin{equation}
\mathbb{P}\left\{ \sup_{s \in \S_0} \bigg| \frac{1}{n}  \widehat{W}^{\phi}_{n,m}(s)- W_{\phi}(s)  \bigg| > t \right\}
\leq C_2e^{ -p (C_3/C_2) N t^2}  ~, 
\end{equation}




\medskip
\noindent where one can choose  as soon as \eqref{eq:remquadthm} is satisfied, $C_2 = 4\max(\{ C_{2,i}, i \leq3\}, C)$, \\ $C_4 = 4 \min (C_{4,1}\| \phi\|_{\infty} , C_{4,2}p\| \phi'\|_{\infty},C_{4,3}(1-p) \| \phi'\|_{\infty})$ such that  $C_4\geq \max(C_{1,i},\; i\leq 3)=: C_1$  and $C_3$ can be chosen as $C_3 = \log(1+C_{4}/(4C_2))/(C_{4} \max(\Sigma^2, \sigma^2))$ as  $1/p,  1/(1-p) \geq p$ and at the price of changing $C_2$.

\subsection{A Generalization Bound in Expectation}\label{app:expectation}
For the sake of completeness, we state and prove a version in expectation of the generalization result formulated in Corollary \ref{cor:bound_twosample}.

\begin{proposition}\label{prop:Ebound_twosample}	Under the assumptions of Proposition \ref{prop:hajek}, the expected risk bound is derived as follows:
	
	\begin{equation}
		\mathbb{E} \left[ W^*_{\phi}-W_{\phi}(\hat{s}) \right] \leq B_1\sqrt{ \frac{\V }{pN}}  + W^*_{\phi}- \mathbb{E}\left[  \sup_{s\in \mathcal{S}_0}W_{\phi}(s)\right]~, 
	\end{equation}
	for $pN \geq B_2 \V $ with constants $B_1, B_2>0$ depending on $\phi, \; \V$.
	
\end{proposition}

\begin{proof} Following the decomposition \eqref{eq:max_dev}, we bound in expectation each process recalling that they are indexed by uniformly bounded \vc-type classes, refer to Proof \ref{pf:bound_twosample} for the details on theoretical guarantees concerning the permanence properties. For the empirical processes $ \widehat{W}_{\phi}$,  $\widehat{V}_{n}^X$ and $\widehat{V}_{m}^Y$, we use Theorem 2.1 in \cite{GG02}, whereas for the remainder process, we require the following result that is proved subsequently.

\begin{lemma}\label{lemm:expboundrem} Under the assumptions of Proposition \ref{prop:hajek}, the remainder process can be uniformly bounded in expectation as follows:
	
	\begin{equation}
		\mathbb{E} \left[ \sup_{s\in \S_0} \left\vert \mathcal{R}_{n,m}(s) \right\vert \right] \leq  D_1 (1 + 1/p + 1/ \sqrt{p(1-p)})~,
	\end{equation}
	for $pN \geq D_2\V $ with constants $D_1>0$ depending on $\phi, \; \V$ and $D_2>0$ on $\phi$.

\end{lemma}

By means of \cite{GG02}, there exist universal constants $B_i > 0$, and $b_i >0 , \; i \in\{1,\; 2,\;  3\}$, depending on $\phi, \; \V$ such that the inequalities below hold true.

\begin{equation}
 \mathbb{E}\left[\sup_{s \in \S_0} \bigg|   \widehat{W}_{\phi}(s)- W_{\phi}(s)   \bigg|  \right]
\leq   B_1 \left(  b_1 \frac{\V \| \phi\|_{\infty}}{ pN } +\| \phi\|_{\infty} \sqrt{ b_1\frac{ \V }{pN}}    \right)~,
\end{equation}
and
\begin{equation}
\mathbb{E}\left[\frac 1 n \sup_{s \in \S_0} \left\vert \widehat{V}_{n}^X(s) -\mathbb{E}\left[ \widehat{V}_{n}^X(s) \right] \right\vert  \right]
\leq B_2 \left(b_2\frac{ \V \| \phi' \|_{\infty} }{pN} +\| \phi' \|_{\infty}\sqrt{ b_2 \frac{\V }{pN}}    \right)~,
\end{equation}
as well as
\begin{equation}
\mathbb{E}\left[ \frac 1 n  \sup_{s \in \S_0}\left\vert \widehat{V}_{m}^Y(s) -\mathbb{E}\left[ \widehat{V}_{m}^Y(s) \right] \right\vert \right]
\leq B_3 \left( b_3 \frac{\V \| \phi' \|_{\infty} }{pN} + \| \phi' \|_{\infty}\sqrt{b_3 \frac{\V }{pN}}    \right) ~,
\end{equation}

observing that $ \int_{[0,1]} \phi^2 \leq \| \phi\|_{\infty}^2 $ and $ \int_{[0,1]} \phi'^2 \leq \vert\vert \phi' \vert\vert_{\infty}^2$.

The remainder process being of higher order, we conclude

\begin{equation}
\mathbb{E}\left[  \sup_{s \in \S_0}\left\vert \frac{1}{n} \widehat{W}_{n,m}^{\phi}(s)-W_{\phi}(s) \right\vert \right]
\leq B\sqrt{b  \frac{\V }{pN}}   ~,
\end{equation}

for $pN \geq \max( b, D_2) \V $ with constants $B>0$ depending on $\phi$ and $b>0$ depending on $\phi, \; \V$. 

\end{proof}

\bigskip

\begin{proof} For all $s\in \S_0$
	\begin{equation}
	 \vert  \mathcal{R}_{n,m}(s)  \vert \leq   \vert  \widehat{R}_{n,m}(s)  \vert +  N  \vert \mathcal{R}_{n}(k_s)\vert + N \vert  \mathcal{R}_{n,m}(\ell_s) \vert + \vert  \widehat{T}_{n,m}(s) \vert
 	\end{equation}
 	
 The process appearing first in the remainder induced by the H\'ajek projection method (Lemma \ref{lemma:HajekB}), is composed of sums of empirical processes, hence applying Theorem 2.1 in \cite{GG02} to each process of \eqref{eq:lemmremhajdecom} yields

 \begin{equation}
 \mathbb{E}\left[ \sup_{s \in \S_0}\left\vert \widehat{R}_{n,m}(s)   \right\vert \right]
 \leq D_1 \left(d \frac{\V \| \phi' \|_{\infty} }{N} + \| \phi' \|_{\infty}\sqrt{d  \frac{p\V  }{N}}    \right) ~,
 \end{equation}
	with constants $D_1>0$ depending on $\phi$ and $d>0$ on $\phi, \; \V$.
	 The stochastic processes $\mathcal{R}_{n}(k_s)$ and $\mathcal{R}_{n,m}(\ell_s)$ being both degenerate $U$-processes, respectively one-sample of degree $2$ and two-sample of degree  $(1,1)$, we apply results in \cite{NoPo87} (see Theorem $6$ therein) and \cite{Neum2004} (see Lemma $2.4$ therein)  so as to get

\begin{equation}
\mathbb{E}\left[   \sup_{s \in \S_0}\left\vert \mathcal{R}_{n}(k_s) \right\vert   \right] \leq \frac{D_2\V}{pN}     ~,
\end{equation}	 
and
\begin{equation}
\mathbb{E}\left[   \sup_{s \in \S_0}\left\vert  \mathcal{R}_{n,m}(\ell_s)  \right\vert   \right] \leq  \frac{D_3\V}{\sqrt{p(1-p)}N}   ~,
\end{equation}	
	$D_2, \; D_3>0$ constants of $\phi, \V$. For $\widehat{T}_{n,m}(s)$, the concentration inequality proved in Eq. \eqref{eqproof:remTbound} holds true for all $\delta\in (0,1)$. Hence, we have
	
\begin{multline}
\mathbb{E}\left[   \sup_{s \in \S_0}\left\vert \widehat{T}_{n,m}(s)  \right\vert   \right] 
 \leq u + \int_{u}^{\infty} \mathbb{P} \left\{    \underset{s\in \mathcal{S}_0}{\sup}\left\vert \widehat{T}_{n,m}(s) \right\vert  \geq x \right\} dx\\
 = u + 2B_2e^{-(u-B_1)/B_2}~.
\end{multline}
Minimizing the bound above \wrt $u>0$, we obtain the point $B_1 + B_2\log(2)$ and the upperbound then writes $B_1 + B_2(1+\log(2))$, where $B_1$ (\resp $B_2$) is a constant that only depends on $\phi$ and $\V$ (\resp on $\phi$). Combining all bounds together permits to conclude: for $N \geq \V \log(d)$, we have
	\begin{multline}
\mathbb{E}\left[   \sup_{s \in \S_0}\left\vert   \mathcal{R}_{n,m}(s)   \right\vert   \right]      \leq D_1 \| \phi' \|_{\infty}+ \frac{D_2\V}{p } +     \frac{D_3\V}{\sqrt{p(1-p)}} +  B_1 + B_2(1+\log(2))\\
\leq D (1 + 1/p + 1/ \sqrt{p(1-p)})~,
\end{multline}
where $D>0$ constant depending on $\phi, \; \V$. $\square$
\end{proof}

\subsection{Proof of Proposition \ref{cor:oracle}}\label{pf:oracle}

We first prove the following lemma.

\begin{lemma}\label{lemm:devmeanbound}Let $\S_0 \subset \S$ and suppose that Assumptions \ref{hyp:sabscont}-\ref{hyp:VC} are fulfilled. For all $t>0$, we have:
	
	\begin{multline}
		\mathbb{P}\left\{ \sup_{s \in \S_0} \left\vert  W_{\phi}(s) -  \widehat{W}^{\phi}_{n,m}(s)/n \right\vert \geq \mathbb{E}\left[\sup_{s \in \S_0} \left\vert W_{\phi}(s) -  \widehat{W}^{\phi}_{n,m}(s)/n \right\vert\right] + t \right\}\\
		\leq  \exp\left\{ -\frac{p^2 Nt^2}{ 6(\Vert \phi \Vert_{\infty}^2 + 9\Vert \phi' \Vert_{\infty}^2+ 9\vert\vert \phi''\vert\vert_{\infty}^2) } \right\}~.
	\end{multline}
	
\end{lemma}

\begin{proof} 
Recall the decomposition of $ \widehat{W}^{\phi}_{n,m}(s)  $, for all $s\in \S_0$, proved in Proposition \ref{prop:hajek}
	
\begin{equation}\label{eq:dec_3}
\widehat{W}_{n,m}(s)= n \widehat{W}_{\phi}(s) + B_{n,m}(s) + \widehat{T}_{n,m}(s)~.
\end{equation}

Considering that $\sup_{s \in \S_0} \left\vert  W_{\phi}(s) -  \widehat{W}^{\phi}_{n,m}(s)/n \right\vert$ is a function of the $N$ independent random variables $\bX_1,\; \ldots,\; \bX_n,\; \bY_1,\; \ldots,\; \bY_m$, observe that changing the value of any of the $\bX_i$'s while keeping all the others fixed changes the value of the supremum by at most

\begin{equation*}
	2\vert\vert \phi\vert\vert_{\infty} + 2\vert\vert \phi'\vert\vert_{\infty}  \left(1+\frac{m+2(n-1)}{N+1}  \right) + 2\vert\vert \phi''\vert\vert_{\infty} \frac{1+2m}{N^2}~,
\end{equation*}
taking into account the jumps of each of the three terms involved in \eqref{eq:dec_3}, see Eq. \eqref{eqB} and \eqref{eq1}.
In a similar way,  changing the value of any of the $\bY_j$'s changes the value of the supremum by at most

\begin{equation*}
	2\vert\vert \phi'\vert\vert_{\infty} \frac{n}{N+1}  + 2\vert\vert \phi''\vert\vert_{\infty}\frac{1+2n}{N^2}~. 
\end{equation*}
When taking the squares, both can be upperbounded by $12( \Vert \phi \Vert_{\infty}^2 + 9\Vert \phi' \Vert_{\infty}^2+ 9\vert\vert \phi''\vert\vert_{\infty}^2) $.
The desired bound stated then straightforwardly results from the application of the bounded difference inequality, see \cite{McD89}. $\square$
\end{proof}
\medskip


Let $\varepsilon >0$, using Proposition \ref{prop:Ebound_twosample} and Lemma \ref{lemm:devmeanbound}, we have, for any $k\geq 1$,

\begin{multline}\label{eq:ms_1}
	\mathbb{P}  \left\{ \widehat{W}^{\phi}_{n,m}(\hat{s}_{k})-B_1\sqrt{\frac{\V_k}{pN}}  - W_{\phi}(\hat{s}_k) >  \varepsilon  \right\}  
\\	\leq	\mathbb{P}\left\{ \sup_{s \in \S_k} \left\vert W_{\phi}(s) -  \widehat{W}^{\phi}_{n,m}(s)/n \right\vert > \mathbb{E}\left[\sup_{s \in \S_k} \left\vert W_{\phi}(s) -  \widehat{W}^{\phi}_{n,m}(s)/n \right\vert\right]+ \varepsilon  \right\} \\
	\leq \exp\left\{-\frac{p^2N\varepsilon^2 }{ C }  \right\}~,
\end{multline}
as soon as $pN \geq B_2\V_k $ and where $C = 6( \Vert \phi \Vert_{\infty}^2 + 9\Vert \phi' \Vert_{\infty}^2+ 9\vert\vert \phi''\vert\vert_{\infty}^2)$. For each $k\geq 1$, denote the penalized empirical ranking performance measure by
\begin{equation}
\widehat{W}_{n,m}^{\phi, k}(\hat{s}_{k})/n=\widehat{W}_{n,m}^{\phi}(\hat{s}_{k})/n-B_1\sqrt{\frac{\V_k}{pN}}-\sqrt{\frac{2C\log k}{p^2N}}.
\end{equation}
For any $\varepsilon>0$, we have, as soon as $pN\geq B_2 \sup_{k\geq 1}\V_k$,
\begin{multline}\label{eq:ms_2}
\mathbb{P}  \left\{ \widehat{W}_{n,m}^{\phi, \hat{k}}(\hat{s}_{\hat{k}})/n  - W_{\phi}(\hat{s}_{\hat{k}}) \geq  \varepsilon  \right\} \leq \sum_{k\geq 1} \mathbb{P}  \left\{ \widehat{W}_{n,m}^{\phi, k}(\hat{s}_{k})/n  - W_{\phi}(\hat{s}_{k}) \geq  \varepsilon  \right\}\\
\leq \sum_{k\geq 1} 	\mathbb{P}  \left\{ \widehat{W}^{\phi}_{n,m}(\hat{s}_{k})/n-B_1\sqrt{\frac{\V_k}{pN}}  - W_{\phi}(\hat{s}_k) >  \varepsilon +\sqrt{\frac{2C\log k}{p^2N}} \right\}  \\
\leq \sum_{k\geq 1} \exp\left(-\frac{p^2N}{ C} \left( \varepsilon +\sqrt{\frac{2C\log k}{p^2N}}  \right)^2 \right) \\
\leq   \exp\left(-\frac{p^2N  \varepsilon ^2}{ C }  \right)\sum_{k\geq 1} k^{-2}<2 \exp\left\{-\frac{ p^2N  \varepsilon ^2}{ C }  \right\}~.
\end{multline}

For all $k\geq 1$, $W^*_k = \sup_{s \in \S_k} W_{\phi}(s) = W_{\phi}(s^*_k)$ and consider the decomposition

\begin{equation*}
	W_k^*  - W_{\phi}(\hat{s}_{\hat{k}}) =
	\left( W_k^*  -\widehat{W}_{n,m}^{\phi, \hat{k}}(\hat{s}_{\hat{k}})/n \right) + \left(   \widehat{W}_{n,m}^{\phi, \hat{k}}(\hat{s}_{\hat{k}})/n -  W_{\phi}(\hat{s}_{\hat{k}}) \right)~.
\end{equation*}

The expectation of the second term of the right hand side of the equation above can be bounded by means of the tail bound \eqref{eq:ms_2}

\begin{equation}\label{eq:modselexp2}
	\mathbb{E}\left[ \widehat{W}_{n,m}^{\phi, \hat{k}}(\hat{s}_{\hat{k}})/n -  W_{\phi}(\hat{s}_{\hat{k}})  \right] 
	\leq 2\sqrt{\frac{C}{p^2N}}~.
\end{equation}
for any $k\geq 1$, as soon as  $pN \geq B_2\sup_{k\geq 1}\V_{k}$. Concerning the expectation of the first term, observe that
\begin{multline*}
\mathbb{E}\left[ W_k^*  -\widehat{W}_{n,m}^{\phi, \hat{k}}(\hat{s}_{\hat{k}})/n  \right] \leq \mathbb{E}\left[ W_k^* -\widehat{W}_{n,m}^{\phi, k}(s^*_{k})  \right]\\
 \leq \mathbb{E}\left[ W_{\phi}(s^*_k)   -\widehat{W}^{\phi}_{n,m}(s^*_{k})  \right] +\text{pen}(N,k)
\leq B_1\sqrt{\frac{\V_k}{pN}} + \text{pen}(N,k)~,
\end{multline*}
for any $k\geq 1$, as soon as  $pN \geq B_2\sup_{k\geq 1}\V_{k}$.
Summing the bound obtained and that in \eqref{eq:modselexp2} gives the desired result.

\subsection{Proof of Proposition \ref{prop:erm_bound2}}\label{pf:erm_bound2}
The proof consists in combining the two results stated below with the decomposition \eqref{eq:decomp_smooth} of the $W_{\phi}$-ranking performance deficit of the maximizer. The first result is the analogue of Theorem \ref{thm:bound_twosample} for the smoothed criterion.

\begin{theorem}\label{thm:bound_twosample2}
Suppose that the assumptions of Proposition \ref{prop:hajek} are fulfilled. Then, for any $\delta \in (0,1)$, there exist constants $C_1, \; C_3 >0, \; C_2 \geq 24$, depending on $\phi, \; K, \; R, \; \V$ 
such that with probability larger than $1-\delta$:

\begin{equation}
\sup_{s\in \mathcal{S}_0}\left\vert \widehat{W}^{\phi}_{n,m, h}(s)/n- \widetilde{W}_{\phi, h}(s)\right\vert 
\leq \sqrt{\frac{\log(C_2/\delta)}{pC_3N}} ~,
\end{equation}	
as soon as $N\geq 1/(p \min(p,1-p)^2C_3C_4^2)  \log(C_2/\delta)$ and $\delta \leq C_2e^{-C_1^2C_3}$, with $C_3 = \log(1+C_{4}/(4C_{2}))/(C_2C_4)$. 
\end{theorem} 
The proof being quite similar to that of Theorem \ref{thm:bound_twosample}, it is omitted. Assumption \ref{hyp:kernelreg} ensuring that the class $\{K((\cdot - t)/h);, \;  t \in \RR^q, \; h >0  \}$ ($q=1$ here) is bounded {\sc VC}-type (see \textit{e.g.} Lemma 22(ii) in \cite{NoPo87} and \cite{GKZ04}), classic permanence properties can be used to check that all the classes of functions over which uniform bounds are taken are of finite {\sc VC} dimension. The second result provides a uniform bound for the additional bias error made when approximating $W_{\phi}(s)$ by $\widetilde{W}_{\phi, h}(s)$ for $s\in \S_0$.

\begin{lemma}
Suppose that Assumptions \ref{hyp:approx} is satisfied. Then, for all $h>0$, we have:
\begin{equation}
 \sup_{s\in \S_0}\left\vert \widetilde{W}_{\phi,h}(s) - W_{\phi}(s) \right\vert \leq C_5 h^2,
\end{equation}
where $C_5>0$ is a constant depending on $\phi$, $K$ and $R$ only.
\end{lemma}

\noindent Details are left to the reader, the proof is straightforward under Assumption \ref{hyp:approx}, using the regularity of the score generating function and the uniform integrated error bound obtained in \cite{Jones90}.

\subsection{Proof of Lemma \ref{lem:devbound2}}\label{pf:2s_Uproc}
We shall prove an exponential bound of Hoeffding's type for the uniformly bounded two-sample degenerate $U$-process $\{U_{n,m}(\ell):\; \ell \in \mathcal{L}\}$, where 
\begin{equation}\label{eq:two_sample_stat}
U_{n,m}(\ell) = \frac{1}{nm} \sum_{i=1}^n\sum_{j=1}^m \ell(X_i,Y_j)~.
\end{equation}
In order to apply standard symmetrization arguments, see $\eg$ section 2.3 in \cite{vdVWell96}, consider independent Rademacher variables $\eps_1, \ldots, \eps_n$ and $\eta_1, \ldots, \eta_m$ and define
\begin{equation}\label{eq:two_sample_stat_rando}
T_{n,m}(\ell) = \frac{1}{nm}\sum_{i=1}^n\sum_{j=1}^m\eps_i \eta_j  \ell(X_i, Y_j)~,
\end{equation}
for all $\ell$ in $\mathcal{L}$. 
We start by proving the following lemmas, involved in the argument.
\begin{lemma}\label{lem:random}Let $P$ and $Q$ be probability distributions on measurable spaces $\X$ and $\Y$ respectively. Consider the degenerate two-sample $U$-statistic of degree $(1,1)$ \eqref{eq:two_sample_stat} with a bounded kernel $\ell:\X \times \Y \rightarrow \mathbb{R}$ based on the independent $\iid$ random samples $X_1,\;  \ldots,\;  X_n $ and  $Y_1,\;  \ldots,\; Y_m$, drawn from $P$ and $Q$ respectively. Let two sequences of $\iid$ Rademacher variables $\eps_1, \ldots, \eps_n$ and $\eta_1, \ldots, \eta_m$, independent of the $X_i$'s and $Y_j$'s, such that the randomized process \eqref{eq:two_sample_stat_rando} is defined.
Then, for any increasing and convex function $\Phi:\mathbb{R}\rightarrow \mathbb{R}$, we have:
\begin{equation}\label{eq:symexpsupconv}
	\mathbb{E}\left[\Phi\left(\sup_{\ell\in \L}  \vert U_{n,m}(\ell) \vert \right) \right] \leq \mathbb{E}\left[\Phi\left(4 \sup_{\ell\in \L}  \vert T_{n,m}(\ell) \vert \right) \right],
	\end{equation}
	and
	\begin{equation}\label{eq:symexpsupconv2}
	\mathbb{E}\left[\Phi\left(\sup_{\ell\in \L}  U_{n,m}(\ell)  \right) \right] \leq \mathbb{E}\left[\Phi\left(4 \sup_{\ell\in \L} \ T_{n,m}(\ell) \right) \right],
	\end{equation}
	assuming that the suprema are measurable and that the expectations exist.
\end{lemma}
\begin{proof} We prove the first inequality, the proof of the second one being similar. Using the independence of the two samples, Fubini's theorem and the degeneracy property, one gets that
\begin{eqnarray*}
		&& \E\left[\Phi\left(\sup_{\ell\in \L} \vert U_{n,m}(\ell)\vert \right)  \right]  \\
		&=& \E \left[\E\left[ \Phi\left(\sup_{\ell\in \L} \left\vert  \frac{1}{nm} \sum_{i=1}^n \left(\sum_{j=1}^m \ell(X_i, Y_j)\right)\right\vert  \right)\mid Y_1,\; \ldots,\; Y_m\right]  \right]   \\
		&\leq& \E\left[ \Phi\left(2\sup_{\ell\in \L}  \left\vert \frac{1}{nm} \sum_{i=1}^n\eps_i \left(\sum_{j=1}^m \ell(X_i, Y_j)\right)\right) \right\vert \right] \\
		&=& \E\left[ \E\left[ \Phi\left(2 \sup_{\ell\in \L} \left\vert \frac{1}{nm} \sum_{j=1}^m\left(\sum_{i=1}^n \eps_i  \ell(X_i, Y_j)\right)\right\vert  \right)\mid (X_1,\eps_1),\; \ldots,\; (X_n,\eps_n)\right] \right] \\
		&\leq&  \E \left[ \Phi\left(4\sup_{\ell\in \L}  \left\vert \frac{1}{nm} \sum_{j=1}^m \eta_j \left(\sum_{i=1}^n \eps_i  \ell(X_i, Y_j)\right)\right\vert  \right) \right] \\
		&=& \E \left[\Phi\left(4 \sup_{\ell\in \L} \vert T_{n,m}(\ell)\vert \right)  \right] 
	\end{eqnarray*}
by applying Lemma 3.5.2 of \cite{PeGin99} twice. Incidentally, notice that we can also show that
$$
\E \left[\Phi\left(\frac{1}{4} \sup_{\ell\in \L} \vert T_{n,m}(\ell)\vert \right)  \right] \leq \E\left[\Phi\left(\sup_{\ell\in \L} \vert U_{n,m}(\ell)\vert \right)  \right].
$$
by applying twice the reverse inequality in Lemma 3.5.2 of \cite{PeGin99}.
$\square$
\end{proof}
\medskip
Next, we prove an exponential bound of Hoeffding's type for degenerate two-sample $U$-statistics with bounded kernels.

\begin{lemma} \label{lemma:hoefS} Let $P$ and $Q$ be probability distributions on measurable spaces $\X$ and $\Y$ respectively. Consider the degenerate two-sample $U$-statistic of degree $(1,1)$ \eqref{eq:two_sample_stat} with a bounded kernel $\ell:\X \times \Y \rightarrow \mathbb{R}$ based on the independent $\iid$ random samples $X_1,\;  \ldots,\;  X_n $ and  $Y_1,\;  \ldots,\; Y_m$, drawn from $P$ and $Q$ respectively.
For all $t >0$, we then have:
	\begin{equation}
	\mathbb{P} \left\{ U_{n,m}(\ell)  \geq t  \right\} \leq   e^{ -nmt^2/(32c_{\ell}^2)},
	\end{equation}
	where $c_{\ell}=\sup_{(x,y)\in \X \times \Y}\vert \ell(x,y) \vert<+\infty$.
\end{lemma}
\begin{proof}
Let $t>0$. The proof is based on Chernoff's method. For all $\lambda>0$, we have
\begin{multline}\label{eq0}
	\mathbb{P} \left\{ U_{n,m}(\ell)  \geq t  \right\} \leq \exp\left( -\lambda t +\log\left(\mathbb{E}[\exp(\lambda U_{n,m}(\ell) )]  \right) \right)\\
	\leq  \exp\left( -\lambda t +\log\left(\mathbb{E}[\exp(4\lambda T_{n,m}(\ell) )]  \right) \right),
\end{multline}
using \eqref{eq:symexpsupconv2} with $\Phi(t)=\exp(\lambda t)$. Observe next that we almost-surely
\begin{multline*}
\mathbb{E}[\exp(4\lambda T_{n,m}(\ell) )\mid X_1,\; \ldots,\,X_n,\; Y_1,\; \ldots, \; Y_m]= \\ \prod_{i=1}^n\prod_{j=1}^{m} \frac{e^{ 4\lambda \ell(X_i,Y_j)/(nm) } + e^{ -4\lambda \ell(X_i,Y_j)/(nm )}}{2}\\ \leq \prod_{i=1}^n\prod_{j=1}^{m} e^{8 \lambda^2  \ell^2(X_i,Y_j) /(nm)^2}
 \leq  e^{ 8\lambda^2 c_{\ell}^2/(nm)},
\end{multline*}
using the fact that $(e^u+e^{-u})/2\leq e^{u^2 /2}$ for all $u\in \mathbb{R}$.
Integrating the bound over the $X_i$'s and $Y_j$'s and plugging it next into \eqref{eq0} yields the desired bound when choosing $\lambda=nmt/(16c_{\ell}^2)$. $\square$
\end{proof}
\medskip
Finally, we prove the tail probability version of Lemma \ref{lem:random} stated below.

\begin{lemma}\label{lem:random2}Let $P$ and $Q$ be probability distributions on measurable spaces $\X$ and $\Y$ respectively. Consider the degenerate two-sample $U$-statistic of degree $(1,1)$ \eqref{eq:two_sample_stat} with a bounded kernel $\ell:\X \times \Y \rightarrow \mathbb{R}$ based on the independent $\iid$ random samples $X_1,\;  \ldots,\;  X_n $ and  $Y_1,\;  \ldots,\; Y_m$, drawn from $P$ and $Q$ respectively. Let two sequences of $\iid$ Rademacher variables $\eps_1, \ldots, \eps_n$ and $\eta_1, \ldots, \eta_m$, independent of the $Xi$s and $Yj$s, such that the randomized process \eqref{eq:two_sample_stat_rando} is defined.
	Then we have for all $ t>0$,
	\begin{equation}\label{eq:symexpsupconv}
	\mathbb{P}\left\{ \sup_{\ell\in \L}  \vert U_{n,m}(\ell) \vert \geq 16t \right\} \leq 16\mathbb{P}\left\{ \sup_{\ell\in \L}  \vert T_{n,m}(\ell) \vert \geq t \right\},
	\end{equation}
	assuming that the suprema are measurable and that the expectations exist.
\end{lemma}

\begin{proof}
This lemma, bounding the tail probability of $\sup_{\ell\in \L}  \vert U_{n,m}(\ell) \vert $ to that of  $\sup_{\ell\in \L}  \vert T_{n,m}(\ell) \vert $, generalizes Lemma 2.7 in \cite{GZ04} and Lemma 3.1 in \cite{Tal94} to degenerate two-sample $U$-processes. It is proved by applying twice a version of the latter result for independent but non necessarily identically distributed random variables. Indeed, we have: $\forall t>0$,

\begin{multline*}
	\mathbb{P}\left\{ \sup_{\ell\in \L}  \left\vert U_{n,m}(\ell) \right\vert \geq 16t \right\}\\
=	\mathbb{E}\left[ \mathbb{P}\left\{ \sup_{\ell\in \L}  \left\vert \frac{1}{n} \sum_{i=1}^n\left\{\frac{1}{m}\sum_{j=1}^m \ell(X_i,Y_j)\right\}\right\vert \geq 16t \mid Y_1,\; \ldots, \; Y_m \right\}  \right] \\
\leq	4\mathbb{E}\left[ \mathbb{P}\left\{ \sup_{\ell\in \L}  \left\vert \frac{1}{n} \sum_{i=1}^n\left\{\frac{1}{m}\sum_{j=1}^m \epsilon_i\ell(X_i,Y_j)\right\} \right\vert \geq 4t \mid Y_1,\; \ldots, \; Y_m \right\}  \right]\\
=	4\mathbb{E}\left[ \mathbb{P}\left\{ \sup_{\ell\in \L}  \left\vert \frac{1}{m} \sum_{j=1}^m\left\{\frac{1}{n}\sum_{i=1}^n \epsilon_i\ell(X_i,Y_j)\right\}\right\vert \geq 4t \mid (X_1,\epsilon_1)\; \ldots, \; (X_n,\; \epsilon_n ) \right\}  \right]
\\ \leq16 	\mathbb{P}\left\{ \sup_{\ell\in \L}  \left\vert T_{n,m}(\ell) \right\vert \geq t \right\}.
\end{multline*}
 $\square$

\end{proof}
\medskip

The proof relies on the chaining method applied to the process $U_{n,m}(\ell)$ indexed by the class of kernels $\L$, see \textit{e.g.} the argument used to establish Lemma 2.14.9 in \cite{vdVWell96}.
Define the random semi-metric on $\L$ by

\begin{equation}
d_{nm}^2(\ell_1, \ell_2) =\frac{1}{nm}\sum_{i \leq n} \sum_{j \leq m} (\ell_1(X_i,Y_j) - \ell_2(X_i,Y_j) )^2 
\end{equation}
for all kernels $\ell_1$ and $\ell_2$ in $\L$.
For all $q\in \NN^*$, consider a number $k_q \leq (A/\varepsilon_q)^{\mathcal{V}} $ of $L_2$-balls with radius $ \eps_q \leq L \leq 1$ and centers $\ell_{q,k}$, $1\leq k\leq k_q$, \wrt the (random) probability measure $(1/nm)\sum_{i \leq n} \sum_{j \leq m}\delta_{(X_i,Y_j)}$ covering the class $\mathcal{L}$.
Assume that the sequence $\eps_q$ is decreasing as $q$ increases, so that $k_q$ is increasing.
Let $\ell \in \L$, $q\geq 1$ and $\tilde{\ell}_{q}$ be the center of a ball s.t. $ d_{nm}( \ell, \tilde{\ell}_{q})\leq \eps_q$. 
Fixing $q_0\leq q$ in $\NN^*$, the following decomposition holds

\begin{equation*}
U_{n,m}(\ell) = (U_{n,m}(\ell) - U_{n,m}(\tilde{\ell}_{q}) )+ U_{n,m}(\tilde{\ell}_{q_0})  +\sum_{\omega = q_0 +1}^q \left( U_{n,m}(\tilde{\ell}_{\omega})   - U_{n,m}(\tilde{\ell}_{\omega-1})    \right).
\end{equation*}

Observe that, for all $\ell$ in $\mathcal{L}$, we almost-surely have
\begin{equation*}
	\vert U_{n,m}(\ell) - U_{n,m}(\tilde{\ell}_{q})  \vert \leq d_{nm}( \ell, \tilde{\ell}_{q})\leq \eps_q~.
\end{equation*}

The triangular inequality yields 

\begin{multline*}
 \|  U_{n,m}(\ell) \|_{\L} \leq  \eps_q +\max_{1\leq k \leq k_{q_0}} \vert U_{n,m}(\ell_{q_0,k}) \vert 
 +\sum_{\omega = q_0 +1 }^q \vert\vert  U_{n,m}(\tilde{\ell}_{\omega})   - U_{n,m}(\tilde{\ell}_{\omega-1})   \vert\vert_{\mathcal{L}}~,
\end{multline*}
where we used the notation $\vert\vert V\vert\vert_{\mathcal{L}}=\sup_{\ell	\in \mathcal{L}}\vert V(\ell)\vert$ for any real-valued stochastic process $V$ indexed by $\mathcal{L}$.
Considering $\eta_{\omega} >0$ and $\beta>0$ constants such that $\sum_{\omega = q_0+1}^q \eta_{\omega}  + \beta \leq 1$, we have for any $t>\eps_q$:
\begin{multline}\label{eq:triangbound}
	\mathbb{P} \left\{  \|  U_{n,m}(\ell)   \|_{\L}  \geq 16t  \right\} \leq \sum_{k=1}^{k_{q_0}}  \mathbb{P} \left\{ \vert U_{n,m}(\ell_{q_0,k})) \vert   \geq 16t \beta \right\}  \\
	+16\sum_{\omega = q_0 + 1}^q  k_{\omega}^2 \mathbb{E}\left[\sup_{\ell \in \mathcal{L}}  \mathbb{P} \left\{ \vert T_{n,m}(\tilde{\ell}_{\omega}   - \tilde{\ell}_{\omega-1})   \vert  \geq t \eta_{\omega} \mid X_1,\; \ldots,\; X_n,\; Y_1,\; \ldots,\; Y_m \right\}\right],
\end{multline}

using the union bound, Lemma \ref{lem:random2} and observing that the suprema corresponding to the terms of the series are actually maxima taken over at most $k_{\omega}k_{\omega-1}\leq k_{\omega}^2$ elements. Lemma \ref{lemma:hoefS} permits to bound the first term on the right hand side of \eqref{eq:triangbound}:
\begin{equation}\label{eq:basic}
	\sum_{k=1}^{k_{q_0}}  \mathbb{P} \left\{ \vert U_{n,m}(\ell_{q_0,k})) \vert   \geq 16 t \beta \right\} \leq 2 k_{q_0} \exp \left\{ - \frac{8nm(t\beta)^2}{L^2} \right\}~.
\end{equation}

Concerning the second term, notice that
\begin{equation}\label{eq:triang}
d_{nm}(\tilde{\ell}_{\omega}, \tilde{\ell}_{\omega-1})\leq d_{nm}(\ell, \tilde{\ell}_{\omega-1})+d_{nm}(\tilde{\ell}_{\omega}, \ell)\leq 2 \eps_{\omega-1}~.
\end{equation}
Re-using  the start of the argument proving Lemma \ref{lemma:hoefS}, we have: $\forall \lambda>0$,
\begin{multline*}
 \mathbb{P} \left\{  T_{n,m}(\tilde{\ell}_{\omega}-\tilde{\ell}_{\omega-1})   \geq t \eta_{\omega} \mid X_1,\; \ldots,\; X_n,\; Y_1,\; \ldots,\; Y_m \right\}\\
 \leq \exp\left( -\lambda t\eta_{\omega}+\mathbb{E}\left[\exp(\lambda T_{n,m}(\tilde{\ell}_{\omega}-\tilde{\ell}_{\omega-1}) )\mid X_1,\; \ldots, X_n,\; Y_1,\; \ldots,\; Y_m\right]  \right)
\end{multline*}
with probability one. Like in Lemma \ref{lemma:hoefS}'s proof,  we almost-surely have
\begin{multline*}
	\mathbb{E}[\exp(\lambda T_{n,m}(\tilde{\ell}_{\omega}-\tilde{\ell}_{\omega-1}) )\mid X_1,\; \ldots,\,X_n,\; Y_1,\; \ldots, \; Y_m]\leq\\
	\prod_{i=1}^n\prod_{j=1}^{m} e^{ \lambda^2  (\tilde{\ell}_{\omega}-\tilde{\ell}_{\omega-1})^2(X_i,Y_j) /2(nm)^2} \leq  e^{ 2\lambda^2 \epsilon_{\omega-1}^2/(nm)}~.
\end{multline*}
Combining the two bounds above with the union bound, it holds with probability one
\begin{multline}\label{eq:series}
 \mathbb{P} \left\{  \left\vert T_{n,m}(\tilde{\ell}_{\omega}-\tilde{\ell}_{\omega-1}) \right\vert   \geq t \eta_{\omega} \mid X_1,\; \ldots,\; X_n,\; Y_1,\; \ldots,\; Y_m \right\}\leq \\
2 \exp  \left\{ - \frac{nm(t\eta_{\omega})^2}{8\eps_{\omega-1}^2}  \right\}~.
\end{multline}

From \eqref{eq:triangbound}, \eqref{eq:basic} and  \eqref{eq:series}, we deduce that


\begin{multline}\label{eq:triangboun2}
	\mathbb{P} \left\{  \|  U_{n,m}(\ell)   \|_{\L}  \geq 16 t  \right\}\\
	 \leq 2 k_{q_0} \exp \left\{ - \frac{8nm(t\beta)^2}{L^2} \right\} 
	+32 \sum_{\omega = q_0 + 1}^q  k_{\omega} ^2 \exp \left\{ - \frac{nm(t\eta_{\omega})^2}{8\eps_{\omega-1}^2}  \right\} \\
	\leq 2 A^{\V}\eps_{q_0}^{-\V} \exp \left\{ - \frac{8nm(t\beta)^2}{L^2} \right\}
	+32A^{2\V}  \sum_{\omega = q_0 + 1}^q  \eps_{\omega}^{-2\V} \exp \left\{ - \frac{nm(t\eta_{\omega})^2}{8\eps_{\omega-1}^2}  \right\}~.
\end{multline}

Following Lemma 3.2 in \cite{vdG00} and choosing $\eps_{\omega} = 2^{-\omega} L$, $\eta_{\omega} = 2^{-\omega}\sqrt{\omega}/8$, so that $\eta_{\omega+1} / \eps_{\omega} = (1/16L )\sqrt{\omega+1}$, we have

\begin{multline}
\eps_{\omega}^{-2\V} \exp \left\{ - \frac{nm(t\eta_{\omega})^2}{8\eps_{\omega-1}^2}  \right\} 
= L^{-2\V}  \exp \left\{  - (- 2 \V \log(2) + \frac{nmt^2}{4\times8^3L^2} )\omega  \right\} 
\end{multline}

If $nmt^2 >   8^4 \log(2) L^2 \V$, the terms of the series are decreasing \wrt $\omega$ and we upperbound by $K_1L^{-2\V}  \exp \left\{ - nmt^2\omega/(4 \times 8^3L^2)  \right\}$. Problem $2.14.3$ in \cite{vdVWell96} applies for $\omega \in \{q_0 +1 , \ldots, q\}$ with $\psi(\omega) = nmt^2\omega/(4 \times 8^3L^2)$

\begin{multline}
\sum_{\omega = q_0 + 1}^q  \eps_{\omega}^{-2\V} \exp \left\{ - \frac{nm(t\eta_{\omega})^2}{8\eps_{\omega-1}^2}  \right\} \leq K_1 L^{-2\V}  \psi'(q_0 )^{-1}   \exp \left\{ - \psi(q_0) \right\} \\
\leq K_2 L^{-2(\V-1)} \exp \left\{  - \frac{nmt^2}{4 \times 8^3L^2} q_0 \right\} 
\end{multline}
$K_1, \; K_2>0$ constants and $nmt^2\geq1$. 
For $\alpha>0$ large, setting $q_0 = 2 +  \lfloor (nmt^2)^{1/(\alpha-1)}\rfloor$ yields to the upperbound $K_2 L^{-2(\V-1)} \exp \left\{ -  3nmt^2/(4 \times 8^3L^2) \right\} $. For the first tail probability, by setting $\beta = 1/2 - 1/(2nmt^2)$ we obtain an upperbound of similar form 

\begin{multline*}
 A^{\V}\eps_{q_0}^{-\V} \exp \left\{ - \frac{8nm(t\beta)^2}{L^2} \right\} \\
  \leq (A/L)^{\V} \exp \left\{ \V \log(2) (2 + (nmt^2)^{1/(\alpha-1)} )  - \frac{2nmt^2}{L^2} (1- 1/(nmt^2))^2\right\}\\
 \leq (2A/L)^{\V} e^{4/L^2 }  \exp \left\{ \V \log(2) (nmt^2)^{1/(\alpha-1)} - \frac{2nmt^2}{L^2} \right\}
 \\
 \leq  (2A/L)^{\V}e^{4/L^2 }  \exp \left\{ -  \frac{2nmt^2}{L^2} \right\}~,
\end{multline*}

as soon as $nmt^2 > ( \log(2) L^2 \V/2)^{1+\delta}, \; \delta =1 / (\alpha-2) \in (0,1)$ for large $\alpha$. Gathering both upperbounds, Eq. \eqref{eq:triangboun2} yields

\begin{equation}
\mathbb{P} \left\{  \|  U_{n,m}(\ell)   \|_{\L}  \geq t  \right\} \leq K 2^{\V+1}(A/L)^{2\V}e^{4/L^2}\exp \left\{ -  \frac{3nmt^2}{4\times8^3L^2} \right\}~,
\end{equation}

for all $nmt^2 >  \max( 1, 8^4\log(2) L^2 \V ,( \log(2) L^2 \V/2)^{1+\delta}) $, and $K\geq 1+ 16K_2e^{-4}$ constant. Checking lastly that, for all $q\geq 1$
\begin{equation}
8\sum_{\omega = q_0 +1}^q \eta_{\omega}  \leq 8\sum_{\omega = 1}^q \eta_{\omega}  \leq 1 + \int_1^{\infty}2^{-x}\sqrt{x}dx \leq 1 + (\pi/\log(2))^{1/2} \leq 4,
\end{equation}
so that $\sum_{\omega = q_0+1}^q \eta_{\omega}  + \beta \leq 1$ as needed.

\section{Additional Numerical Experiments}\label{ann:expes}

Following Section \ref{sec:num}, this section gathers the numerical results of three models Loc1, Loc3 and Scale2, Scale3, as well as additional experiments regarding 
the difference in performance of the $W$-criteria for the RTB score-generating function, when we vary the rate $u_0$, for both the location (Fig. \ref{fig:roclocrtball}) and the scale (Fig. \ref{fig:rocscalertball}) models.

\paragraph{Location model.} (Fig. \ref{fig:rocloc1}, \ref{fig:rocloc3})

\begin{figure}[!h]
	\centering
	\begin{tabular}{cc}
		\parbox{4cm}{	
			\includegraphics[width=4cm, height=4cm]{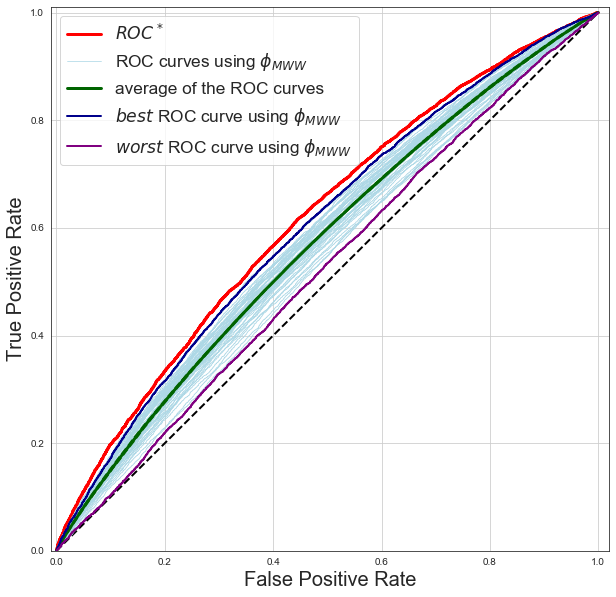}
			\includegraphics[width=4cm, height=4cm]{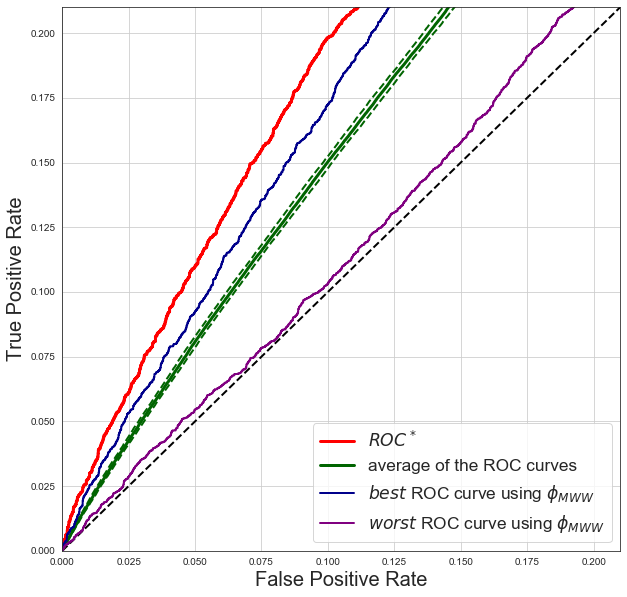}
			{\scriptsize 1.  $\phi_{MWW}(u) = u$ }\\
		}
		\parbox{4cm}{
			\includegraphics[width=4cm, height=4cm]{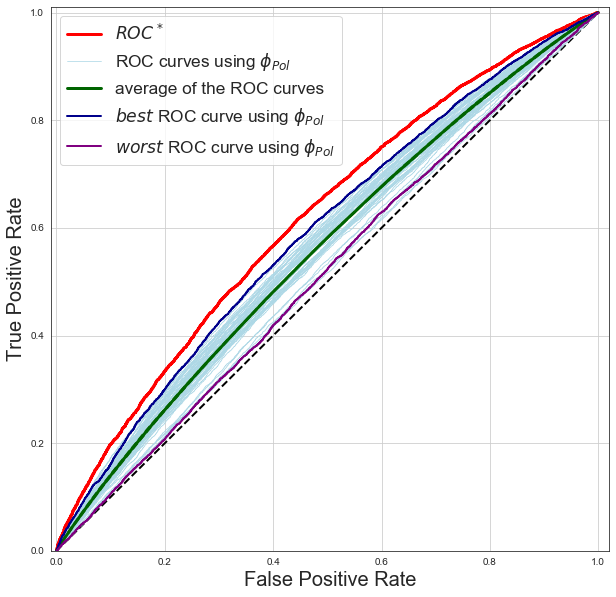}
			\includegraphics[width=4cm, height=4cm]{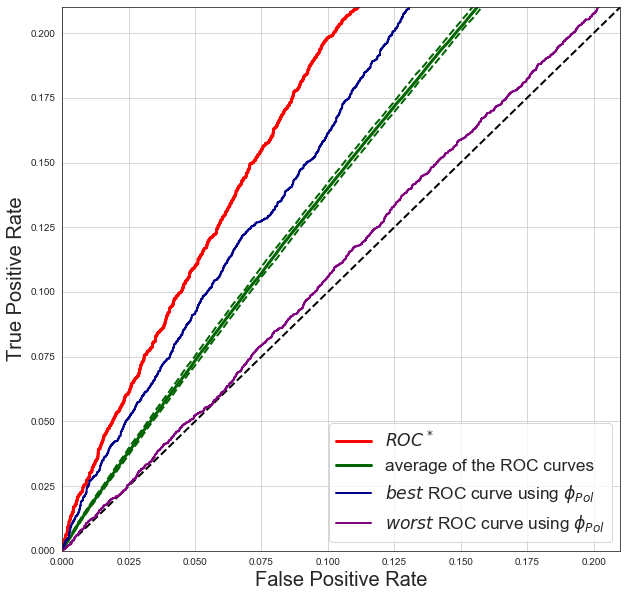}
			{\scriptsize 2.  $\phi_{Pol}(u) = u^3$}\\
		}
		\parbox{4cm}{	
			\includegraphics[width=4cm, height=4cm]{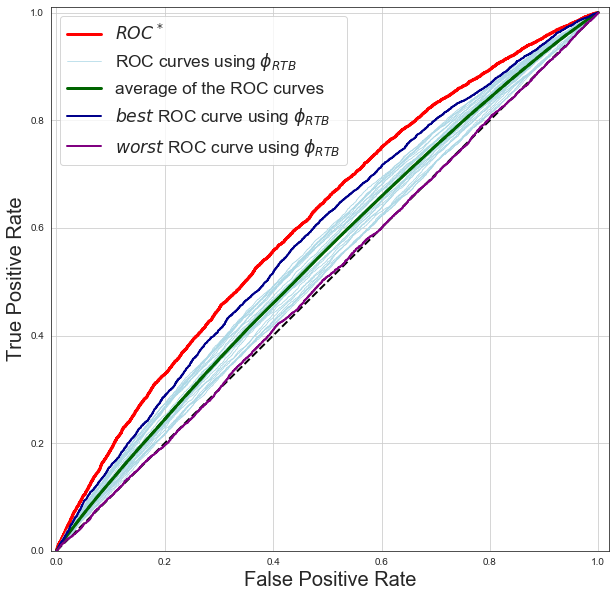}
			\includegraphics[width=4cm, height=4cm]{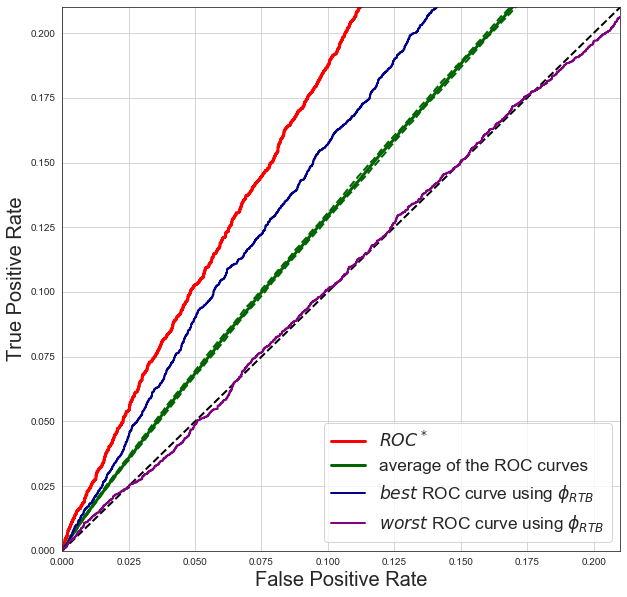}
			{\scriptsize 3. $\phi_{RTB}(u) =  u\mathbb{I}\{u \geq 0.9\}$}\\
		}
		\medskip
	\end{tabular}
	\caption{Empirical $\roc$ curves and average $\roc$ curve for Loc1  ($\varepsilon = 0.10$). Samples are drawn from multivariate Gaussian distributions according to section \ref{sec:synthdata},
		scored with early-stopped GA algorithm's optimal parameter for the class of scoring functions. Hyperparameters: $u_0 = 0.9$, $q = 3$, $B = 50$, $T = 50$. Parameters for the training set: $n=m=150$; $d=15$; for the testing set:  $n=m=10^6$; $d=15$.
		Figures $1, 2, 3$ correspond \resp to the models MMW, Pol, RTB. Light blue curves are the $B(=50)$ $\roc$ curves that are averaged in green (solid line) with $+/-$ its standard deviation (dashed green lines). The dark blue and purple curves correspond to the best and worst scoring functions in the sense of minimization and maximization of the generalization error among the $B$ curves. The red curve corresponds to $\roc^*$.}
	\label{fig:rocloc1}
\end{figure}

\begin{figure}[!h]
	\centering
	\begin{tabular}{cc}
		\parbox{4cm}{	
			\includegraphics[width=4cm, height=4cm]{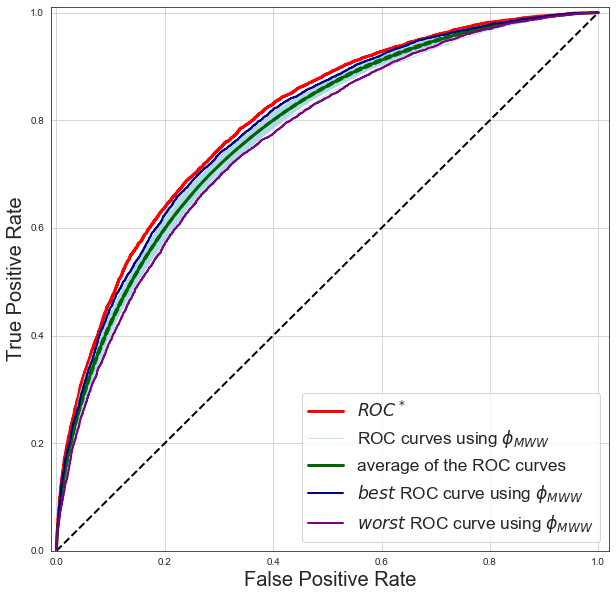}
			\includegraphics[width=4cm, height=4cm]{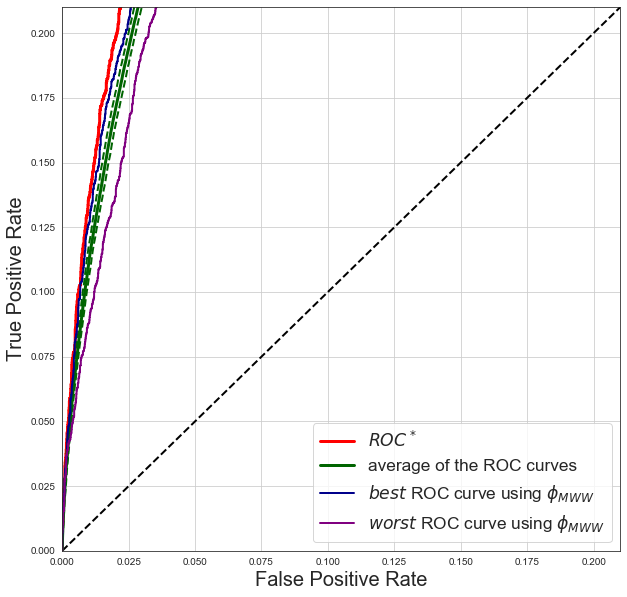}
			{\scriptsize 1.  $\phi_{MWW}(u) = u$ }\\
		}
		\parbox{4cm}{
			\includegraphics[width=4cm, height=4cm]{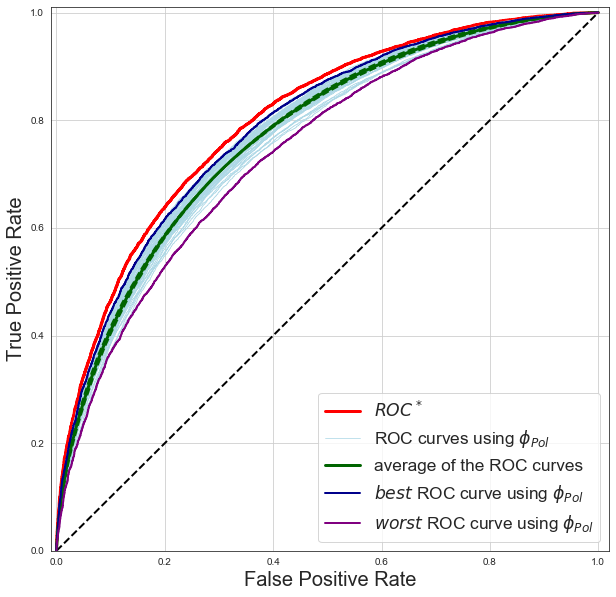}
			\includegraphics[width=4cm, height=4cm]{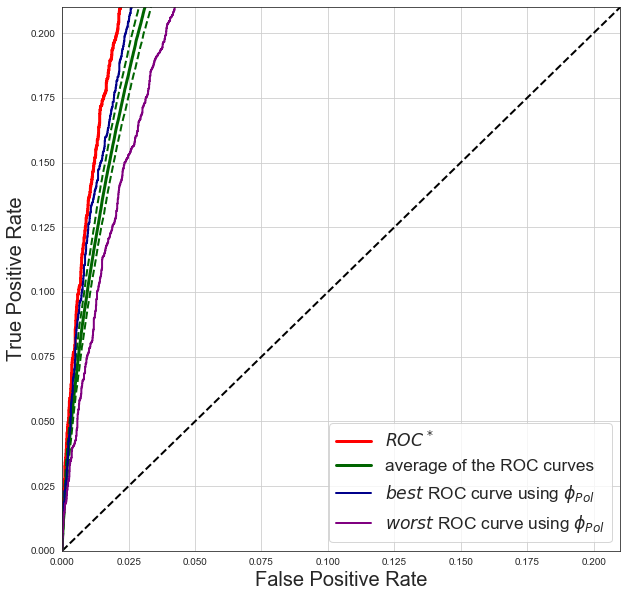}
			{\scriptsize 2.  $\phi_{Pol}(u) = u^3$}\\
		}
		\parbox{4cm}{	
			\includegraphics[width=4cm, height=4cm]{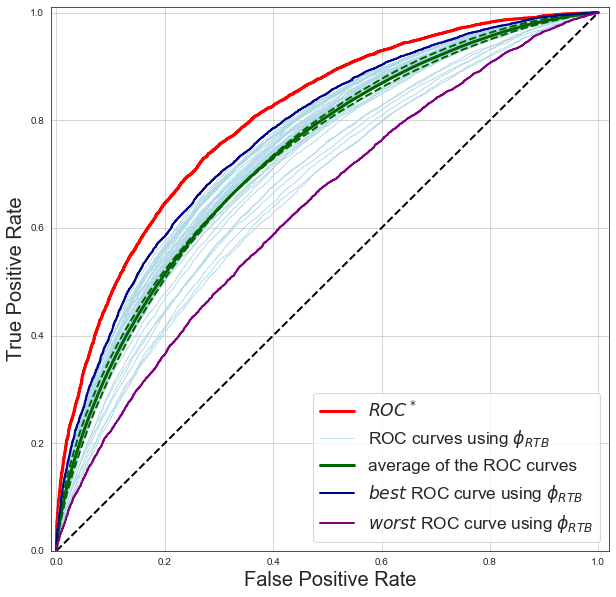}
			\includegraphics[width=4cm, height=4cm]{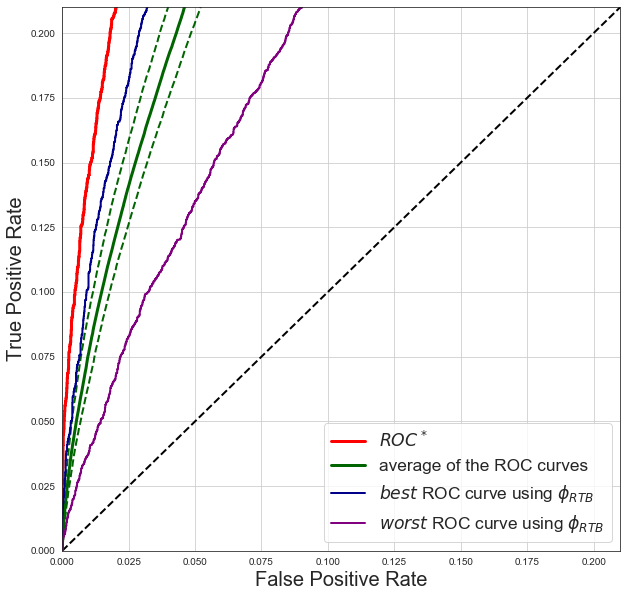}
			{\scriptsize 3. $\phi_{RTB}(u) =  u\mathbb{I}\{u \geq 0.9\}$}\\
		}
		\medskip
	\end{tabular}
	\caption{Empirical $\roc$ curves and average $\roc$ curve for Loc3 ($\varepsilon = 0.30$). Samples are drawn from multivariate Gaussian distributions according to section \ref{sec:synthdata},
		scored with early-stopped GA algorithm's optimal parameter for the class of scoring functions. Hyperparameters: $u_0 = 0.9$, $q = 3$, $B = 50$, $T = 50$. Parameters for the training set: $n=m=150$; $d=15$; for the testing set:  $n=m=10^6$; $d=15$.
		Figures $1, 2, 3$ correspond \resp to the models MMW, Pol, RTB. Light blue curves are the $B(=50)$ $\roc$ curves that are averaged in green (solid line) with $+/-$ its standard deviation (dashed green lines). The dark blue and purple curves correspond to the best and worst scoring functions in the sense of minimization and maximization of the generalization error among the $B$ curves. The red curve corresponds to $\roc^*$.}
	\label{fig:rocloc3}
\end{figure}

\paragraph{Scale model.} (Fig. \ref{fig:rocscale2}, \ref{fig:rocscale3})

\begin{figure}[!h]
	\centering
	\begin{tabular}{cc}
		\parbox{4cm}{	
			\includegraphics[width=4cm, height=4cm]{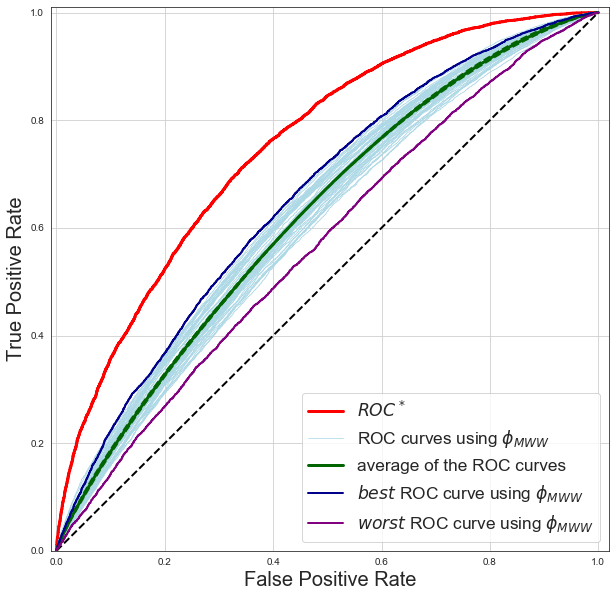}
			\includegraphics[width=4cm, height=4cm]{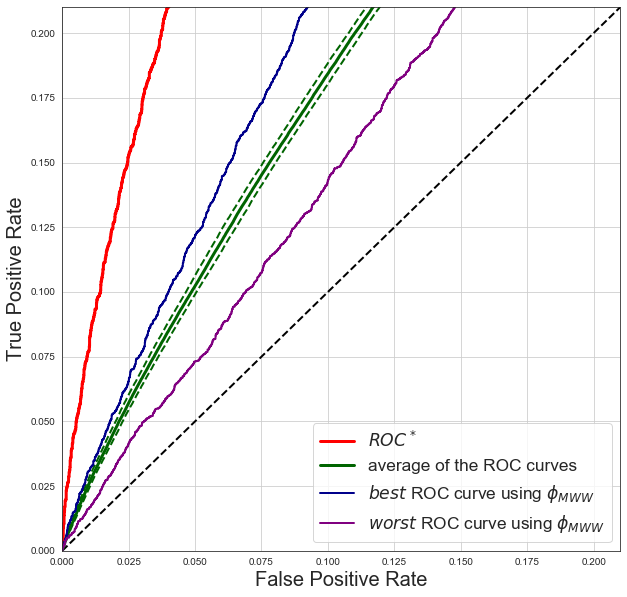}
			{\scriptsize 1.  $\phi_{MWW}(u) = u$ }\\
		}
		\parbox{4cm}{
			\includegraphics[width=4cm, height=4cm]{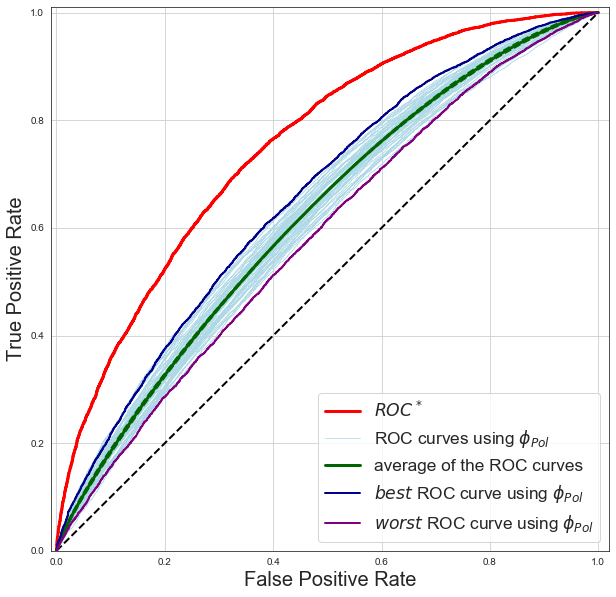}
			\includegraphics[width=4cm, height=4cm]{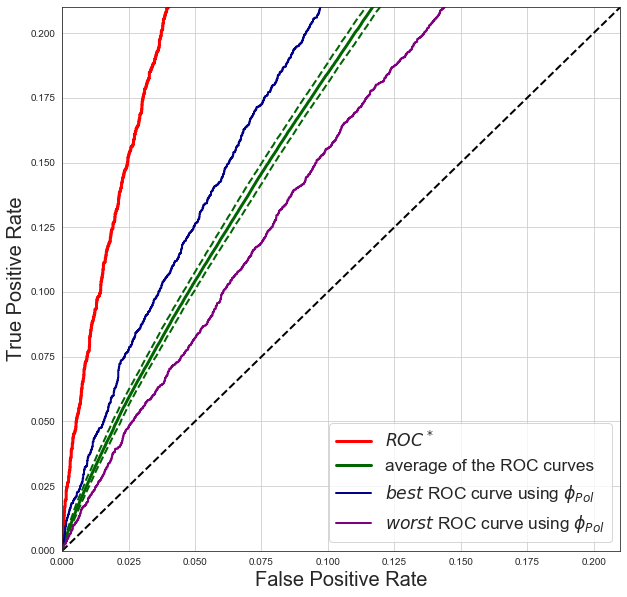}
			{\scriptsize 2.  $\phi_{Pol}(u) = u^3$}\\
		}
		\parbox{4cm}{	
			\includegraphics[width=4cm, height=4cm]{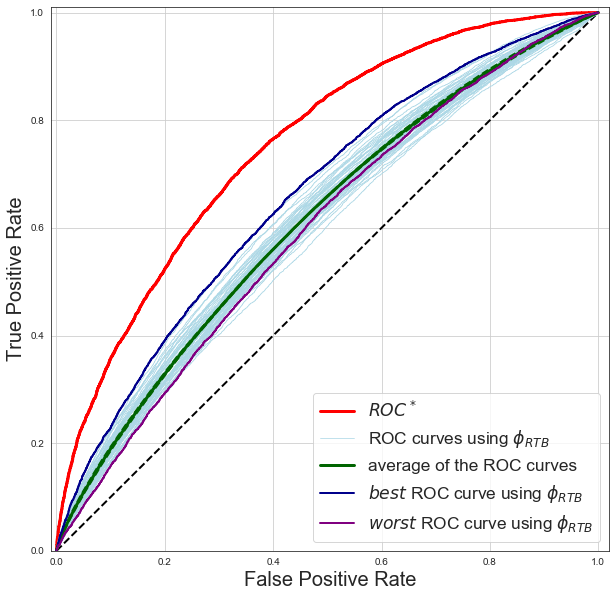}
			\includegraphics[width=4cm, height=4cm]{scalemww90_zoom}
			{\scriptsize 3. $\phi_{RTB}(u) =  u\mathbb{I}\{u \geq 0.9\}$}\\
		}
		\medskip
	\end{tabular}
	\caption{Empirical $\roc$ curves and average $\roc$ curve for Scale2 ($\varepsilon = 0.90$). Samples are drawn from multivariate Gaussian distributions according to section \ref{sec:synthdata},
		scored with early-stopped GA algorithm's optimal parameter for the class of scoring functions. Hyperparameters: $u_0 = 0.9$, $q = 3$, $B = 50$, $T = 50$. Parameters for the training set: $n=m=150$; $d=15$; for the testing set:  $n=m=10^6$; $d=15$.
		Figures $1, 2, 3$ correspond \resp to the models MMW, Pol, RTB. Light blue curves are the $B(=50)$ $\roc$ curves that are averaged in green (solid line) with $+/-$ its standard deviation (dashed green lines). The dark blue and purple curves correspond to the best and worst scoring functions in the sense of minimization and maximization of the generalization error among the $B$ curves. The red curve corresponds to $\roc^*$.}
	\label{fig:rocscale2}
\end{figure}

\begin{figure}[!h]
	\centering
	\begin{tabular}{cc}
		\parbox{4cm}{	
			\includegraphics[width=4cm, height=4cm]{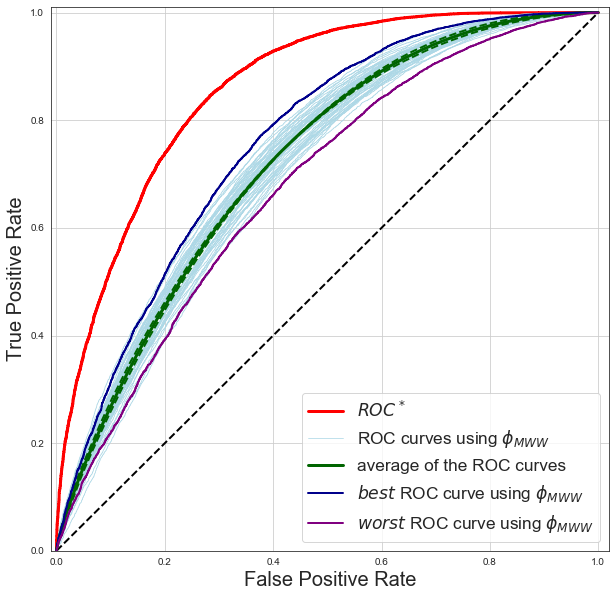}
			\includegraphics[width=4cm, height=4cm]{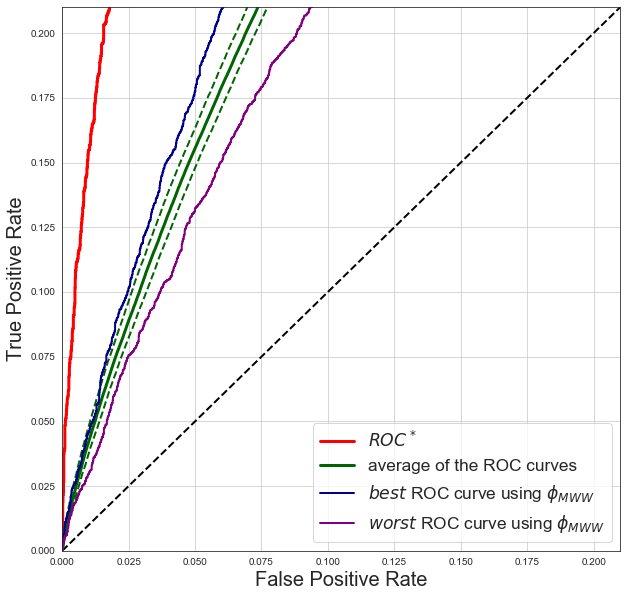}
			{\scriptsize 1.  $\phi_{MWW}(u) = u$ }\\
		}
		\parbox{4cm}{
			\includegraphics[width=4cm, height=4cm]{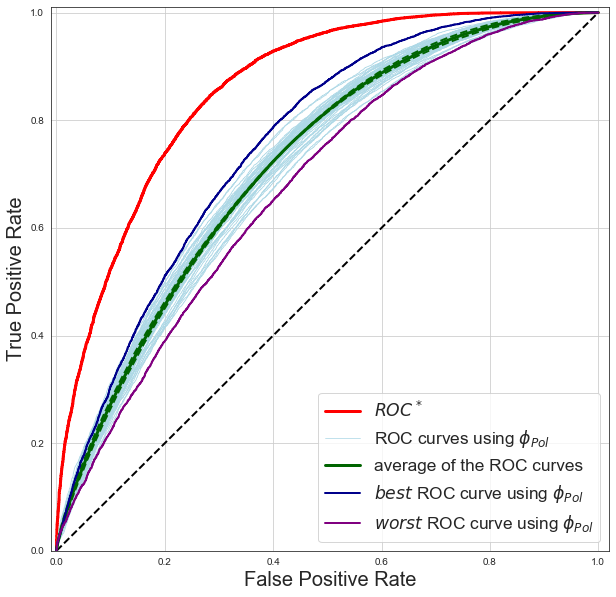}
			\includegraphics[width=4cm, height=4cm]{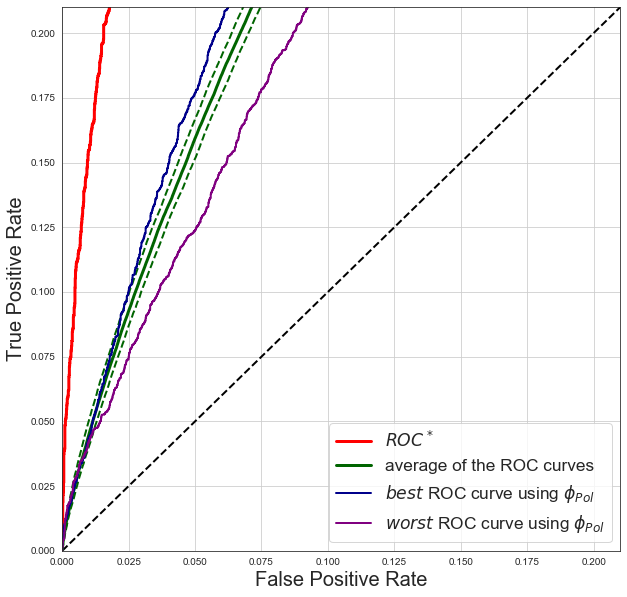}
			{\scriptsize 2.  $\phi_{Pol}(u) = u^3$}\\
		}
		\parbox{4cm}{	
			\includegraphics[width=4cm, height=4cm]{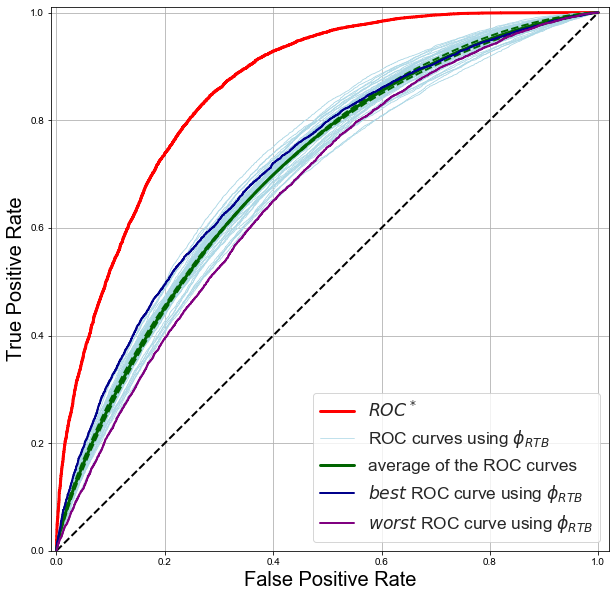}
			\includegraphics[width=4cm, height=4cm]{scalemww110_zoom}
			{\scriptsize 3. $\phi_{RTB}(u) =  u\mathbb{I}\{u \geq 0.9\}$}\\
		}
		\medskip
	\end{tabular}
	\caption{Empirical $\roc$ curves and average $\roc$ curve for Scale3 ($\varepsilon = 1.10$). Samples are drawn from multivariate Gaussian distributions according to section \ref{sec:synthdata},
		scored with early-stopped GA algorithm's optimal parameter for the class of scoring functions. Hyperparameters: $u_0 = 0.9$, $q = 3$, $B = 50$, $T = 50$. Parameters for the training set: $n=m=150$; $d=15$; for the testing set:  $n=m=10^6$; $d=15$.
		Figures $1, 2, 3$ correspond \resp to the models MMW, Pol, RTB. Light blue curves are the $B(=50)$ $\roc$ curves that are averaged in green (solid line) with $+/-$ its standard deviation (dashed green lines). The dark blue and purple curves correspond to the best and worst scoring functions in the sense of minimization and maximization of the generalization error among the $B$ curves. The red curve corresponds to $\roc^*$.}
	\label{fig:rocscale3}
\end{figure}

\paragraph{Comparison of three RTB score-generating functions for two location models.}(Fig. \ref{fig:roclocrtball})

\begin{figure}[!h]
	\centering
	\begin{tabular}{cc}
		\parbox{5.5cm}{	
	\includegraphics[width=5.5cm, height=5cm]{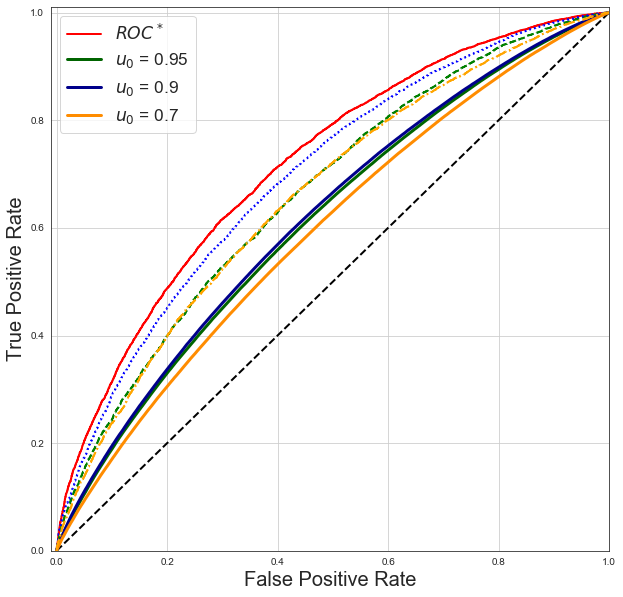}		
	\includegraphics[width=5.5cm, height=5cm]{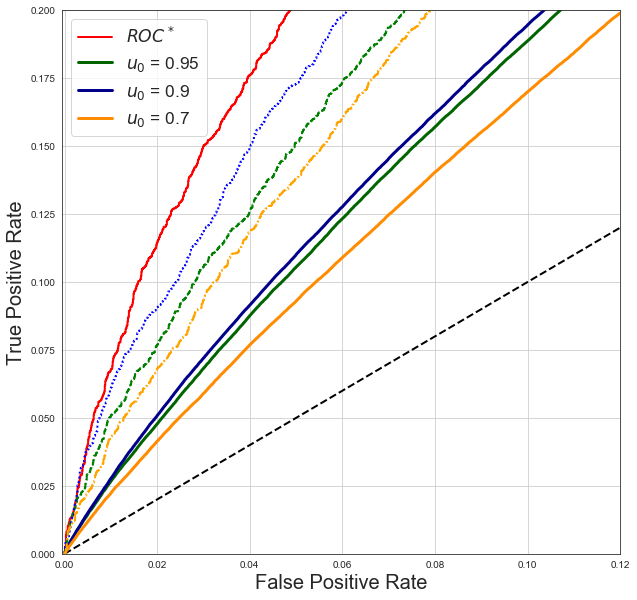}		
	{\scriptsize 1. Loc2, $\varepsilon = 0.2$}\\
}
\parbox{5.5cm}{
	\includegraphics[width=5.5cm, height=5cm]{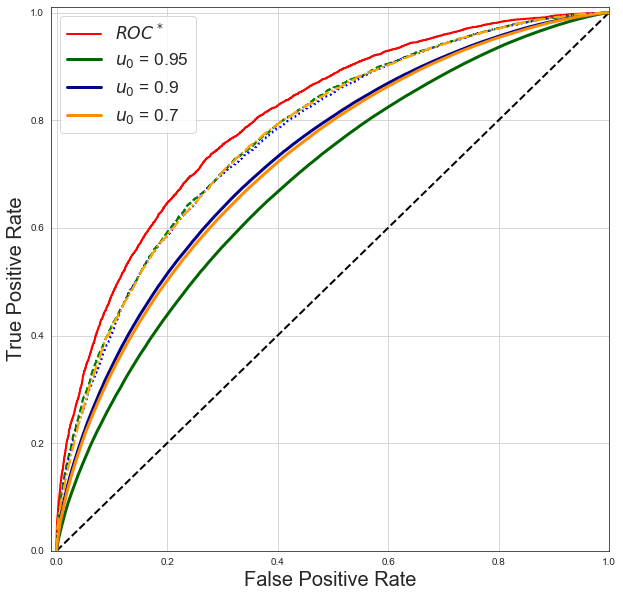}		
	\includegraphics[width=5.5cm, height=5cm]{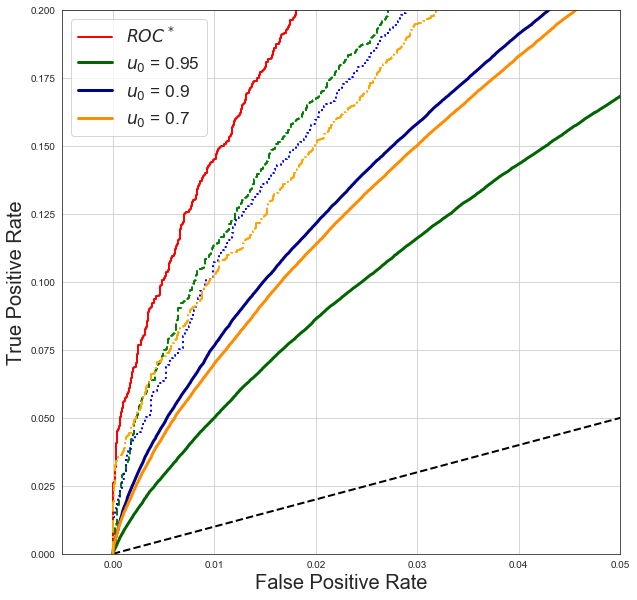}		
	{\scriptsize 2. Loc3, $\varepsilon = 0.3$}\\
}
		\medskip
	\end{tabular}
\caption{Comparison of three RTB models. Average of the $\roc$ curves (solid line), \textit{best} $\roc$ curves (dashed line) for the two location models Loc2 and Loc3. In green for $u_0 = 0.95$, blue for $u_0 = 0.90$, orange for $u_0 = 0.70$,  red for $\roc^*$. Samples are drawn from multivariate Gaussian distributions according to section \ref{sec:synthdata}, scored with early-stopped GA algorithm's optimal parameter for the class of scoring functions and averaged after $B$ loops. Hyperparameters: $B = 50$, $T = 50$. Parameters for the training set: $n=m=150$; $d=15$; for the testing set:  $n=m=10^6$; $d=15$.}
	\label{fig:roclocrtball}
\end{figure}

\paragraph{Comparison of three RTB score-generating functions for the scale model.}(Fig. \ref{fig:rocscalertball})
\begin{figure}[!h]
	\centering
	\begin{tabular}{cc}
		\parbox{5.5cm}{	
			\includegraphics[width=5.5cm, height=5cm]{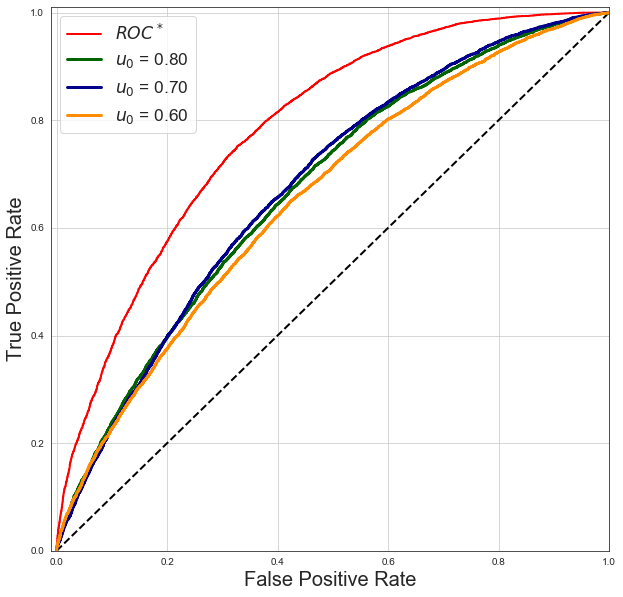}		
		}
		\parbox{5.5cm}{
			\includegraphics[width=5.5cm, height=5cm]{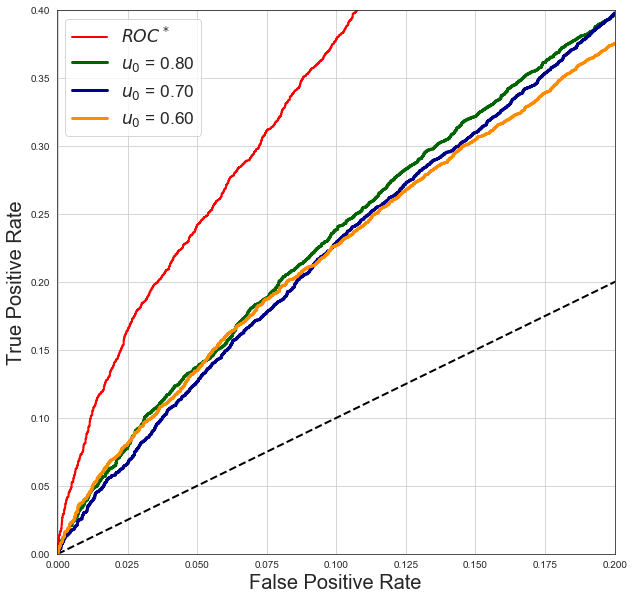}		
		}
		\medskip
	\end{tabular}
	\caption{Comparison of three RTB models. \textit{Best} $\roc$ curves for the Scale2 model. In green for $u_0 = 0.80$, orange for $u_0 = 0.70$, green for $u_0 = 0.60$, red for $\roc^*$. Samples are drawn from multivariate Gaussian distributions according to section \ref{sec:synthdata}, scored with early-stopped GA algorithm's optimal parameter for the class of scoring functions and averaged after $B$ loops. Hyperparameters: $B = 50$, $T = 70$. Parameters for the training set: $n=m=150$; $d=15$; for the testing set:  $n=m=10^6$; $d=15$.}
	\label{fig:rocscalertball}
\end{figure}



\vskip 0.2in
\bibliography{References_ranking1,References_ranking2}

\begin{thebibliography}{10}

\bibitem{AGHHPR05}
S.~Agarwal, T.~Graepel, R.~Herbrich, S.~Har-Peled, and D.~Roth.
\newblock Generalization bounds for the area under the {ROC} curve.
\newblock {\em Journal of Machine Learning Research}, 6:393--425, 2005.

\bibitem{BBL05}
S.~Boucheron, O.~Bousquet, and G.~Lugosi.
\newblock {T}heory of {C}lassification: {A} {S}urvey of {S}ome {R}ecent
  {A}dvances.
\newblock {\em ESAIM: Probability and Statistics}, 9:323--375, 2005.

\bibitem{CLEM14}
S.~Cl\'emen\c{c}on.
\newblock A statistical view of clustering performance through the theory of
  {U}-processes.
\newblock {\em Journal of Multivariate Analysis}, 124:42--56, 2014.

\bibitem{CLV08}
S.~Cl\'emen\c{c}on, G.~Lugosi, and N.~Vayatis.
\newblock Ranking and empirical risk minimization of {U}-statistics.
\newblock {\em The Annals of Statistics}, 36(2):844--874, 2008.

\bibitem{CR15}
S.~Cl\'emen\c{c}on and S.~Robbiano.
\newblock {The TreeRank Tournament Algorithm for Multipartite Ranking}.
\newblock {\em Journal of Nonparametric Statistics}, 27(1):107--126, 2015.

\bibitem{CRV13}
S.~Cl\'emen\c{c}on, S.~Robbiano, and N.~Vayatis.
\newblock {Ranking Data with Ordinal Labels: Optimality and Pairwise
  Aggregation}.
\newblock {\em Machine Learning}, 93(1):67--104, 2013.

\bibitem{CV07}
S.~Cl\'emen\c{c}on and N.~Vayatis.
\newblock Ranking the best instances.
\newblock {\em Journal of Machine Learning Research}, 8:2671--2699, 2007.

\bibitem{CV08NIPS1}
S.~Cl\'emen\c{c}on and N.~Vayatis.
\newblock Empirical performance maximization based on linear rank statistics.
\newblock {\em Advances in Neural Information Processing Systems}, 3559:1--15,
  2009.

\bibitem{CV09ieee}
S.~Cl\'emen\c{c}on and N.~Vayatis.
\newblock Tree-based ranking methods.
\newblock {\em IEEE Transactions on Information Theory}, 55(9):4316--4336,
  2009.

\bibitem{CV10CA}
S.~Cl\'emen\c{c}on and N.~Vayatis.
\newblock Overlaying classifiers: a practical approach to optimal scoring.
\newblock {\em Constructive Approximation}, 32(3):619--648, 2010.

\bibitem{CosZha06b}
D.~Cossock and T.~Zhang.
\newblock Subset ranking using regression.
\newblock {\em Proceedings of COLT 2006}, 4005:605--619, 2006.

\bibitem{DGL96}
L.~Devroye, L.~Gy\"orfi, and G.~Lugosi.
\newblock {\em A Probabilistic Theory of Pattern Recognition}.
\newblock Springer, 1996.

\bibitem{GG02}
E.~Giné and A.~Guillou.
\newblock Rates of strong uniform consistency for multivariate kernel density
  estimators.
\newblock {\em Annales de l'Institut Henri Poincare (B) Probability and
  Statistics}, 38(6):907 -- 921, 2002.

\bibitem{GKZ04}
E.~Giné, V.~Koltchinskii, and J.~Zinn.
\newblock Weighted uniform consistency of kernel density estimators.
\newblock {\em The Annals of Probability}, 32(3B):2570--2605, 2004.

\bibitem{GZ04}
E.~Giné and J.~Zinn.
\newblock Some limit theorems for empirical processes.
\newblock {\em The Annals of Probability}, 12(4):929--989, 2004.

\bibitem{GKKW02}
L.~Gy\"{o}rfi, M.~K\"{o}hler, A.~Krzyzak, and H.~Walk.
\newblock {\em A Distribution-Free Theory of Nonparametric Regression}.
\newblock Springer, 2002.

\bibitem{Haj62}
J.~H\'ajek.
\newblock Asymptotically most powerful rank-order tests.
\newblock {\em The Annals of Mathematical Statistics}, 33(3):112--1147, 1962.

\bibitem{Haj68}
J.~H\'ajek.
\newblock Asymptotic normality of simple linear rank statistics under
  alternatives.
\newblock {\em The Annals of Mathematical Statistics}, 39:325--346, 1968.

\bibitem{HajSid67}
J.~H\'ajek and Z.~Sid\'ak.
\newblock {\em Theory of Rank Tests}.
\newblock Academic Press, 1967.

\bibitem{Hoeffding48}
W.~Hoeffding.
\newblock A class of statistics with asymptotically normal distribution.
\newblock {\em The Annals of Mathematical Statistics}, 19:293--325, 1948.

\bibitem{Jones90}
M.C. Jones.
\newblock The performance of kernel density functions in kernel distribution
  function estimation.
\newblock {\em Statistics and Probability Letters}, 9(2):129--132, 1990.

\bibitem{Kolt06}
V.~Koltchinskii.
\newblock {Local Rademacher complexities and oracle inequalities in risk
  minimization}.
\newblock {\em The Annals of Statistics}, 34(6):2593 -- 2656, 2006.

\bibitem{Lee90}
A.~J. Lee.
\newblock {\em {${U}$-statistics: Theory and practice}}.
\newblock Marcel Dekker, Inc., New York, 1990.

\bibitem{Major2006}
P.~Major.
\newblock An estimate on the supremum of a nice class of stochastic integrals
  and {U}-statistics.
\newblock {\em Probability Theory and Related Fields}, 134(3):489--537, 2006.

\bibitem{McD89}
C.~McDiarmid.
\newblock {\em On the method of bounded differences}, page 148–188.
\newblock London Mathematical Society Lecture Note Series. Cambridge University
  Press, 1989.

\bibitem{MW16}
A.K. Menon and R.C. Williamsson.
\newblock Bipartite ranking: A risk theoretic perspective.
\newblock {\em Journal of Machine Learning Research}, 7:1--102, 2016.

\bibitem{Nadaraya64}
E.A. Nadaraya.
\newblock Somenew estimates for distribution functions.
\newblock {\em Theory of Probability and its Applications}, 9(3):497--500,
  1964.

\bibitem{Neum2004}
N.~Neumeyer.
\newblock A central limit theorem for two-sample u-processes.
\newblock {\em Statistics and Probability Letters}, 67(1):73 -- 85, 2004.

\bibitem{NoPo87}
D.~Nolan and D.~Pollard.
\newblock $u$-processes: Rates of convergence.
\newblock {\em Annals of Statistics}, 15(2):780--799, 1987.

\bibitem{Rud06}
C.~Rudin.
\newblock Ranking with a {P}-{N}orm {P}ush.
\newblock {\em Proceedings of COLT 2006}, 4005:589--604, 2006.

\bibitem{Ser80}
R.J. Serfling.
\newblock {\em Approximation theorems of mathematical statistics}.
\newblock John Wiley \& Sons, 1980.

\bibitem{Tal94}
M.~Talagrand.
\newblock {Sharper Bounds for Gaussian and Empirical Processes}.
\newblock {\em The Annals of Probability}, 22(1):28 -- 76, 1994.

\bibitem{PeGin99}
E.~Giné V.~De~la Pena.
\newblock {\em Decoupling: from dependence to independence}.
\newblock Springer Science and Business Media, 1999.

\bibitem{vdG00}
S.~van~de Geer.
\newblock {\em Empirical Processes in {M}-Estimation}.
\newblock Cambridge University Press, 2000.

\bibitem{vdV98}
A.~van~der Vaart.
\newblock {\em Asymptotic Statistics}.
\newblock Cambridge University Press, 1998.

\bibitem{vdVWell96}
A.~van~der Vaart and J.~Wellner.
\newblock {\em Weak Convergence and Empirical Processes}.
\newblock Springer-Verlag New York, 1996.

\bibitem{Wil45}
F.~Wilcoxon.
\newblock Individual comparisons by ranking methods.
\newblock {\em Biometrics}, 1:80--83, 1945.

\end{thebibliography}
\bibliographystyle{plain}

\end{document}